\documentclass[12pt,reqno]{amsart}
\usepackage{}
\usepackage{amsfonts}
\usepackage{a4wide}
\allowdisplaybreaks
\numberwithin{equation}{section}

\newtheorem{theorem}{Theorem}[section]
\newtheorem{proposition}[theorem]{Proposition}

\newtheorem{lemma}[theorem]{Lemma}

\theoremstyle{definition}

\newtheorem{remark}[theorem]{Remark}

\begin{document}
\title
 [Schr\"{o}dinger-Poisson equations with critical Sobolev exponents]
 {Standing waves for a class of Schr\"{o}dinger-Poisson equations in ${\mathbb{R}^3}$ involving critical Sobolev exponents*}\footnotetext{*This work was supported by Natural Science Foundation of China (Grant No. 11371159), Hubei Key Laboratory of Mathematical Sciences and Program for Changjiang Scholars and Innovative Research Team in University $\# $ IRT13066.}

\maketitle
 \begin{center}
\author{Yi He}
and
\author{Gongbao Li}\footnote{Corresponding Author: Gongbao Li. Email addresses: heyi19870113@163.com (Y. He), ligb@mail.ccnu.edu.cn(G. Li).}
\end{center}

\begin{center}
\address{Hubei Key Laboratory of Mathematical Sciences and School of Mathematics and Statistics, Central China
Normal University, Wuhan, 430079, P. R. China }
\end{center}

\maketitle

\begin{abstract}
We are concerned with the following Schr\"{o}dinger-Poisson equation with critical nonlinearity:
\[
\left\{ \begin{gathered}
   - {\varepsilon ^2}\Delta u + V(x)u + \psi u = \lambda |u{|^{p - 2}}u + |u{|^4}u{\text{ in }}{\mathbb{R}^3}, \hfill \\
   - {\varepsilon ^2}\Delta \psi  = {u^2}{\text{ in }}{\mathbb{R}^3},{\text{ }}u > 0,{\text{ }}u \in {H^1}({\mathbb{R}^3}), \hfill \\
\end{gathered}  \right.
\]
where $\varepsilon > 0$ is a small positive parameter, $\lambda  > 0$, $3 < p \le 4$. Under certain assumptions on the potential $V$, we construct a family of positive solutions ${u_\varepsilon } \in {H^1}({\mathbb{R}^3})$ which concentrates around a local minimum of $V$ as $\varepsilon  \to 0$.

Although, subcritical growth Schr\"{o}dinger-Poisson equation
\[
\left\{ \begin{gathered}
   - {\varepsilon ^2}\Delta u + V(x)u + \psi u = f(u){\text{ in }}{\mathbb{R}^3}, \hfill \\
   - {\varepsilon ^2}\Delta \psi  = {u^2}{\text{ in }}{\mathbb{R}^3},{\text{ }}u > 0,{\text{ }}u \in {H^1}({\mathbb{R}^3}) \hfill \\
\end{gathered}  \right.
\]
has been studied extensively, where the assumption for $f(u)$ is that $f(u) \sim |u{|^{p - 2}}u$ with $4 < p < 6$ and satisfies the Ambrosetti-Rabinowitz condition which forces the boundedness of any Palais-Smale sequence of the corresponding energy functional of the equation. The more difficult critical case is studied in this paper. As $g(u): = \lambda {| u |^{p - 2}}u + {| u |^4u}$ with $3<p\leq 4$ does not satisfy the  Ambrosetti-Rabinowitz condition ($\exists \mu  > 4, 0 < \mu \int_0^u {g(s)ds \le g(u)u}$), the boundedness of Palais-smale sequence becomes a major difficulty in proving the existence of a positive solution. Also, the fact that the function $\frac{{g(s)}}
{{{s^3}}}$ is not increasing for $s > 0$ prevents us from using the Nehari manifold directly as usual. The main result we obtained in this paper is new.


\end{abstract}

\maketitle

\section{Introduction and Main Result}

\setcounter{equation}{0}

In this paper, we study the following Schr\"{o}dinger-Poisson equation with critical nonlinearity:
\begin{equation}\label{1.1}
\left\{ \begin{gathered}
   - {\varepsilon ^2}\Delta u + V(x)u + \psi u = \lambda |u{|^{p - 2}}u + |u{|^4}u{\text{ in }}{\mathbb{R}^3}, \hfill \\
   - {\varepsilon ^2}\Delta \psi  = {u^2}{\text{ in }}{\mathbb{R}^3},{\text{ }}u > 0,{\text{ }}u \in {H^1}({\mathbb{R}^3}), \hfill \\
\end{gathered}  \right.
\end{equation}
where $\varepsilon > 0 $ is a small positive parameter, $\lambda  > 0$, $3 < p \le 4$. We assume that the potential $V$ satisfies:\\
$(V_1)$ $V \in C({\mathbb{R}^3},\mathbb{R})$ and $\mathop {\inf }\limits_{x \in {\mathbb{R}^3}} V(x) = \alpha  > 0$;\\
$(V_2)$ There is a bounded domain $\Lambda $ such that
\[
{V_0}: = \mathop {\inf }\limits_\Lambda  V < \mathop {\min }\limits_{\partial \Lambda } V.
\]
We also set $\mathcal{M}: = \{ x \in \Lambda :V(x) = {V_0}\} $. Without loss of generality, we may assume that $0 \in \mathcal{M}$.

Problem \eqref{1.1} is a variant of the following Schr\"{o}dinger-Poisson problem
\begin{equation}\label{1.2}
\left\{ \begin{gathered}
  \frac{{{\hbar ^2}}}
{{2m}}\Delta v - v - \omega \phi v + f(v) = 0{\text{ in }}{\mathbb{R}^3}, \hfill \\
  \Delta \phi  + 4\pi \omega {v^2} = 0{\text{ in }}{\mathbb{R}^3}, \hfill \\
  v,\phi  > 0,{\text{ }}v,\phi  \to 0{\text{ as }}|x| \to \infty , \hfill \\
\end{gathered}  \right.
\end{equation}
where $\hbar ,m,\omega  > 0$, $v,\phi :{\mathbb{R}^3} \to \mathbb{R}$, $f:\mathbb{R} \to \mathbb{R}$. This equation arises in Quantum Mechanics: in 1998, V. Benci and D. Fortunato \cite{bf} firstly introduced it as a model to describe the interaction of a charged particle with the electrostatic field. In \eqref{1.2}, $m$ denotes the mass of the particle, $\omega$ denotes the electric charge and $\hbar$ is a constant which is known under the name of Planck's constant. The unknowns of the equation are the wave function $v$ associated to the particle and the electric potential $\phi$. The presence of the nonlinear term $f(v)$ simulates the interaction effect among many particles.

In the last years, there has been a great deal of works dealing with the Schr\"{o}dinger-Poisson equations by means of variational tools.

V. Benci and D. Fortunato \cite{bf} considered the eigenvalue problem for \eqref{1.2} of the following form
\begin{equation}\label{1.3}
\left\{ \begin{gathered}
   - \frac{1}
{2}\Delta u - \phi u = \omega u{\text{ in }}\Omega {\text{,}} \hfill \\
  \Delta \phi  = 4\pi {u^2}{\text{ in }}\Omega {\text{,}} \hfill \\
  u(x) = 0,{\text{ }}\phi (x) = g{\text{ on }}\partial \Omega ,{\text{ }}{\left\| u \right\|_{{L^2}(\Omega )}} = 1,{\text{ }}\omega  > 0, \hfill \\
\end{gathered}  \right.
\end{equation}
where $\Omega $ is a bounded set in ${\mathbb{R}^3}$ and $g$ is a smooth function on the closure ${\bar \Omega }$. They used a constrained minimization argument to show that, there is a sequence $({\omega _n},{u_n},{\phi _n})$ with $\{ {\omega _n}\}  \subset \mathbb{R}$, ${\omega _n} \to \infty $ and ${u_n}$, ${\phi _n}$ real functions, solving \eqref{1.3}.

D. Ruiz \cite{r1} considered the following Schr\"{o}dinger-Poisson equation:
\begin{equation}\label{1.4}
\left\{ \begin{gathered}
   - \Delta u + u + \lambda \phi u = {u^{p - 1}}{\text{ in }}{\mathbb{R}^3}, \hfill \\
   - \Delta \phi  = {u^2}{\text{ in }}{\mathbb{R}^3}, \hfill \\
\end{gathered}  \right.
\end{equation}
where $\lambda  > 0$ is a positive parameter and $2 < p < 6$. Ruiz proved that when $2 < p < 3$ (respectively $p=3$), \eqref{1.4} has at least two (respectively one) positive solutions for $\lambda  > 0$ small by using the Mountain-Pass theorem (see \cite{ar}) and Ekeland's variational principle (see \cite{e}) and \eqref{1.4} has no nontrivial solution if $2 < p \le 3$, $\lambda  > \frac{1}
{4}$. For the case $3 < p < 6$, it was shown in \cite{r1} that there is a positive radial nontrivial solution to \eqref{1.4} by using the constrained minimization method on a new manifold which is obtained by combining
the usual Nehari manifold and the Pohozaev's identity.

A. Azzollini, P. d'Avenia and A. Pomponio \cite{adp} used a technique due to L. Jeanjean (\cite{j} Theorem 1.1) to show that the equation
\[\left\{ \begin{gathered}
   - \Delta u + q\phi u = g(u){\text{ in }}{\mathbb{R}^3}, \hfill \\
   - \Delta \phi  = q{u^2}{\text{ in }}{\mathbb{R}^3} \hfill \\
\end{gathered}  \right.\]
has a nontrivial positive radial solution $(u,\phi ) \in {H^1}({\mathbb{R}^3}) \times {D^{1,2}}({\mathbb{R}^3})$ for $q > 0$ small where the nonlinear term $g$ satisfies :\\
$(g_1)$ $g \in C(\mathbb{R},\mathbb{R})$;\\
$(g_2)$ $ - \infty  < \mathop {\underline {\lim } }\limits_{s \to {0^ + }} g(s)/s \le \mathop {\overline {\lim } }\limits_{s \to {0^ + }} g(s)/s =  - m < 0$;\\
$(g_3)$ $ - \infty  \le \mathop {\overline {\lim } }\limits_{s \to  + \infty } g(s)/{s^5} \le 0$;\\
$(g_4)$ $\exists \xi  > 0$ such that
\[
G(\xi ): = \int_0^\xi  {g(s)} ds > 0.
\]
Note that the hypotheses on $g$ was firstly introduced by H. Berestycki and P. L. Lions, in their celebrated paper \cite{bl1}.

D. Mugnai \cite{m} proved that for any $\omega  > 0$, there exist $\lambda  > 0$ such that the following Schr\"{o}dinger-Poisson equation
\begin{equation}\label{1.5}
\left\{ \begin{gathered}
   - \Delta u + \omega u - \lambda u\phi  + {W_u}(x,u) = 0{\text{ in }}{\mathbb{R}^3}, \hfill \\
   - \Delta \phi  = {u^2}{\text{ in }}{\mathbb{R}^3} \hfill \\
\end{gathered}  \right.
\end{equation}
has a nontrivial radial function $(u,\phi ) \in {H^1}({\mathbb{R}^3}) \times {D^{1,2}}({\mathbb{R}^3})$ by using the minimization argument on an appropriate manifold when the nonlinear term $W:{\mathbb{R}^3} \times \mathbb{R} \to \mathbb{R}$ satisfies:\\
$(W_1)$ $W:{\mathbb{R}^3} \times \mathbb{R} \to [0,\infty )$ is such that the derivative ${W_u}:{\mathbb{R}^3} \times \mathbb{R} \to \mathbb{R}$ ia a Carath\'{e}odory function, $W(x,s) = W(|x|,s)$ for a.e. $x \in {\mathbb{R}^3}$ and for every $s \in \mathbb{R}$, and $W(x,0) = {W_u}(x,0) = 0$ for a.e. $x \in {\mathbb{R}^3}$;\\
$(W_2)$ $\exists {C_1},{C_2} > 0$ and $1 < q < p < 5$ such that $|{W_u}(x,s)| \le {C_1}|s{|^q} + {C_2}|s{|^p}$ for every $s \in \mathbb{R}$ and a.e. $x \in {\mathbb{R}^3}$;\\
$(W_3)$ $\exists k \ge 2$ such that $0 \le s{W_u}(x,s) \le kW(x,s)$ for every $s \in \mathbb{R}$ and a.e. $x \in {\mathbb{R}^3}$.

Recently, Y. Jiang and H. Zhou \cite{jz} studied the Schr\"{o}dinger-Poisson equation
\begin{equation}\label{1.6}
\left\{ \begin{gathered}
   - \Delta u + (1 + \mu g(x))u + \lambda \phi u = |u{|^{p - 2}}u{\text{ in }}{\mathbb{R}^3}, \hfill \\
   - \Delta \phi  = {u^2}{\text{ in }}{\mathbb{R}^3},{\text{  }}\mathop {\lim }\limits_{|x| \to \infty } \phi (x) = 0, \hfill \\
\end{gathered}  \right.
\end{equation}
where $\lambda$, $\mu $ are positive parameters, $p \in (2,6)$, $g(x) \in {L^\infty }({\mathbb{R}^3})$ is
nonnegative, $g(x) \equiv 0$ on a bounded domain in ${\mathbb{R}^3}$ and $\mathop {\lim }\limits_{|x| \to \infty } g(x) = 1$. They used a priori estimate and approximation methods to show that \eqref{1.6} with $p \in (2,3)$ has a ground state solution if $\mu$ large and $\lambda$ small. Meanwhile, they also proved that \eqref{1.6} with $p \in [4,6)$ has a nontrivial solution for any $\lambda > 0$
and $\mu$ large.

As far as we know, there is no result on the existence of positive ground state solutions for \eqref{1.4} when the nonlinearity ${u^{p - 1}}(2 < p < 6)$ is replaced by $\lambda |u{|^{p - 2}}u + |u{|^4}u(3 < p \le 4)$. In this paper, we will fill this gap.

We note that problem \eqref{1.2} with  $\omega  = 0$ and $\frac{{{\hbar ^2}}}
{{2m}} = 1$ is motivated by the search for standing wave solutions  for the nonlinear Schr\"{o}dinger equation, which is one of the main subjects in nonlinear analysis. Different approaches have been taken to deal with this problem under various hypotheses on the potentials and the nonlinearities (see \cite{bl1,bl2} and so on).

Our motivation to study \eqref{1.1} mainly comes from the results of perturbed Schr\"{o}dinger equations, i.e.
\begin{equation}\label{1.7}
 - {\varepsilon ^2}\Delta u + V(x)u = |u{|^{q - 2}}u,{\text{ }}x \in {\mathbb{R}^N},
\end{equation}
where $2 < q < 2N/(N - 2)$, $N \ge 1$.

Many mathematicians proved the existence, concentration and multiplicity of solutions for \eqref{1.7}.

A. Floer and A. Weinstein \cite{fw} studied \eqref{1.7} in the case where $N = 1$, $q = 4$, $V \in {L^\infty }$ with $\inf V > 0$. They construct a single peak solution which concentrates around any given non-degenerate critical point of the potential $V$. Y. G. Oh \cite{o1,o2} extended this result in higher dimensions when $2 < q < 2N/(N - 2)$ and the potential $V$ belongs to a Kato class which means that $V$ satisfies the following condition:
\[
{(V)_a}:V \equiv a{\text{ or }}V > a{\text{ and }}{(V - a)^{ - \frac{1}
{2}}} \in {\text{Lip}}({\mathbb{R}^N}){\text{ for some }}a \in \mathbb{R}.
\]
Furthermore, Y. G. Oh \cite{o3} proved the existence of multi-peak solutions which concentrate around any finite subsets of the non-degenerate critical points of $V$. The arguments in \cite{fw,o1,o2,o3} are mainly based on a Lyapunov-Schmidt reduction.

P. Rabinowitz \cite{r} studied \eqref{1.7} under the conditions:\\
$(V_3)$ ${V_\infty } = \mathop {\lim \inf }\limits_{|x| \to \infty } V(x) > {V_0} = \mathop {\inf }\limits_{x \in {\mathbb{R}^N}} V(x) > 0$.\\
Rabinowitz proved that \eqref{1.7} possesses a positive ground state solution for $\varepsilon  > 0$ small by using the Mountain Pass Theorem (see \cite{ar}).

The concentration behavior for the family of positive ground state solutions, which was obtained in \cite{r}, was proved by X. Wang \cite{w}. Wang proved that the positive ground state solutions of \eqref{1.7} must concentrate at global minima of $V$ as $\varepsilon  \to 0$.

Under the same condition $(V_3)$ on $V(x)$, S. Cingolani and N. Lazzo \cite{cl} proved the multiplicity of positive ground state solutions for \eqref{1.7} by using Ljusternik-Schnirelmann theory(see \cite{c}, for example).

M. del Pino and P. L. Felmer \cite{pf} studied \eqref{1.7} with the conditions on $V$ replaced by $(V_1)$ and $(V_2)$. They proved that \eqref{1.7} possesses a positive bound state solution for $\varepsilon  > 0$ small which concentrates around the local minima of $V$ in $\Lambda$ as $\varepsilon  \to 0$.

C. Gui \cite{g} studied \eqref{1.7} under the conditions $(V_1)$ and\\
$(V_4)$ There exist $k$ disjoint bounded regions ${\Omega _1},...,{\Omega _k}$ such that
\[
{V_0}: = \mathop {\inf }\limits_{{\Omega _i}} V < \mathop {\min }\limits_{\partial {\Omega _i}} V,{\text{ }}i = 1,...k.
\]
Gui showed that \eqref{1.7} possesses a positive classial bound state solution for $\varepsilon  > 0$ small which exactly possesses $k$ local maximum ${P_{\varepsilon ,1}},...,{P_{\varepsilon ,k}}$ satisfying ${P_{\varepsilon ,i}} \in {\Omega _i}$  and $\mathop {\lim }\limits_{\varepsilon  \to 0} V({P_{\varepsilon ,i}}) = \mathop {\inf }\limits_{{\Omega _i}} V$.

T. D'Aprile and J. Wei \cite{dw} studied \eqref{1.2} and extended the method in \cite{fw,o1,o2,o3,o3}, which was based on Lyapunov-Schmidt reduction, to conclude a similar result in the Schr\"{o}dinger-Poisson equation \eqref{1.2}.

Under the same condition $(V_3)$ on $V(x)$, X. He \cite{h} studied \eqref{1.1} with the nonlinearity replaced by $f(u)$, where $f \in {C^1}({\mathbb{R}^ + },{\mathbb{R}^ + })$ and satisfies the Ambrosetti-Rabinowitz condition ((AR) condition in short)
\[
\exists \mu  > 4,{\text{ }}0 < \mu \int_0^u {f(s)} ds \le f(u)u,
\]
$\mathop {\lim }\limits_{s \to 0} \frac{{f(s)}}
{{{s^3}}} = 0$, $\mathop {\lim }\limits_{|s| \to \infty } \frac{{f(s)}}
{{|s{|^q}}} = 0$ for some $3 < q < 5$ and $\frac{{f(s)}}
{{{s^3}}}$ is strictly increasing for $s > 0$. They obtained the existence, concentration and multiplicity of solutions for \eqref{1.7} by the same arguments as in \cite{r,w,cl}.

For more results, we can refer to \cite{a,ar1,a1,bf1,cv,dm,dw1,r2,v} and the references therein.

Our main result is the following:

\begin{theorem}\label{1.1.}
Let $(V_1)$, $(V_2)$ hold. There exist ${\lambda ^ * } > 0$ and ${\varepsilon ^ * } > 0$ such that for each $\lambda  \in [{\lambda ^ * },\infty )$ and $\varepsilon  \in (0,{\varepsilon ^ * })$, \eqref{1.1} possesses a positive solution ${u_\varepsilon } \in {H^1}({\mathbb{R}^3})$ such that\\
$(i)$ there exists a maximum point ${x_\varepsilon }$ of ${u_\varepsilon }$ such that
\[
\mathop {\lim }\limits_{\varepsilon  \to 0} {\text{dist}}({x_\varepsilon },\mathcal{M}) = 0;
\]
$(ii)$ $\exists {C_1},{C_2} > 0$, such that
\[
{u_\varepsilon }(x) \le {C_1}\exp \Bigl( { - \frac{{{C_2}}}
{\varepsilon }|x - {x_\varepsilon }|} \Bigr),
\]
where $C_1$, $C_2$ are independent of $\varepsilon$.
\end{theorem}

We note that, to the best of our knowledge, there is no result on the existence and concentration of positive bound state solutions for Schr\"{o}dinger-Poisson type equation with the nonlinearity $\lambda |u{|^{p - 2}}u + |u{|^4}u(3 < p \le 4)$.

The proof of Theorem~\ref{1.1.} is based on variational method. The main difficulties in proving Theorem~\ref{1.1.} lie in two aspects: (i) The nonlinearity $\lambda |u{|^{p - 2}}u + |u{|^4}u$ with $p \in (3,4]$ does not satisfy $({\text{AR}})$ condition and the fact that the function $\frac{{\lambda {u^{p - 1}} + {u^5}}}
{{{u^3}}}$ is not increasing for $u > 0$  prevent us from obtaining a bounded Palais-smale sequence ((PS) sequence in short) and using the Nehari manifold respectively. The arguments in \cite{pf} can not be applied in this paper. (ii) The unboundedness of the domain $\mathbb{R}^3$ and the nonlinearity $\lambda |u{|^{p - 2}}u + |u{|^4}u(3 < p \le 4)$ with the critical Sobolev growth lead to the lack of compactness. As we will see later, the above two aspects prevent us from using the variational method in a standard way.

To overcome these difficulties, inspired by \cite{bj,fis}, we use a version of quantitative deformation lemma due to G. M. Figueiredo, N. Ikoma, J. R. Santos Junior (see Proposition~\ref{4.6.} below) to construct a special bounded (PS) sequence and recover the compactness by using a penalization method which was firstly introduced in \cite{bw}.

To complete this section, we sketch our proof.

Firstly, we need to consider the existence of ground state solutions of the associated "limiting problem" of \eqref{1.1}, which is given as
\begin{equation}\label{1.8}
\left\{ \begin{gathered}
   - \Delta u + au + \phi u = \lambda |u{|^{p - 2}}u + |u{|^4}u{\text{ in }}{\mathbb{R}^3}, \hfill \\
   - \Delta \phi  = {u^2}{\text{ in }}{\mathbb{R}^3},{\text{ }}u > 0,{\text{ }}u \in {H^1}({\mathbb{R}^3}), \hfill \\
  a > 0,{\text{ }}3 < p \le 4 \hfill \\
\end{gathered}  \right.
\end{equation}
with the corresponding energy functional
\[\begin{gathered}
  {I_a}(u) = \frac{1}
{2}\int_{{\mathbb{R}^3}} {|\nabla u{|^2}}  + \frac{a}
{2}\int_{{\mathbb{R}^3}} {{u^2}}  + \frac{1}
{{16\pi }}\int_{{\mathbb{R}^3}} {\int_{{\mathbb{R}^3}} {\frac{{{u^2}(x){u^2}(y)}}
{{|x - y|}}} } dxdy \hfill \\
  {\text{     }} - \frac{\lambda }
{p}\int_{{\mathbb{R}^3}} {{{({u^ + })}^p}}  - \frac{1}
{6}\int_{{\mathbb{R}^3}} {{{({u^ + })}^6}} ,{\text{ }}u \in {H^1}({\mathbb{R}^3}). \hfill \\
\end{gathered} \]

In \cite{hit}, J. Hirata, N. Ikoma and K. Tanaka studied the following Schr\"{o}dinger equation
\[
 - \Delta u = g(u),{\text{ }}u \in {H^1}({\mathbb{R}^N})
\]
with the corresponding energy functional
\[
I(u) = \frac{1}
{2}\int_{{\mathbb{R}^N}} {|\nabla u{|^2}}  - \int_{{\mathbb{R}^N}} {G(u)} ,{\text{ }}u \in H_r^1({\mathbb{R}^N}),
\]
where $G(u) = \int_0^u {g(s)} ds$ and $g$ satisfies the conditions due to the celebrated work by H. Berestycki and P. L. Lions \cite{bl1}. By studing the behavior of $I(u({e^{ - \theta }}x))$ for $\theta  \in \mathbb{R}$, they constructed a ${{\text{(PS)}}_c}$ sequence $\{ {u_n}\} _{n = 1}^\infty $ with an extra property $P({u_n}) \to 0$ as $n \to \infty $ where $c$ is the mountain pass level of $I$ and $P(u)=0$ is the corresponding Pohozaev's identity and then proved that the ${{\text{(PS)}}_{{c}}}$ sequence is bounded. But for the Schr\"{o}dinger-Poisson equation \eqref{1.8}, one still need something more than $P({u_n}) \to 0$ as $n \to \infty $.

For the critical case \eqref{1.8}, the constrained minimization on a new manifold due to D. Ruiz \cite{r1} seems to be difficult to be applied directly.

Motivated by \cite{hit}, by studying the behavior of ${I_a}({e^{2\theta }}u({e^\theta }x))$ for $\theta  \in \mathbb{R}$, we construct a ${{\text{(PS)}}_{{c_a}}}$ sequence $\{ {u_n}\} _{n = 1}^\infty $ with an extra property ${G_a}({u_n}) \to 0$ as $n \to \infty $ where $c_a$ is the mountain pass level of $I_a$, ${G_a}(u) = 2\left\langle {{{I'}_a}(u),u} \right\rangle  - {P_a}(u)$ and ${P_a}(u)=0$ is the Pohozaev's identity of \eqref{1.8} (see Proposition~\ref{add3} below). From this fact, the boundedness of the ${{\text{(PS)}}_{{c_a}}}$ sequence is proved easily. Proceeding by the standard arguments, the existence of ground state solution \eqref{1.8} follows (see Proposition~\ref{3.6.} below). Denoting $S_a$ the set of ground state solutions $U$ of \eqref{1.8} satisfying $U(0) = \mathop {\max }\limits_{x \in {\mathbb{R}^3}} U(x)$, we then show that $S_a$ is compact in ${H^1}({\mathbb{R}^3})$ (see Proposition~\ref{3.7.} below).

To study \eqref{1.1}, We will work with the following equivalent equation
\begin{equation}\label{1.9}
\left\{ \begin{gathered}
   - \Delta u + V(\varepsilon x)u + \phi u = \lambda |u{|^{p - 2}}u + |u{|^4}u{\text{ in }}{\mathbb{R}^3}, \hfill \\
   - \Delta \phi  = {u^2}{\text{ in }}{\mathbb{R}^3},{\text{ }}u > 0,{\text{ }}u \in {H^1}({\mathbb{R}^3}) \hfill \\
\end{gathered}  \right.
\end{equation}
with the energy functional
\[
\begin{gathered}
  {I_\varepsilon }(u) = \frac{1}
{2}\int_{{\mathbb{R}^3}} {|\nabla u{|^2}}  + \frac{1}
{2}\int_{{\mathbb{R}^3}} {V(\varepsilon x){u^2}}  + \frac{1}
{{16\pi }}\int_{{\mathbb{R}^3}} {\int_{{\mathbb{R}^3}} {\frac{{{u^2}(x){u^2}(y)}}
{{|x - y|}}} }  \hfill \\
  {\text{   }} - \frac{\lambda }
{p}\int_{{\mathbb{R}^3}} {{{({u^ + })}^p}}  - \frac{1}
{6}\int_{{\mathbb{R}^3}} {{{({u^ + })}^6}} ,{\text{ }}u \in {H_\varepsilon }, \hfill \\
\end{gathered}
\]
where ${H_\varepsilon }: = \{ {v \in {H^1}({\mathbb{R}^3})| {\int_{{\mathbb{R}^3}} {V(\varepsilon x){v^2}}  < \infty } } \}$ endowed with the norm
\[
{\left\| v \right\|_{{H_\varepsilon }}}: = {\Bigl( {\int_{{\mathbb{R}^3}} {|\nabla v{|^2}}  + \int_{{\mathbb{R}^3}} {V(\varepsilon x){v^2}} } \Bigr)^{1/2}}.
\]

Unlike \cite{h}, where the minimum of $V(x)$ is global and the nonlinear term $f(u)$ satisfies the $({\text{AR}})$ condition, the Mountain Pass Theorem can be used globally, here in the present paper, the condition $(V_2)$ is local and $3 < p \le 4$, we need to use a penalization method introduced in \cite{bw}, which helps us to overcome the obstacle caused by the non-compactness due to the
unboundedness of the domain and the lack of $({\text{AR}})$ condition. To this end, we should modify the energy functional.

Following \cite{bj}, we set ${J_\varepsilon }:{H_\varepsilon } \to \mathbb{R}$ be given by
\[
{J_\varepsilon }(v) = {I_\varepsilon }(v) + {Q_\varepsilon }(v),
\]
where
\[
{Q_\varepsilon }(v) = \Bigl( {\int_{{\mathbb{R}^3}} {{\chi _\varepsilon }{v^2}}  - 1} \Bigr)_ + ^2
\]
and
\[
{\chi _\varepsilon }(x) = \left\{ \begin{gathered}
  0{\text{ if }}x \in \Lambda /\varepsilon , \hfill \\
  {\varepsilon ^{ - 1}}{\text{ if }}x \notin \Lambda /\varepsilon.  \hfill \\
\end{gathered}  \right.
\]
It will be shown that the functional ${Q_\varepsilon }$ will acts as a penalization to force the concentration phenomena to occur inside $\Lambda $ (see Lemma~\ref{4.3.} below).

Using a version of quantitative deformation lemma due to G. M. Figueiredo, N. Ikoma, J. R. Santos Junior (see Proposition~\ref{4.6.} below) to construct a special bounded and convergent (PS) sequence of ${J_\varepsilon }$ in a neighborhood of the compact set ${S_{{V_0}}}$ for $\varepsilon  > 0$ small, i.e. ${J_\varepsilon }$ possesses a critical point ${v_\varepsilon }$. To verify the critical point ${v_\varepsilon }$ of ${J_\varepsilon }$ is indeed a solution of the original problem \eqref{1.9}, we need to establish a uniform estimate on ${L^\infty }$-norm of ${v_\varepsilon }$ (independent of $\varepsilon $) by using the idea of Brezis-Kato type argument and the Moser iteration technique (see also \cite{l1,zy} and Lemma~\ref{2.4.} below).

Moreover, for the critical case, the existence and concentration phenomenon of problem \eqref{1.1} has not been studied so far by variational methods. In the present paper, we will adopt some ideas of Byeon and Jeanjean \cite{bj} to study the existence and concentration of positive solutions for equation \eqref{1.1} with critical growth. But the method of Byeon and Jeanjean \cite{bj} can not be used directly and more careful analysis is needed. For this aspect, we refer to \cite{bc,rww,zcz}.

This paper is organized as follows, in Section 2, we give some
preliminary results. In Section 3, we analyze the "limiting problem" \eqref{1.8} and show the existence of ground state solutions. In Section 4, we prove the main result Theorem~\ref{1.1.}.

\section{Preliminaries}

\setcounter{equation}{0}

In the following, we recall that by the Lax-Milgram theorem, for each $u \in {H^1}({\mathbb{R}^3})$, there exists a unique ${\phi _u} \in {D^{1,2}}({\mathbb{R}^3})$ such that $ - \Delta {\phi _u} = {u^2}$. Moreover, ${\phi _u}$ can be expressed as
\[
{\phi _u}(x) = \frac{1}
{{4\pi }}\int_{{\mathbb{R}^3}} {\frac{{{u^2}(y)}}
{{|x - y|}}} dy.
\]
The function ${\phi _u}$ has the following property, see \cite{cv} and \cite{r1}.

\begin{lemma}\label{2.1.}
For any $u \in {H^1}({\mathbb{R}^3})$, we have\\
(i) $\left\| {{\phi _u}} \right\|_{{D^{1,2}}({\mathbb{R}^3})}^2 = \int_{{\mathbb{R}^3}} {{\phi _u}{u^2}}  \le C\left\| u \right\|_{{L^{12/5}}({\mathbb{R}^3})}^4 \le C\left\| u \right\|_{{H^1}({\mathbb{R}^3})}^4$;\\
(ii) ${\phi _u} \ge 0$;\\
(iii) If ${u_n} \rightharpoonup u$ in ${H^1}({\mathbb{R}^3})$, then ${\phi _{{u_n}}} \rightharpoonup {\phi _u}$ in ${{D^{1,2}}({\mathbb{R}^3})}$ and $\int_{{\mathbb{R}^3}} {{\phi _u}{u^2}}  \le \mathop {\underline {\lim } }\limits_{n \to \infty } \int_{{\mathbb{R}^3}} {{\phi _{{u_n}}}u_n^2} $;\\
(iv) If $y \in {\mathbb{R}^3}$ and $\tilde u(x) = u(x + y)$, then ${\phi _{\tilde u}}(x) = {\phi _u}(x + y)$ and $\int_{{\mathbb{R}^3}} {{\phi _{\tilde u}}{{\tilde u}^2}}  = \int_{{\mathbb{R}^3}} {{\phi _u}{u^2}} $.
\end{lemma}

Define $N:{H^1}({\mathbb{R}^3}) \to \mathbb{R}$ by
\[
N(u) = \int_{{\mathbb{R}^3}} {{\phi _u}{u^2}}.
\]
Then, the functional $N$ and its derivatives ${N'}$ and ${N''}$ possess Brezis-Lieb splitting property, which is similar to the well-known Brezis-Lieb's Lemma (see \cite{bl}) and can be stated as the following form (see \cite{zz}).

\begin{lemma}\label{2.2.}
Let ${u_n} \rightharpoonup u$ in ${H^1}({\mathbb{R}^3})$ and ${u_n} \to u$ a.e. in ${\mathbb{R}^3}$, then, as $n \to \infty $,\\
(i) $N({u_n} - u) = N({u_n}) - N(u) + o(1)$;\\
(ii) $N'({u_n} - u) = N'({u_n}) - N'(u) + o(1)$ in ${H^{ - 1}}({\mathbb{R}^3})$ and $N':{H^1}({\mathbb{R}^3}) \to {H^{ - 1}}({\mathbb{R}^3})$ is weakly sequentially continuous;\\
(iii) $N''({u_n} - u) = N''({u_n}) - N''(u) + o(1)$ in $L({H^1}({\mathbb{R}^3}),{H^{ - 1}}({\mathbb{R}^3}))$ and $N''(u) \in L({H^1}({\mathbb{R}^3}),{H^{ - 1}}({\mathbb{R}^3}))$ is compact for any $u \in {H^1}({\mathbb{R}^3})$.
\end{lemma}

\begin{lemma}\label{2.3.}
(General Minimax Principle) (\cite{w1} Theorem 2.8)\\
Let $X$ be a Banach space. Let ${M_0}$ be a closed subspace of the metric space $M$ and ${\Gamma _0} \subset C({M_0},X)$. Define
\[
\Gamma : = \left\{ {\gamma  \in C(M,X):\gamma \left| {_{{M_0}} \in {\Gamma _0}} \right.} \right\}.
\]
If $\varphi  \in {C^1}(X,\mathbb{R})$ satisfies
\[
\infty  > c: = \mathop {\inf }\limits_{\gamma  \in \Gamma } \mathop {\sup }\limits_{u \in M} \varphi (\gamma (u)) > a: = \mathop {\sup }\limits_{{\gamma _0} \in {\Gamma _0}} \mathop {\sup }\limits_{u \in {M_0}} \varphi ({\gamma _0}(u)),
\]
then, for every $\varepsilon  \in (0,(c - a)/2)$, $\delta  > 0$ and ${\gamma  \in \Gamma }$ such that $\mathop {\sup }\limits_M \varphi  \circ \gamma  \le c + \varepsilon $, there exists $u \in X$ such that\\
$(a)$ $c - 2\varepsilon  \le \varphi (u) \le c + 2\varepsilon $,\\
$(b)$ ${\text{dist}}(u,\gamma (M)) \le 2\delta $,\\
$(c)$ $\left\| {\varphi '(u)} \right\| \le 8\varepsilon /\delta $.
\end{lemma}

Consider the following equation
\begin{equation}\label{2.1}
 - \Delta u + {V_n}(x)u = {f_n}( {x,u} ){\text{ in }}{\mathbb{R}^3},
\end{equation}
where $\{ {V_n}\} $ is a sequence of continuous functions satisfying for some positive constant $\alpha $ independent of $n$ such that
\[
{V_n}(x) \ge \alpha  > 0{\text{ for all }}x \in {\mathbb{R}^3}
\]
and ${f_n}({x,t})$ is a Carathedory function such that for any $\delta  > 0$, there exists ${C_\delta } > 0$ and
\[
|{f_n}( {x,t} )| \le \delta | t | + {C_\delta }{| t |^5},{\text{ }}\forall (x,t) \in {\mathbb{R}^3} \times \mathbb{R},
\]
where $\delta$ is independent of $n$.

From the process of proof of Theorem 1 in \cite{zy} and Theorem 1.11 in \cite{l1}, we have the following lemma:

\begin{lemma}\label{2.4.}
Assume that $\{ {v_n}\} $ is a sequence of weak solutions to \eqref{2.1} satisfying ${\left\| {{v_n}} \right\|_{{H^1}({\mathbb{R}^3})}} \le C$ for $n \in \mathbb{N}$.\\
(i) If $\{ |{v_n}{|^6}\} $ is uniformly integrable in any bounded domain in ${\mathbb{R}^3}$, then for any ${x_0} \in {\mathbb{R}^3}$, $\exists {R_0}({x_0}) > 0$ such that
\[
{\left\| {{v_n}} \right\|_{{L^\infty }\left( {{B_{{R_0}({x_0})/4}}\left( {{x_0}} \right)} \right)}} \le C({R_0}({x_0})),
\]
where ${R_0}({x_0})$ and $C({R_0}({x_0}))$ are independent of $n$.\\
(ii) If $\{ |{v_n}{|^6}\} $ is uniformly integrable near $\infty $, i.e. $\forall \varepsilon  > 0$, $\exists R > 0$, for any $r > R$, $\int_{{\mathbb{R}^3}\backslash {B_r}(0)} {|{v_n}{|^6}}  < \varepsilon $, then
\[
\mathop {\lim }\limits_{|x| \to \infty } {v_n}(x) = 0{\text{ uniformly for }}n.
\]
\end{lemma}
\begin{proof}
See Lemma 2.10 of \cite{hl}.
\end{proof}

\begin{lemma}\label{2.5.}
(\cite{rww}) Let $R$ be a positive number and $\{ {u_n}\} $ a bounded sequence in ${H^1}({\mathbb{R}^N})$$(N \ge 3)$. If
\[
\mathop {\lim }\limits_{n \to \infty } \mathop {\sup }\limits_{x \in {\mathbb{R}^N}} \int_{{B_R}(x)} {|{u_n}{|^{2N/(N - 2)}}}  = 0,
\]
then ${u_n} \to 0$ in ${L^{2N/(N - 2)}}({\mathbb{R}^N})$ as $n \to \infty $.
\end{lemma}

\begin{lemma}\label{2.6.}
(Lemma 2.7 of \cite{bc})
Let $\{ {u_n}\}  \subset H_{{\text{loc}}}^1({\mathbb{R}^N})$$(N \ge 3)$ be a sequence of functions such that
\[
{u_n} \rightharpoonup 0{\text{ in }}{H^1}({\mathbb{R}^N}).
\]
Suppose that there exist a bounded open set $Q \subset {\mathbb{R}^N}$ and a positive constant $\gamma  > 0$ such that
\[
\int_Q {|\nabla {u_n}{|^2}}  \ge \gamma  > 0,{\text{ }}\int_Q {|{u_n}{|^{2N/(N - 2)}}}  \ge \gamma  > 0.
\]
Moreover suppose that
\[
\Delta {u_n} + |{u_n}{|^{4/(N - 2)}}{u_n} = {\chi _n},
\]
where ${\chi _n} \in {H^{ - 1}}({\mathbb{R}^N})$ and
\[
|\langle {{\chi _n},\varphi } \rangle | \le {\varepsilon _n}{\left\| \varphi  \right\|_{{H^1}({\mathbb{R}^N})}},{\text{ }}\forall \varphi  \in C_c^\infty (U),
\]
where $U$ is an open neighborhood of $Q$ and $\{ {\varepsilon _n}\} $ is a sequence of positive numbers converging to $0$. Then there exist a sequence of points $\{ {y_n}\}  \subset {\mathbb{R}^N}$ and a sequence of positive numbers $\{ {\sigma _n}\}$ such that
\[
{v_n}(x): = \sigma _n^{(N - 2)/2}{u_n}({\sigma _n}x + {y_n})
\]
converges weakly in ${D^{1,2}}({\mathbb{R}^N})$ to a nontrivial solution $v$ of
\[
 - \Delta u = |u{|^{4/(N - 2)}}u,{\text{ }}u \in {D^{1,2}}({\mathbb{R}^N}).
\]
Moreover,
\[
{y_n} \to \bar y \in \bar Q{\text{ and }}{\sigma _n} \to 0.
\]
\end{lemma}

The following lemma is a special case of Lemma 8.17 in \cite{gt} for $\Delta$.

\begin{lemma}\label{2.7.}
(Lemma 8.17 of \cite{gt}) Let $\Omega $ be an open subset of ${\mathbb{R}^N}$$(N \ge 2)$. Suppose that $t > N$, $h \in {L^{t/2}}(\Omega )$ and $u \in {H^1}(\Omega )$ satisfies $ - \Delta u(x) \le h(x),{\text{ }}x \in \Omega $ in the weak sense. Then for any ball ${B_{2r}}(y) \subset \Omega $,
\[
\mathop {\sup }\limits_{{B_r}(y)} u \le C\Bigl( {{{\left\| {{u^ + }} \right\|}_{{L^2}\left( {{B_{2r}}(y)} \right)}} + {{\left\| h \right\|}_{{L^{t/2}}\left( {{B_{2r}}(y)} \right)}}} \Bigr),
\]
where $C = C(N,t,r)$ is independent of $y$.
\end{lemma}

\section{The limiting problem}

The following equation for $a > 0$
\begin{equation}\label{3.1}
\left\{ \begin{gathered}
   - \Delta u + au + \phi u = \lambda |u{|^{p - 2}}u + |u{|^4}u{\text{ in }}{\mathbb{R}^3}, \hfill \\
   - \Delta \phi  = {u^2}{\text{ in }}{\mathbb{R}^3},{\text{ }}u > 0,{\text{ }}u \in {H^1}({\mathbb{R}^3}) \hfill \\
\end{gathered}  \right.
\end{equation}
 is the limiting equation of \eqref{1.1}.

We define the energy functional for the limiting problem \eqref{3.1} by
\[{I_a}(u) = \frac{1}
{2}\int_{{\mathbb{R}^3}} {|\nabla u{|^2}}  + \frac{a}
{2}\int_{{\mathbb{R}^3}} {{u^2}}  + \frac{1}
{4}\int_{{\mathbb{R}^3}} {{\phi _u}{u^2}}  - \frac{\lambda }
{p}\int_{{\mathbb{R}^3}} {{{({u^ + })}^p}}  - \frac{1}
{6}\int_{{\mathbb{R}^3}} {{{({u^ + })}^6}} ,{\text{ }}u \in {H^1}({\mathbb{R}^3}).\]
In view of \cite{ps}, if $u \in {H^1}({\mathbb{R}^3})$ is a weak solution to problem  \eqref{3.1}, then we have the following Pohozaev's identity:
\begin{equation}\label{3.2}
{P_a}(u) = \frac{1}
{2}\int_{{\mathbb{R}^3}} {|\nabla u{|^2}}  + \frac{3}
{2}a\int_{{\mathbb{R}^3}} {{u^2}}  + \frac{5}
{4}\int_{{\mathbb{R}^3}} {{\phi _u}{u^2}}  - \frac{3}
{p}\lambda \int_{{\mathbb{R}^3}} {{{({u^ + })}^p}}  - \frac{1}
{2}\int_{{\mathbb{R}^3}} {{{({u^ + })}^6}}  = 0.
\end{equation}

As in \cite{r1}, we introduce the following manifold
\[
{M_a}: = \left\{ {u \in {H^1}({\mathbb{R}^3})\backslash \{ 0\} \left| {{G_a}(u) = 0} \right.} \right\},
\]
where
\[
{G_a}(u) = \frac{3}
{2}\int_{{\mathbb{R}^3}} {|\nabla u{|^2}}  + \frac{1}
{2}a\int_{{\mathbb{R}^3}} {{u^2}}  + \frac{3}
{4}\int_{{\mathbb{R}^3}} {{\phi _u}{u^2}}  - \frac{{(2p - 3)}}
{p}\lambda \int_{{\mathbb{R}^3}} {{{({u^ + })}^p}}  - \frac{3}
{2}\int_{{\mathbb{R}^3}} {{{({u^ + })}^6}} .
\]
It is clear that
\begin{equation}\label{3.3}
{G_a}(u) = 2\left\langle {{{I'}_a}(u),u} \right\rangle  - {P_a}(u),
\end{equation}
where ${P_a}(u)$ is given in \eqref{3.2}.

\begin{remark}\label{3.1.}
If $u \in {H^1}({\mathbb{R}^3})$ is a nontrivial weak solution to \eqref{3.1}, then by \eqref{3.2}, \eqref{3.3}, we see that $u \in {M_a}$.
\end{remark}

\begin{lemma}\label{3.2.}
For any ${u \in {H^1}({\mathbb{R}^3})\backslash \{ 0\} }$, there is a unique $\tilde t > 0$ such that ${u_{\tilde t}} \in {M_a}$, where ${u_{\tilde t}}(x): = {{\tilde t}^2}u(\tilde tx)$. Moreover, ${I_a}({u_{\tilde t}}) = \mathop {\max }\limits_{t > 0} {I_a}({u_t})$.
\end{lemma}

\begin{proof}
For any ${u \in {H^1}({\mathbb{R}^3})\backslash \{ 0\} }$ and $t>0$, set ${u_t}(x): = {t^2}u(tx)$. Consider
\[
\gamma (t): = {I_a}({u_t}) = \frac{1}
{2}{t^3}\int_{{\mathbb{R}^3}} {|\nabla u{|^2}}  + \frac{1}
{2}at\int_{{\mathbb{R}^3}} {{u^2}}  + \frac{1}
{4}{t^3}\int_{{\mathbb{R}^3}} {{\phi _u}{u^2}}  - \frac{\lambda }
{p}{t^{2p - 3}}\int_{{\mathbb{R}^3}} {{{({u^ + })}^p}}  - \frac{1}
{6}{t^9}\int_{{\mathbb{R}^3}} {{{({u^ + })}^6}} .
\]
Since $2p-3 > 3$, by elementary computations, $\gamma (t)$ has a unique critical point $\tilde t > 0$ corresponding to its maximum, i.e. $\gamma (\tilde t) = \mathop {\max }\limits_{t > 0} \gamma (t)$ and $\gamma '(\tilde t) = 0$. Hence
\[
\frac{3}
{2}{{\tilde t}^2}\int_{{\mathbb{R}^3}} {|\nabla u{|^2}}  + \frac{1}
{2}a\int_{{\mathbb{R}^3}} {{u^2}}  + \frac{3}
{4}{{\tilde t}^2}\int_{{\mathbb{R}^3}} {{\phi _u}{u^2}}  - \frac{{(2p - 3)}}
{p}\lambda {{\tilde t}^{2p - 4}}\int_{{\mathbb{R}^3}} {{{({u^ + })}^p}}  - \frac{3}
{2}{{\tilde t}^8}\int_{{\mathbb{R}^3}} {{{({u^ + })}^6}}  = 0,
\]
then ${G_a}({u_{\tilde t}}) = 0$, ${u_{\tilde t}} \in {M_a}$ and ${I_a}({u_{\tilde t}}) = \mathop {\max }\limits_{t > 0} {I_a}({u_t})$.
\end{proof}

\begin{lemma}\label{3.3.}
${I_a}$ possesses the Mountain-Pass geometry.
\end{lemma}

\begin{proof}
$\exists \rho ,\delta  > 0$ small such that
\[\begin{gathered}
  {I_a}(u) = \frac{1}
{2}\int_{{\mathbb{R}^3}} {|\nabla u{|^2}}  + \frac{1}
{2}a\int_{{\mathbb{R}^3}} {{u^2}}  + \frac{1}
{4}\int_{{\mathbb{R}^3}} {{\phi _u}{u^2}}  - \frac{\lambda }
{p}\int_{{\mathbb{R}^3}} {{{({u^ + })}^p}}  - \frac{1}
{6}\int_{{\mathbb{R}^3}} {{{({u^ + })}^6}}  \hfill \\
  {\text{ }} \ge \frac{1}
{2}\left\| u \right\|_{{H^1}({\mathbb{R}^3})}^2 - C\lambda \left\| u \right\|_{{H^1}({\mathbb{R}^3})}^p - C\left\| u \right\|_{{H^1}({\mathbb{R}^3})}^6 \hfill \\
  {\text{ }} \ge \delta  > 0{\text{ for }}{\left\| u \right\|_{{H^1}({\mathbb{R}^3})}} = \rho  > 0. \hfill \\
\end{gathered} \]
Fix ${u \in {H^1}({\mathbb{R}^3})\backslash \{ 0\} }$, set ${u_t}(x): = {t^2}u(tx)$,
\[{I_a}({u_t}) = \frac{1}
{2}{t^3}\int_{{\mathbb{R}^3}} {|\nabla u{|^2}}  + \frac{1}
{2}at\int_{{\mathbb{R}^3}} {{u^2}}  + \frac{1}
{4}{t^3}\int_{{\mathbb{R}^3}} {{\phi _u}{u^2}}  - \frac{\lambda }
{p}{t^{2p - 3}}\int_{{\mathbb{R}^3}} {{{({u^ + })}^p}}  - \frac{1}
{6}{t^9}\int_{{\mathbb{R}^3}} {{{({u^ + })}^6}}  < 0\]
for $t > 0$ large, then $\exists {t_0} > 0$, set ${u_0}: = {u_{{t_0}}}$, $I({u_0}) < 0$.
\end{proof}

Hence we can define the Mountain-Pass level of ${I_a}$:
\begin{equation}\label{3.4}
{c_a}: = \mathop {\inf }\limits_{\gamma  \in {\Gamma _a}} \mathop {\sup }\limits_{t \in [0,1]} {I_a}(\gamma (t)),
\end{equation}
where the set of paths is defined as
\begin{equation}\label{3.5}
{\Gamma_a}: = \left\{ {\gamma  \in C([0,1],{H^1}({\mathbb{R}^3})):\gamma (0) = 0{\text{ and }}{I_a}(\gamma (1)) < 0} \right\}.
\end{equation}
Next, we will construct a (PS) sequence $\{ {u_n}\} _{n = 1}^\infty $ for $I_a$ at the level $c_a$ that satisfies ${G_a}({u_n}) \to 0$ as $n \to \infty $ i.e.

\begin{proposition}\label{add3}
There exists a sequence $\{ {u_n}\} _{n = 1}^\infty $ in ${H^1}({\mathbb{R}^3})$ such that, as $n \to \infty $,
\begin{equation}\label{3.12}
{I_a}({u_n}) \to {c_a},{\text{ }}{{I'}_a}({u_n}) \to 0,{\text{ }}{G_a}({u_n}) \to 0.
\end{equation}
\end{proposition}

\begin{proof}
We define the map $\Phi :\mathbb{R} \times {H^1}({\mathbb{R}^3}) \to {H^1}({\mathbb{R}^3})$ for $\theta  \in \mathbb{R}$, $v \in {H^1}({\mathbb{R}^3})$ and $x \in {\mathbb{R}^3}$ by $\Phi (\theta ,v) = {e^{2\theta }}v({e^\theta }x)$. For every $\theta  \in \mathbb{R}$, $v \in {H^1}({\mathbb{R}^3})$, the functional ${I_a} \circ \Phi $ is computed as
\[\begin{gathered}
  {I_a} \circ \Phi (\theta ,v) = \frac{1}
{2}{e^{3\theta }}\int_{{\mathbb{R}^3}} {|\nabla v{|^2}}  + \frac{1}
{2}a{e^\theta }\int_{{\mathbb{R}^3}} {{v^2}}  + \frac{1}
{4}{e^{3\theta }}\int_{{\mathbb{R}^3}} {{\phi _v}{v^2}}  \hfill \\
  {\text{ }} - \frac{\lambda }
{p}{e^{(2p - 3)\theta }}\int_{{\mathbb{R}^3}} {{{({v^ + })}^p}}  - \frac{1}
{6}{e^{9\theta }}\int_{{\mathbb{R}^3}} {{{({v^ + })}^6}} . \hfill \\
\end{gathered} \]
In view of Lemma~\ref{3.3.}, we can easily check that ${I_a} \circ \Phi (\theta ,v) > 0$ for all $(\theta ,v)$ with $|\theta |$, ${\| v \|_{{H^1}({\mathbb{R}^3})}}$ small and $({I_a} \circ \Phi )(0,{u_0}) < 0$, i.e. ${I_a} \circ \Phi $ possesses the Mountain-Pass geometry in $\mathbb{R} \times {H^1}({\mathbb{R}^3})$. Hence we can define the Mountain-Pass level of ${I_a} \circ \Phi $:
\begin{equation}\label{3.6}
{{\tilde c}_a}: = \mathop {\inf }\limits_{\tilde \gamma  \in {{\tilde \Gamma }_a}} \mathop {\sup }\limits_{t \in [0,1]} ({I_a} \circ \Phi )(\tilde \gamma (t)),
\end{equation}
where the set of paths is defined as
\begin{equation}\label{3.7}
{{\tilde \Gamma }_a}: = \left\{ {\tilde \gamma  \in C([0,1],\mathbb{R} \times {H^1}({\mathbb{R}^3})):\tilde \gamma (0) = (0,0){\text{ and }}({I_a} \circ \Phi )(\tilde \gamma (1)) < 0} \right\}.
\end{equation}
As ${\Gamma _a} = \{ {\Phi  \circ \tilde \gamma :\tilde \gamma  \in {{\tilde \Gamma }_a}} \}$, the Mountain-Pass levels of ${I_a}$ and ${I_a} \circ \Phi $ coincide, i.e. ${c_a} = {{\tilde c}_a}$.

By Lemma~\ref{2.3.}, we see that there exists a sequence ${\{ ({\theta _n},{v_n})\} _{n \in \mathbb{N}}}$ in $\mathbb{R} \times {H^1}({\mathbb{R}^3})$ such that as $n \to \infty $,
\begin{equation}\label{3.8}
({I_a} \circ \Phi )({\theta _n},{v_n}) \to {c_a},
\end{equation}
\begin{equation}\label{3.9}
({I_a} \circ \Phi )'({\theta _n},{v_n}) \to 0{\text{ in (}}\mathbb{R} \times {H^1}({\mathbb{R}^3}){)^{ - 1}},
\end{equation}
\begin{equation}\label{3.10}
{\theta _n} \to 0.
\end{equation}

Indeed, set $\varepsilon  = {\varepsilon _n}: = \frac{1}
{{{n^2}}}$, $\delta  = {\delta _n}: = \frac{1}
{n}$ in Lemma~\ref{2.3.}, \eqref{3.8}, \eqref{3.9} are direct conclusions from $(a)$, $(c)$ of Lemma~\ref{2.3.}, we just need to verify \eqref{3.10}. In view of \eqref{3.4}, \eqref{3.5}, for $\varepsilon  = {\varepsilon _n}: = \frac{1}
{{{n^2}}}$, $\exists {\gamma _n} \in {\Gamma_a} $, such that
\[
\mathop {\sup }\limits_{t \in [0,1]} {I_a}({\gamma _n}(t)) \le {c_a} + \frac{1}
{{{n^2}}}.
\]
Set ${{\tilde \gamma }_n}(t) = (0,{\gamma _n}(t))$, then
\[
\mathop {\sup }\limits_{t \in [0,1]} {I_a} \circ \Phi ({{\tilde \gamma }_n}(t)) = \mathop {\sup }\limits_{t \in [0,1]} {I_a}({\gamma _n}(t)) \le {c_a} + \frac{1}
{{{n^2}}}.
\]
By $(b)$ of Lemma~\ref{2.3.}, there exists $({\theta _n},{v_n}) \in \mathbb{R} \times {H^1}({\mathbb{R}^3})$ such that ${\text{dist}}(({\theta _n},{v_n}),(0,{\gamma _n}(t))) \le \frac{2}
{n}$, then \eqref{3.10} holds.

For every $(h,w) \in \mathbb{R} \times {H^1}({\mathbb{R}^3})$,
\begin{equation}\label{3.11}
\left\langle {({I_a} \circ \Phi )'({\theta _n},{v_n}),(h,w)} \right\rangle  = \left\langle {{{I'}_a}(\Phi ({\theta _n},{v_n})),\Phi ({\theta _n},w)} \right\rangle  + {G_a}(\Phi ({\theta _n},{v_n}))h.
\end{equation}
Taking $h=1$, $w=0$ in \eqref{3.11}, we get
\[
{G_a}(\Phi ({\theta _n},{v_n})) \to 0{\text{ as }}n \to \infty .
\]
Denote ${u_n}: = \Phi ({\theta _n},{v_n})$, we have
\[
{G_a}({u_n}) \to 0{\text{ as }}n \to \infty .
\]

For any $v \in {H^1}({\mathbb{R}^3})$, set $w(x) = {e^{ - 2{\theta _n}}}v({e^{ - {\theta _n}}}x)$, $h=0$ in \eqref{3.11}, we get
\[
\left\langle {{{I'}_a}({u_n}),v} \right\rangle  = o(1){\left\| {{e^{ - 2{\theta _n}}}v({e^{ - {\theta _n}}}x)} \right\|_{{H^1}({\mathbb{R}^3})}} = o(1){\left\| v \right\|_{{H^1}({\mathbb{R}^3})}}
\]
for ${\theta _n} \to 0$ as $n \to \infty $, i.e. ${{I'}_a}({u_n}) \to 0$ in ${({H^1}({\mathbb{R}^3}))^{ - 1}}$ as $n \to \infty $.

Hence, we have got a bounded sequence $\{ {u_n}\} _{n = 1}^\infty \subset {H^1}({\mathbb{R}^3})$ that satisfies \eqref{3.12}.
\end{proof}

Moreover, using the same argument as in \cite{r}, we can prove
\begin{equation}\label{3.13}
{c_a} = \mathop {\inf }\limits_{u \in {H^1}({\mathbb{R}^3})\backslash \{ 0\} } \mathop {\max }\limits_{t > 0} {I_a}({u_t}) = \mathop {\inf }\limits_{u \in {M_a}} {I_a}(u) > 0.
\end{equation}

For the Mountain-Pass level $c_a$ for $I_a$, we have the following estimate:

\begin{lemma}\label{3.4.}
\[{c_a} < \frac{1}
{3}{S^{\frac{3}
{2}}}\]
for $\lambda  > 0$ large, where $S$ is the best Sobolev constant for the embedding ${D^{1,2}}({\mathbb{R}^3}) \hookrightarrow {L^6}({\mathbb{R}^3})$.
\end{lemma}

\begin{proof}
Let $\varphi  \in C_c^\infty ({B_2}(0))$ satisfying $\varphi\equiv 1 $ on ${B_1}(0)$ and $0 \le \varphi  \le 1$ on ${B_{2}}(0)$. Given $\delta > 0$, we set
${\psi _\delta }(x): = \varphi (x){w_\delta }(x)$, where
\[{w_\delta }( x ) = {( {3\delta } )^{\frac{1}
{4}}}\frac{1} {{{{( {\delta  + {{| x |}^2}} )}^{\frac{1} {2}}}}}\]
satisfies
\begin{equation}\label{3.14}
\int_{{\mathbb{R}^3}} {|\nabla {w_\delta }{|^2}}  = \int_{{\mathbb{R}^3}} {|{w_\delta }{|^6}}  = {S^{\frac{3}
{2}}}.
\end{equation}
We see that
\begin{equation}\label{3.15}
\int_{{\mathbb{R}^3}\backslash {B_1}(0)} {|\nabla {\psi _\delta }{|^2}}  = O({\delta ^{1/2}}){\text{ as }}\delta  \to 0.
\end{equation}
Let ${X_\delta }: = \int_{{\mathbb{R}^3}} {{{| {\nabla {v_\delta }} |}^2}}
$, where ${v_\delta }: = {\psi _\delta }/{(\int_{{B_{2}}(0)}
{|{\psi _\delta }{|^6}} )^{\frac{1} {6}}}$. We find
\begin{equation}\label{3.16}
{X_\delta } \le S + O( {{\delta ^{1/2}}} ){\text{ as }}\delta  \to 0.
\end{equation}

In view of Lemma~\ref{3.2.}, there exists ${t_\delta } > 0$ such that $\mathop {\sup }\limits_{t \ge 0} {I_a}({({v_\delta })_t}) = {I_a}({({v_\delta })_{{t_\delta }}})$. Hence
$\frac{{d{I_a}({{({v_\delta })}_t})}}
{{dt}}\left| {_{t = {t_\delta }}} \right. = 0$, that is
\[\frac{3}
{2}t_\delta ^2\int_{{\mathbb{R}^3}} {|\nabla {v_\delta }{|^2}}  + \frac{1}
{2}a\int_{{\mathbb{R}^3}} {v_\delta ^2}  + \frac{3}
{4}t_\delta ^2\int_{{\mathbb{R}^3}} {{\phi _{{v_\delta }}}v_\delta ^2}  - \frac{{(2p - 3)}}
{p}\lambda t_\delta ^{2p - 5}\int_{{\mathbb{R}^3}} {v_\delta ^p}  - \frac{3}
{2}t_\delta ^8\int_{{\mathbb{R}^3}} {v_\delta ^6}  = 0\]
which implies
\begin{equation}\label{3.17}
t_\delta ^8 \le t_\delta ^2{X_\delta } + \frac{1}
{3}a\int_{{\mathbb{R}^3}} {v_\delta ^2}  + \frac{1}
{2}t_\delta ^2\int_{{\mathbb{R}^3}} {{\phi _{{v_\delta }}}v_\delta ^2} .
\end{equation}
Direct calculations show that
\begin{equation}\label{3.18}
\int_{{\mathbb{R}^3}} {v_\delta ^2}  = O({\delta ^{1/2}}),{\text{ }}{\Bigl( {\int_{{\mathbb{R}^3}} {v_\delta ^{12/5}} } \Bigr)^{5/3}} = O(\delta ).
\end{equation}
\eqref{3.16}, \eqref{3.17}, \eqref{3.18} and Lemma~\ref{2.1.} (i) imply that $|{t_\delta }| \le {C_1}$, where ${C_1}$ is independent of $\delta  > 0$ small.

We can assume that there is a positive constant ${C_2}$ such that ${t_\delta } \ge {C_2} > 0$ for $\delta  > 0$ small. Otherwise, we could find a sequence ${\delta _n} \to 0$ as $n \to \infty $ such that ${t_{{\delta _n}}} \to 0$ as $n \to \infty $. Now, up to a subsequence, we have
${({v_{{\delta _n}}})_{{t_{{\delta _n}}}}} \to 0$ in ${H^1}({\mathbb{R}^3})$ as $n \to \infty $. Therefore
\[
0 < {c_a} \le \mathop {\sup }\limits_{t \ge 0} {I_a}({({v_{{\delta _n}}})_t}) = {I_a}({({v_{{\delta _n}}})_{{t_{{\delta _n}}}}}) \to {I_a}(0) = 0,
\]
which is a contradiction.

Denote $g(t) = \frac{{{t^3}}}
{2}\int_{{\mathbb{R}^3}} {|\nabla {v_\delta }{|^2}}  - \frac{{{t^9}}}
{6}\int_{{\mathbb{R}^3}} {v_\delta ^6}$, it is easy to check that
\[
\mathop {\sup }\limits_{t > 0} g(t) = \frac{1}
{3}{\Bigl( {\int_{{\mathbb{R}^3}} {|\nabla {v_\delta }{|^2}} } \Bigr)^{\frac{3}
{2}}} \le \frac{1}
{3}{\Bigl( {S + O({\delta ^{1/2}})} \Bigr)^{3/2}} \le \frac{1}
{3}{S^{\frac{3}
{2}}} + O({\delta ^{1/2}}).
\]
Thus
\begin{equation}\label{3.19}
\begin{gathered}
  {\text{  }}{I}({({v_\delta })_{{t_\delta }}}) \hfill \\
   = \frac{1}
{2}t_\delta ^3\int_{{\mathbb{R}^3}} {|\nabla {v_\delta }{|^2}}  + \frac{1}
{2}{t_\delta }\int_{{\mathbb{R}^3}} {v_\delta ^2}  + \frac{1}
{4}t_\delta ^3\int_{{\mathbb{R}^3}} {{\phi _{{v_\delta }}}v_\delta ^2}  - \frac{\lambda }
{p}t_\delta ^{2p - 3}\int_{{\mathbb{R}^3}} {v_\delta ^p}  - \frac{1}
{6}t_\delta ^9\int_{{\mathbb{R}^3}} {v_\delta ^6}  \hfill \\
   \le \mathop {\sup }\limits_{t > 0} g(t) + C\int_{{\mathbb{R}^3}} {v_\delta ^2}  + C{\left( {\int_{{\mathbb{R}^3}} {v_\delta ^{12/5}} } \right)^{5/3}} - C\lambda \int_{{\mathbb{R}^3}} {v_\delta ^p}  \hfill \\
   \le \frac{1}
{3}{S^{\frac{3}
{2}}} + O({\delta ^{1/2}}) + C\int_{{\mathbb{R}^3}} {v_\delta ^2}  - C\lambda \int_{{\mathbb{R}^3}} {v_\delta ^p} , \hfill \\
\end{gathered}
\end{equation}
where we have used \eqref{3.18}.

From \eqref{3.19}, to complete the proof,  it suffices to show that
\begin{equation}\label{3.20}
\mathop {\lim }\limits_{\delta  \to {0^ + }} \frac{1}
{{{\delta ^{1/2}}}}\Bigl[ {C\int_{{B_1}(0)} {v_\delta ^2} - C\lambda \int_{{B_1}(0)} {v_\delta ^p} } \Bigr] =  - \infty
\end{equation}
and
\begin{equation}\label{3.21}
\mathop {\lim }\limits_{\delta  \to {0^ + }} \frac{1}
{{{\delta ^{1/2}}}}\Bigl[ {C\int_{{B_2}(0)\backslash {B_1}(0)} {v_\delta ^2} - C\lambda \int_{{B_2}(0)\backslash {B_1}(0)} {v_\delta ^p} } \Bigr] \le  C.
\end{equation}
To this end, we find
\[\begin{gathered}
  \frac{1}
{{{\delta ^{1/2}}}}C\lambda \int_{{B_1}(0)} {v_\delta ^p}  \ge \frac{{C\lambda }}
{{{\delta ^{1/2}}}}\int_{{B_1}(0)} {\frac{{{\delta ^{p/4}}}}
{{{{(\delta  + |x{|^2})}^{p/2}}}}}  \hfill \\
  {\text{ }}\mathop  \ge \limits^{x' = x/{\delta ^{1/2}}} \frac{{C\lambda }}
{{{\delta ^{\frac{1}
{2}}}}}\int_{{B_{1/{\delta ^{1/2}}}}(0)} {\frac{{{\delta ^{\frac{p}
{4}}}}}
{{{{(\delta  + \delta |x'{|^2})}^{\frac{p}
{2}}}}}{\delta ^{\frac{3}
{2}}}}  \ge C\lambda {\delta ^{1 - \frac{p}
{4}}}\int_{{B_{1/{\delta ^{1/2}}}}(0)} {\frac{1}
{{{{(1 + |x'{|^2})}^{p/2}}}}} . \hfill \\
\end{gathered} \]
Since $p \in (3,4]$, choosing $\lambda  = 1/\delta$ and combining with \eqref{3.18}, \eqref{3.20} holds.

Since
\[
\frac{1}
{{{\delta ^{1/2}}}}\Bigl[ {\int_{{B_2}(0)\backslash {B_1}(0)} {v_\delta ^2}  - C\lambda \int_{{B_2}(0)\backslash {B_1}(0)} {v_\delta ^p} } \Bigr] \le \frac{C}
{{{\delta ^{1/2}}}}\int_{{B_2}(0)\backslash {B_1}(0)} {v_\delta ^2}  \le C,
\]
where we have used \eqref{3.18}, then \eqref{3.21} holds.
\end{proof}

\begin{lemma}\label{add4}
Every sequence $\{ {u_n}\} _{n = 1}^\infty $ satisfying \eqref{3.12} is bounded in ${H^1}({\mathbb{R}^3})$.
\end{lemma}
\begin{proof}
By \eqref{3.12}, we have
\[
\begin{gathered}
  {c_a} + o(1) = {I_a}({u_n}) - \frac{1}
{{2p - 3}}{G_a}({u_n}) \hfill \\
  {\text{ }} = \frac{{p - 3}}
{{2p - 3}}\int_{{\mathbb{R}^3}} {|\nabla {u_n}{|^2}}  + \frac{{p - 2}}
{{2p - 3}}a\int_{{\mathbb{R}^3}} {|{u_n}{|^2}}  + \frac{{p - 3}}
{{2(2p - 3)}}\int_{{\mathbb{R}^3}} {{\phi _{{u_n}}}u_n^2}  + \frac{{6 - p}}
{{3(2p - 3)}}\int_{{\mathbb{R}^3}} {{{(u_n^ + )}^6}} , \hfill \\
\end{gathered}
\]
we get the upper bound of ${\left\| {{u_n}} \right\|_{{H^1}({\mathbb{R}^3})}}$.
\end{proof}

\begin{lemma}\label{3.5.}
There is a sequence $\{ {x_n}\}  \subset {\mathbb{R}^3}$ and $R > 0$, $\beta  > 0$ such that
\[
\int_{{B_R}({x_n})} {u_n^2}  \ge \beta ,
\]
where $\{ {u_n}\} $ is the sequence given in \eqref{3.12}.
\end{lemma}

\begin{proof}
Assume the contrary that the lemma does not hold. By the Vanishing Theorem (Lemma 1.1 of \cite{l2}), it follows that as $n \to \infty $,
\[
\int_{{\mathbb{R}^3}} {|{u_n}{|^s}}  \to 0{\text{ for all }}2 < s < 6{\text{ and }}\int_{{\mathbb{R}^3}} {{\phi _{{u_n}}}u_n^2}  \to 0.
\]
Using $\left\langle {{{I'}_a}({u_n}),{u_n}} \right\rangle  = o(1)$, we get
\[
\int_{{\mathbb{R}^3}} {|\nabla {u_n}{|^2}}  + a\int_{{\mathbb{R}^3}} {u_n^2}  - \int_{{\mathbb{R}^3}} {{{(u_n^ + )}^6}}  = o(1).
\]
By ${I_a}({u_n}) \to {c_a}$, we have
\begin{equation}\label{3.22}
\frac{1}
{2}\int_{{\mathbb{R}^3}} {|\nabla {u_n}{|^2}}  + \frac{1}
{2}a\int_{{\mathbb{R}^3}} {u_n^2}  - \frac{1}
{6}\int_{{\mathbb{R}^3}} {{{(u_n^ + )}^6}}  = {c_a} + o(1).
\end{equation}
Let $l \ge 0$ be such that
\begin{equation}\label{3.23}
\int_{{\mathbb{R}^3}} {|\nabla {u_n}{|^2}}  + a\int_{{\mathbb{R}^3}} {u_n^2}  \to l
\end{equation}
and
\begin{equation}\label{3.24}
\int_{{\mathbb{R}^3}} {{{(u_n^ + )}^6}}  \to l.
\end{equation}
It is easy to check that $l > 0$, otherwise ${\left\| {{u_n}} \right\|_{{H^1}({\mathbb{R}^3})}} \to 0$ as $n \to \infty $ which contradicts to ${c_a} > 0$. From \eqref{3.22}, \eqref{3.23}, \eqref{3.24}, we get ${c_a} = \frac{1}
{3}l$.

Now, using the definition of the constant $S$, we have
\[
\int_{{\mathbb{R}^3}} {|\nabla {u_n}{|^2}}  + \int_{{\mathbb{R}^3}} {u_n^2}  \ge S{\Bigl( {\int_{{\mathbb{R}^3}} {{{(u_n^ + )}^6}} } \Bigr)^{\frac{1}
{3}}}.
\]
Letting $n \to \infty $ in the above inequality, we achieve that $l \ge {S^{3/2}}$. Hence
\[
{c_a} = \frac{1}
{3}l \ge \frac{1}
{3}{S^{\frac{3}
{2}}},
\]
which contradicts to Lemma~\ref{3.4.}.
\end{proof}

We have the following proposition:
\begin{proposition}\label{3.6.}
\eqref{3.1} has a positive ground state solution $\tilde u \in {H^1}({\mathbb{R}^3})$.
\end{proposition}

\begin{proof}
Let $\{ {u_n}\} $ be the sequence given in \eqref{3.12} and $c_a$ be the Mountain-Pass value for $I_a$ respectively. Denote ${\tilde u_n}(x) = {u_n}(x + {x_n})$, where  $\{{x_n}\}$ is the sequence given in Lemma~\ref{3.5.}. Using standard argument, up to a subsequence, we may assume that there is a $\tilde u \in {H^1}({\mathbb{R}^3})$ such that
\begin{equation}\label{3.25}
\left\{ \begin{gathered}
  {{\tilde u}_n} \rightharpoonup \tilde u{\text{ in }}{H^1}({\mathbb{R}^3}), \hfill \\
  {{\tilde u}_n} \to \tilde u{\text{ in }}L_{{\text{loc}}}^s({\mathbb{R}^3}){\text{ for all }}1 \le s < 6, \hfill \\
  {{\tilde u}_n} \to \tilde u{\text{ a}}{\text{.e}}{\text{. in }}{\mathbb{R}^3}.{\text{ }} \hfill \\
\end{gathered}  \right.
\end{equation}
By Lemma~\ref{3.5.}, ${\tilde u}$ is nontrivial. Moreover, ${\tilde u}$ satisfies
\begin{equation}\label{3.26}
 - \Delta u + au + {\phi _u}u = \lambda {({u^ + })^{p - 1}} + {({u^ + })^5}{\text{ in }}{\mathbb{R}^3}
\end{equation}
and ${G_a}(\tilde u) = 0$. By \eqref{3.13}, we have
\[\begin{gathered}
  {c_a} \le {I_a}(\tilde u) = {I_a}(\tilde u) - \frac{1}
{3}{G_a}(\tilde u) = \frac{1}
{3}a\int_{{\mathbb{R}^3}} {{{\tilde u}^2}}  + \frac{{2p - 6}}
{{3p}}\lambda\int_{{\mathbb{R}^3}} {{{({{\tilde u}^ + })}^p}}  + \frac{1}
{3}\int_{{\mathbb{R}^3}} {{{({{\tilde u}^ + })}^6}}  \hfill \\
   \le \mathop {\underline {\lim } }\limits_{n \to \infty } \frac{1}
{3}a\int_{{\mathbb{R}^3}} {\tilde u_n^2}  + \frac{{2p - 6}}
{{3p}}\lambda\int_{{\mathbb{R}^3}} {{{(\tilde u_n^ + )}^p}}  + \frac{1}
{3}\int_{{\mathbb{R}^3}} {{{(\tilde u_n^ + )}^6}}  = \mathop {\underline {\lim } }\limits_{n \to \infty } \Bigl[ {{I_a}({{\tilde u}_n}) - \frac{1}
{3}{G_a}({{\tilde u}_n})} \Bigr] \hfill \\
   = \mathop {\underline {\lim } }\limits_{n \to \infty } \Bigl[ {{I_a}({u_n}) - \frac{1}
{3}{G_a}({u_n})} \Bigr] = {c_a}. \hfill \\
\end{gathered} \]
Hence ${I_a}(\tilde u) = {c_a}$ and ${{I'}_a}(\tilde u) = 0$. By the standard elliptic estimate and strong maximum principle, $\tilde u(x) > 0$ for all $x \in {\mathbb{R}^3}$. In view of \eqref{3.13}, $\tilde u$ is in fact a positive ground state solution of \eqref{3.1}.
\end{proof}

Let $S_a$ the set of ground state solutions $U$ of \eqref{3.1} satisfying $U(0) = \mathop {\max }\limits_{x \in {\mathbb{R}^3}} U(x)$. Then, we obtain the following compactness of $S_a$.

\begin{proposition}\label{3.7.}
For each $a > 0$, $S_a$ is compact in ${H^1}({\mathbb{R}^3})$.
\end{proposition}

\begin{proof}
For any $U \in {S_a}$, we have
\[\begin{gathered}
  {c_a} = {I_a}(U) - \frac{1}
{{2p - 3}}{G_a}(U) \hfill \\
  {\text{ }} = \frac{{p - 3}}
{{2p - 3}}\int_{{\mathbb{R}^3}} {|\nabla U{|^2}}  + \frac{{p - 2}}
{{2p - 3}}a\int_{{\mathbb{R}^3}} {{U^2}}  + \frac{{p - 3}}
{{2(2p - 3)}}\int_{{\mathbb{R}^3}} {{\phi _U}{U^2}}  + \frac{{6 - p}}
{{3(2p - 3)}}\int_{{\mathbb{R}^3}} {{U^6}} . \hfill \\
\end{gathered} \]
Thus $S_a$ is bounded in ${H^1}({\mathbb{R}^3})$.

For any sequence $\{ {U_k}\}  \subset {S_a}$, up to a subsequence, we may assume that there is a ${U_0} \in {H^1}({\mathbb{R}^3})$ such that
\begin{equation}\label{3.27}
{U_k} \rightharpoonup {U_0}{\text{ in }}{H^1}({\mathbb{R}^3})
\end{equation}
and ${U_0}$ satisfies
\[
- \Delta {U_0} + a{U_0} + {\phi _{{U_0}}}{U_0} = \lambda U_0^{p - 1} + U_0^5{\text{ in }}{\mathbb{R}^3},{\text{ }}{U_0} \ge 0.
\]
Next, we will show that ${U_0}$ is nontrivial. First, we claim that, up to a subsequence,
\begin{equation}\label{3.28}
{U_k} \to {U_0}{\text{ in }}L_{{\text{loc}}}^6({\mathbb{R}^3}).
\end{equation}
Indeed, in view of \eqref{3.27}, we may assume that
\[
|\nabla {U_k}{|^2} \rightharpoonup |\nabla {U_0}{|^2} + \mu {\text{ and }}U_k^6 \rightharpoonup U_0^6 + \nu ,
\]
where $\mu$ and $\nu$ are two bounded nonnegative measures on ${\mathbb{R}^3}$.  By the Concentration Compactness Principle II (Lemma 1.1 of \cite{l3}), we obtain an at most countable index set $\Gamma $, sequence $\{ {x_i}\}  \subset {\mathbb{R}^3}$ and $\{ {{\mu _i}} \},\{ {{\nu _i}} \} \subset ( {0,\infty } )$ such that
\begin{equation}\label{3.29}
\mu  \ge \sum\limits_{i \in \Gamma } {{\mu _i}} {\delta
_{{x_i}}},\nu = \sum\limits_{i \in \Gamma } {{\nu _i}} {\delta
_{{x_i}}}{\text{ and }}S{( {{\nu _i}} )^{\frac{1} {3}}} \le {\mu
_i}.
\end{equation}
It suffices to show that for any bounded domain $\Omega $, ${\{ {x_i}\} _{i \in \Gamma }} \cap \Omega  = \emptyset$. Suppose, by contradiction, that ${x_i} \in \Omega $ for some $i \in \Gamma $. Define, for $\rho  > 0$, the function ${\psi _\rho }( x ): = \psi ( {\frac{{x - {x_i}}}
{\rho }} )$ where $\psi$ is a smooth cut-off function such that $\psi  = 1$ on ${B_1}( 0 )$, $\psi  = 0$ on ${\mathbb{R}^3}\backslash {B_2}(0)$, $0 \le \psi  \le 1$ and $| {\nabla \psi } | \le C$. We suppose that $\rho $ is chosen in such a way that the support of ${\psi _\rho }$ is contained in $\Omega $. Using $\left\langle {{{I'}_a}({U_k}),{\psi _\rho }{U_k}} \right\rangle  = 0$, we see
\begin{equation}\label{3.30}
\begin{gathered}
  \int_{{\mathbb{R}^3}} {|\nabla {U_k}{|^2}{\psi _\rho }}  + \int_{{\mathbb{R}^3}} {(\nabla {U_k}\cdot\nabla {\psi _\rho }){U_k}}  + a\int_{{\mathbb{R}^3}} {U_k^2{\psi _\rho }}  + \int_{{\mathbb{R}^3}} {{\phi _{{U_k}}}U_k^2{\psi _\rho }}  \hfill \\
  {\text{ }} = \lambda \int_{{\mathbb{R}^3}} {U_k^p{\psi _\rho }}  + \int_{{\mathbb{R}^3}} {U_k^6{\psi _\rho }} . \hfill \\
\end{gathered}
\end{equation}
Since
\begin{equation}\label{3.31}
\begin{gathered}
  \mathop {\overline {\lim } }\limits_{k \to \infty } \left| {\int_{{\mathbb{R}^3}} {(\nabla {U_k}\cdot\nabla {\psi _\rho }){U_k}} } \right| \le \mathop {\overline {\lim } }\limits_{k \to \infty } {\Bigl( {\int_{{\mathbb{R}^3}} {|\nabla {U_k}{|^2}} } \Bigr)^{\frac{1}
{2}}}\cdot{\Bigl( {\int_{{\mathbb{R}^3}} {U_k^2|\nabla {\psi _\rho }{|^2}} } \Bigr)^{\frac{1}
{2}}} \hfill \\
   \le C{\Bigl( {\int_{{\mathbb{R}^3}} {U_0^2|\nabla {\psi _\rho }{|^2}} } \Bigr)^{\frac{1}
{2}}} \le C{\Bigl( {\int_{{B_{2\rho }}({x_i})} {U_0^6} } \Bigr)^{\frac{1}
{6}}}{\Bigl( {\int_{{B_{2\rho }}({x_i})} {|\nabla {\psi _\rho }{|^3}} } \Bigr)^{\frac{1}
{3}}} \hfill \\
   \le C{\Bigl( {\int_{{B_{2\rho }}({x_i})} {U_0^6} } \Bigr)^{\frac{1}
{6}}} \to 0{\text{ as }}\rho  \to 0, \hfill \\
\end{gathered}
\end{equation}
\begin{equation}\label{3.32}
\mathop {\overline {\lim } }\limits_{k \to \infty } \int_{{\mathbb{R}^3}} {|\nabla {U_k}{|^2}{\psi _\rho }}  \ge \int_{{\mathbb{R}^3}} {|\nabla {U_0}{|^2}{\psi _\rho }}  + {\mu _i} \to {\mu _i}{\text{ as }}\rho  \to 0,
\end{equation}
\begin{equation}\label{3.33}
\mathop {\overline {\lim } }\limits_{k \to \infty } \lambda \int_{{\mathbb{R}^3}} {U_k^p{\psi _\rho }}  = \lambda \int_{{\mathbb{R}^3}} {U_0^p{\psi _\rho }}  \to 0{\text{ as }}\rho  \to 0,
\end{equation}
and
\begin{equation}\label{3.34}
\mathop {\overline {\lim } }\limits_{k \to \infty } \int_{{\mathbb{R}^3}} {U_k^6{\psi _\rho }}  = \int_{{\mathbb{R}^3}} {U_0^6{\psi _\rho }}  + {\nu _i} \to {\nu _i}{\text{ as }}\rho  \to 0.
\end{equation}
We obtain from \eqref{3.30} that ${\mu _i} \le {\nu _i}$. Combining with \eqref{3.29}, we have ${\nu _i} \ge {S^{3/2}}$. On the other hand,
\[
{c_a} = {I_a}({U_k}) - \frac{1}
{3}{G_a}({U_k}) = \frac{1}
{3}a\int_{{\mathbb{R}^3}} {U_k^2}  + \frac{{2p - 6}}
{{3p}}\int_{{\mathbb{R}^3}} {U_k^p}  + \frac{1}
{3}\int_{{\mathbb{R}^3}} {U_k^6}  \ge \frac{1}
{3}{\nu _i} \ge \frac{1}
{3}{S^{\frac{3}
{2}}},\]
which contradicts to Lemma~\ref{3.4.}, then \eqref{3.28} holds.

From \eqref{3.28}, $\{ U_k^6\} $ is uniformly integrable in any bounded domain in ${{\mathbb{R}^3}}$. By Lemma~\ref{2.4.} (i), ${\left\| {{U_k}} \right\|_{L_{{\text{loc}}}^\infty ({\mathbb{R}^3})}} \le C$. In view of \cite{t}, $\exists \alpha  \in (0,1)$ such that ${\left\| {{U_k}} \right\|_{C_{{\text{loc}}}^{1,\alpha }({\mathbb{R}^3})}} \le C$, and using Schauder's estimate, we have
\[
{\left\| {{U_k}} \right\|_{C_{{\text{loc}}}^{2,\alpha }({\mathbb{R}^3})}} \le C.
\]
By the Arzela-Ascoli's Theorem, we have
\[
{U_k}(0) \to {U_0}(0){\text{ as }}k \to \infty .
\]
Since $\Delta {U_k}(0) \le 0$, from \eqref{3.1}, we can check that $\exists b > 0$ such that ${U_k}(0) \ge b > 0$, then $ {U_0}(0) \ge b > 0$, this means that ${U_0}$ is nontrivial.

Since
\[\begin{gathered}
  {c_a} \le {I_a}({U_0}) - \frac{1}
{{2p - 3}}{G_a}({U_0}) \hfill \\
  {\text{ }} = \frac{{p - 3}}
{{2p - 3}}\int_{{\mathbb{R}^3}} {|\nabla {U_0}{|^2}}  + \frac{{p - 2}}
{{2p - 3}}a\int_{{\mathbb{R}^3}} {U_0^2}  + \frac{{p - 3}}
{{2(2p - 3)}}\int_{{\mathbb{R}^3}} {{\phi _{{U_0}}}U_0^2}  + \frac{{6 - p}}
{{3(2p - 3)}}\int_{{\mathbb{R}^3}} {U_0^6}  \hfill \\
  {\text{ }} = \mathop {\underline {\lim } }\limits_{k \to \infty } \frac{{p - 3}}
{{2p - 3}}\int_{{\mathbb{R}^3}} {|\nabla {U_k}{|^2}}  + \frac{{p - 2}}
{{2p - 3}}a\int_{{\mathbb{R}^3}} {U_k^2}  + \frac{{p - 3}}
{{2(2p - 3)}}\int_{{\mathbb{R}^3}} {{\phi _{{U_k}}}U_k^2}  + \frac{{6 - p}}
{{3(2p - 3)}}\int_{{\mathbb{R}^3}} {U_k^6}  \hfill \\
  {\text{ }} =\mathop {\underline {\lim } }\limits_{k \to \infty }\Bigl[ {I_a}({U_k}) - \frac{1}
{{2p - 3}}{G_a}({U_k})\Bigr] = {c_a}, \hfill \\
\end{gathered} \]
which means that ${I_a}({U_0}) = {c_a}$ and ${U_k} \to {U_0}$ in ${H^1}({\mathbb{R}^3})$. This completes the proof that $S_a$ is compact in ${H^1}({\mathbb{R}^3})$.
\end{proof}

\section{Proof of Theorem~\ref{1.1.}}
\eqref{1.1} can be rewritten as
\begin{equation}\label{4.1}
\left\{ \begin{gathered}
   - \Delta v + V(\varepsilon x)v + \phi v = \lambda |v{|^{p - 2}}v + |v{|^4}v{\text{ in }}{\mathbb{R}^3}, \hfill \\
   - \Delta \phi  = {v^2}{\text{ in }}{\mathbb{R}^3},{\text{ }}v > 0,{\text{ }}v \in {H^1}({\mathbb{R}^3}) \hfill \\
\end{gathered}  \right.
\end{equation}
and the corresponding energy functional is
\[
{I_\varepsilon }(v) = \frac{1}
{2}\int_{{\mathbb{R}^3}} {|\nabla v{|^2}}  + \frac{1}
{2}\int_{{\mathbb{R}^3}} {V(\varepsilon x){v^2}}  + \frac{1}
{4}\int_{{\mathbb{R}^3}} {{\phi _v}{v^2}}  - \frac{1}
{p}\lambda \int_{{\mathbb{R}^3}} {{{({v^ + })}^p}}  - \frac{1}
{6} \int_{{\mathbb{R}^3}} {{{({v^ + })}^6}} ,{\text{ }}v \in {H_\varepsilon },
\]
where ${H_\varepsilon }: = \{ {v \in {H^1}({\mathbb{R}^3})| {\int_{{\mathbb{R}^3}} {V(\varepsilon x){v^2}}  < \infty } } \}$ endowed with the norm
\[
{\left\| v \right\|_{{H_\varepsilon }}}: = {\Bigl( {\int_{{\mathbb{R}^3}} {|\nabla v{|^2}}  + \int_{{\mathbb{R}^3}} {V(\varepsilon x){v^2}} } \Bigr)^{1/2}}.
\]
We define
\[
{\chi _\varepsilon }(x) = \left\{ \begin{gathered}
  0{\text{ if }}x \in \Lambda /\varepsilon , \hfill \\
  {\varepsilon ^{ - 1}}{\text{ if }}x \notin \Lambda /\varepsilon  \hfill \\
\end{gathered}  \right.
\]
and
\[
{Q_\varepsilon }(v) = \Bigl( {\int_{{\mathbb{R}^3}} {{\chi _\varepsilon }{v^2}}  - 1} \Bigr)_ + ^2.
\]
Finally, set ${J_\varepsilon }:{H_\varepsilon } \to \mathbb{R}$ be given by
\[
{J_\varepsilon }(v) = {I_\varepsilon }(v) + {Q_\varepsilon }(v).
\]
Note that this type of penalization was firstly introduced in \cite{bw}. It is standard to show that ${J_\varepsilon } \in {C^1}({H_\varepsilon },\mathbb{R})$. To find solutions of \eqref{4.1} which concentrate around the local minimum of $V$ in $\Lambda $ as $\varepsilon  \to 0$, we shall search critical points of ${J_\varepsilon }$ for which ${Q_\varepsilon }$ is zero.

Let ${c_{{V_0}}} = {I_{{V_0}}}(w)$ for $w \in {S_{{V_0}}}$ and $10\delta  = {\text{dist}}\{ \mathcal{M},{\mathbb{R}^3}\backslash \Lambda \} $, we fix a $\beta  \in (0,\delta )$ and a cut-off function $\varphi  \in C_c^\infty ({\mathbb{R}^3})$ such that $0 \le \varphi  \le 1$, $\varphi (x) = 1$ for $|x| \le \beta $, $\varphi (x) = 0$ for $|x| \ge 2\beta $ and $|\nabla \varphi | \le C/\beta $. We will find a solution of \eqref{4.1} near the set
\[
{X_\varepsilon }: = \Bigl\{ {\varphi (\varepsilon x - x')w\bigl( {x - \frac{{x'}}
{\varepsilon }} \bigr):x' \in {\mathcal{M}^\beta },w \in {S_{{V_0}}}} \Bigr\}
\]
for sufficiently small $\varepsilon  > 0$, where ${\mathcal{M}^\beta }: = \{ y \in {\mathbb{R}^3}:\mathop {\inf }\limits_{z \in \mathcal{M}} |y - z| \le \beta \} $. Similarly, for $A \subset {H_\varepsilon }$, we use the notation
\[
{A^a}: = \bigl\{ {u \in {H_\varepsilon }:\mathop {\inf }\limits_{v \in A} {{\left\| {u - v} \right\|}_{{H_\varepsilon }}} \le a} \bigr\}.
\]
For ${U^ * } \in {S_{{V_0}}}$ arbitrary but fixed, we define ${W_{\varepsilon ,t}}(x): = {t^2}\varphi (\varepsilon x){U^ * }(tx)$, we will show that ${J_\varepsilon }$ possesses the Mountain-Pass geometry.

Denote $U_t^ * : = {t^2}{U^ * }(tx)$, we have
\[\begin{gathered}
  {\text{  }}{I_{{V_0}}}(U_t^ * ) \hfill \\
   = \frac{1}
{2}{t^3}\int_{{\mathbb{R}^3}} {|\nabla {U^ * }{|^2}}  + \frac{1}
{2}{V_0}t\int_{{\mathbb{R}^3}} {{{({U^ * })}^2}}  + \frac{1}
{4}{t^3}\int_{{\mathbb{R}^3}} {{\phi _{{U^ * }}}{{({U^ * })}^2}}  \hfill \\
  {\text{  }} - \frac{1}
{p}\lambda {t^{2p - 3}}\int_{{\mathbb{R}^3}} {{{({U^ * })}^p}}  - \frac{1}
{6}{t^9}\int_{{\mathbb{R}^3}} {{{({U^ * })}^6}}  \hfill \\
   \to  - \infty {\text{ as }}t \to \infty , \hfill \\
\end{gathered} \]
then $\exists {t_0} > 0$ such that ${I_{{V_0}}}(U_{{t_0}}^ * ) <  - 3$.

We can easily check that ${Q_\varepsilon }({W_{\varepsilon ,{t_0}}}) = 0$, then
\begin{equation}\label{4.2}
\begin{gathered}
  {\text{  }}{J_\varepsilon }({W_{\varepsilon ,{t_0}}}) \hfill \\
   = {I_\varepsilon }({W_{\varepsilon ,{t_0}}}) \hfill \\
   = \frac{1}
{2}\int_{{\mathbb{R}^3}} {|\nabla {W_{\varepsilon ,{t_0}}}{|^2}}  + \frac{1}
{2}\int_{{\mathbb{R}^3}} {V(\varepsilon x)W_{\varepsilon ,{t_0}}^2}  + \frac{1}
{4}\int_{{\mathbb{R}^3}} {{\phi _{{W_{\varepsilon ,{t_0}}}}}W_{\varepsilon ,{t_0}}^2}  - \frac{1}
{p}\lambda \int_{{\mathbb{R}^3}} {W_{\varepsilon ,{t_0}}^p}  - \frac{1}
{6}\int_{{\mathbb{R}^3}} {W_{\varepsilon ,{t_0}}^6}  \hfill \\
  \mathop  = \limits^{\tilde x = {t_0}x} \frac{1}
{2}t_0^3\int_{{\mathbb{R}^3}} {{{\Bigl| {\frac{\varepsilon }
{{{t_0}}}\nabla \varphi \bigl( {\frac{\varepsilon }
{{{t_0}}}\tilde x} \bigr){U^ * }(\tilde x) + \varphi \bigl( {\frac{\varepsilon }
{{{t_0}}}\tilde x} \bigr)\nabla {U^ * }(\tilde x)} \Bigr|}^2}} d\tilde x \hfill \\
   + \frac{1}
{2}{t_0}\int_{{\mathbb{R}^3}} {V\bigl( {\frac{\varepsilon }
{{{t_0}}}\tilde x} \bigr){\varphi ^2}\bigl( {\frac{\varepsilon }
{{{t_0}}}\tilde x} \bigr)({U^ * }(\tilde x)} {)^2} + \frac{1}
{4}t_0^3\int_{{\mathbb{R}^3}} {{\phi _{\varphi ( {\frac{\varepsilon }
{{{t_0}}}\tilde x} ){U^ * }(\tilde x)}}{\varphi ^2}\bigl( {\frac{\varepsilon }
{{{t_0}}}\tilde x} \bigr){{({U^ * }(\tilde x))}^2}}  \hfill \\
   - \frac{1}
{p}\lambda t_0^{2p - 3}\int_{{\mathbb{R}^3}} {{\varphi ^p}\bigl( {\frac{\varepsilon }
{{{t_0}}}\tilde x} \bigr){{({U^ * }(\tilde x))}^p}}  - \frac{1}
{6}t_0^9\int_{{\mathbb{R}^3}} {{\varphi ^6}\bigl( {\frac{\varepsilon }
{{{t_0}}}\tilde x} \bigr){{({U^ * }(\tilde x))}^6}}  \hfill \\
   = {I_{{V_0}}}(U_{{t_0}}^ * ) + o(1) <  - 2{\text{ for }}\varepsilon  > 0{\text{ small,}} \hfill \\
\end{gathered}
\end{equation}
where we have used the Dominated Convergence Theorem and Lemma~\ref{2.2.} (i).

Using the Sobolev's Imbedding Theorem, we have
\[\begin{gathered}
  {\text{  }}{J_\varepsilon }(u) \hfill \\
   \ge {I_\varepsilon }(u) \hfill \\
   \ge \frac{1}
{2}\left\| u \right\|_{{H_\varepsilon }}^2 - \frac{1}
{p}\lambda \int_{{\mathbb{R}^3}} {|u{|^p}}  - \frac{1}
{6}\int_{{\mathbb{R}^3}} {|u{|^6}}  \hfill \\
   \ge \frac{1}
{2}\left\| u \right\|_{{H_\varepsilon }}^2 - C \cdot \lambda \left\| u \right\|_{{H_\varepsilon }}^p - C\left\| u \right\|_{{H_\varepsilon }}^6 > 0 \hfill \\
\end{gathered} \]
for ${\left\| u \right\|_{{H_\varepsilon }}}$ small since $p>2$.

Hence, we can define the Mountain-Pass value of ${J_\varepsilon }$ as follows,
\[
{c_\varepsilon }: = \mathop {\inf }\limits_{\gamma  \in {\Gamma _\varepsilon }} \mathop {\max }\limits_{s \in [0,1]} {J_\varepsilon }(\gamma (s))
\]
where ${\Gamma _\varepsilon }: = \{ \gamma  \in C([0,1],{H_\varepsilon })|\gamma (0) = 0,{\text{ }}\gamma (1) = {W_{\varepsilon ,{t_0}}}\} $.

\begin{lemma}\label{4.1.}
\[
\overline {\mathop {\lim }\limits_{\varepsilon  \to 0} } {c_\varepsilon } \le {c_{{V_0}}}.
\]
\end{lemma}
\begin{proof}
Denote ${W_{\varepsilon ,0}} = \mathop {\lim }\limits_{t \to 0} {W_{\varepsilon ,t}}$ in ${H_\varepsilon }$ sense, then ${W_{\varepsilon ,0}} = 0$. Thus, setting $\gamma (s): = {W_{\varepsilon ,s{t_0}}}(0 \le s \le 1)$, we have $\gamma (s) \in {\Gamma _\varepsilon }$, then
\[
{c_\varepsilon } \le \mathop {\max }\limits_{s \in [0,1]} {J_\varepsilon }(\gamma (s)) = \mathop {\max }\limits_{t \in [0,{t_0}]} {J_\varepsilon }({W_{\varepsilon ,t}})
\]
and we just need to verify that
\[
\overline {\mathop {\lim }\limits_{\varepsilon  \to 0} } \mathop {\max }\limits_{t \in [0,{t_0}]} {J_\varepsilon }({W_{\varepsilon ,t}}) \le {c_{{V_0}}}.
\]
Indeed, similar to \eqref{4.2}, we have
\[\begin{gathered}
  {\text{ }}\mathop {\max }\limits_{t \in [0,{t_0}]} {J_\varepsilon }({W_{\varepsilon ,t}}) = \mathop {\max }\limits_{t \in [0,{t_0}]} {I_{{V_0}}}(U_t^ * ) + o(1) \hfill \\
   \le \mathop {\max }\limits_{t \in [0,\infty )} {I_{{V_0}}}(U_t^ * ) + o(1) = {I_{{V_0}}}({U^ * }) + o(1) = {c_{{V_0}}} + o(1). \hfill \\
\end{gathered} \]
\end{proof}

\begin{lemma}\label{4.2.}
\[
\mathop {\underline {\lim } }\limits_{\varepsilon  \to 0} {c_\varepsilon } \ge {c_{{V_0}}}.
\]
\end{lemma}
\begin{proof}
Assuming the contrary that $\mathop {\underline {\lim } }\limits_{\varepsilon  \to 0} {c_\varepsilon } < {c_{{V_0}}}$, then, there exist ${\delta _0} > 0$, ${\varepsilon _n} \to 0$ and ${\gamma _n} \in {\Gamma _{{\varepsilon _n}}}$ satisfying ${J_{{\varepsilon _n}}}({\gamma _n}(s)) < {c_{{V_0}}} - {\delta _0}$ for $s \in [0,1]$. We can fix an ${\varepsilon _n}$ such that
\begin{equation}\label{4.3}
\frac{1}
{2}{V_0}{\varepsilon _n}(1 + {(1 + {c_{{V_0}}})^{1/2}}) < \min \{ {\delta _0},1\} .
\end{equation}
Since ${I_{{\varepsilon _n}}}({\gamma _n}(0)) = 0$ and ${I_{{\varepsilon _n}}}({\gamma _n}(1)) \le {J_{{\varepsilon _n}}}({\gamma _n}(1)) = {J_{{\varepsilon _n}}}({W_{{\varepsilon _n},{t_0}}}) <  - 2$, we can find an ${s_n} \in (0,1)$ such that ${I_{{\varepsilon _n}}}({\gamma _n}(s)) \ge  - 1$ for $s \in [0,{s_n}]$ and ${I_{{\varepsilon _n}}}({\gamma _n}({s_n})) =  - 1$. Then, for any $s \in [0,{s_n}]$,
\[
{Q_{{\varepsilon _n}}}({\gamma _n}(s)) = {J_{{\varepsilon _n}}}({\gamma _n}(s)) - {I_{{\varepsilon _n}}}({\gamma _n}(s)) \le 1 + {c_{{V_0}}} - {\delta _0},
\]
this implies that
\[
\int_{{\mathbb{R}^3}\backslash (\Lambda /{\varepsilon _n})} {\gamma _n^2(s)}  \le {\varepsilon _n}(1 + {(1 + {c_{{V_0}}})^{1/2}}){\text{ for }}s \in [0,{s_n}].
\]
Then, for $s \in [0,{s_n}]$,
\[\begin{gathered}
  {\text{   }}{I_{{\varepsilon _n}}}({\gamma _n}(s)) \hfill \\
   = {I_{{V_0}}}({\gamma _n}(s)) + \frac{1}
{2}\int_{{\mathbb{R}^3}} {(V({\varepsilon _n}x) - {V_0})\gamma _n^2(s)}  \hfill \\
   \ge {I_{{V_0}}}({\gamma _n}(s)) + \frac{1}
{2}\int_{{\mathbb{R}^3}\backslash (\Lambda /{\varepsilon _n})} {(V({\varepsilon _n}x) - {V_0})\gamma _n^2(s)}  \hfill \\
   \ge {I_{{V_0}}}({\gamma _n}(s)) - \frac{1}
{2}{V_0}{\varepsilon _n}(1 + {(1 + {c_{{V_0}}})^{1/2}}), \hfill \\
\end{gathered} \]
then
\[
\begin{gathered}
  {I_{{V_0}}}({\gamma _n}({s_n})) \le {I_{{\varepsilon _n}}}({\gamma _n}({s_n})) + \frac{1}
{2}{V_0}{\varepsilon _n}(1 + {(1 + {c_{{V_0}}})^{1/2}}) \hfill \\
  {\text{ }} =  - 1 + \frac{1}
{2}{V_0}{\varepsilon _n}(1 + {(1 + {c_{{V_0}}})^{1/2}}) < 0 \hfill \\
\end{gathered}
\]
and recalling \eqref{3.4}, we have
\[
\mathop {\max }\limits_{s \in [0,{s_n}]} {I_{{V_0}}}({\gamma _n}(s)) \ge {c_{{V_0}}}.
\]
Hence, we deduce that
\[\begin{gathered}
  {c_{{V_0}}} - {\delta _0} \ge \mathop {\max }\limits_{s \in [0,1]} {J_{{\varepsilon _n}}}({\gamma _n}(s)) \ge \mathop {\max }\limits_{s \in [0,1]} {I_{{\varepsilon _n}}}({\gamma _n}(s)) \ge \mathop {\max }\limits_{s \in [0,{s_n}]} {I_{{\varepsilon _n}}}({\gamma _n}(s)) \hfill \\
   \ge \mathop {\max }\limits_{s \in [0,{s_n}]} {I_{{V_0}}}({\gamma _n}(s)) - \frac{1}
{2}{V_0}{\varepsilon _n}(1 + {(1 + {c_{{V_0}}})^{1/2}}), \hfill \\
\end{gathered} \]
i.e. $0 < {\delta _0} \le \frac{1}
{2}{V_0}{\varepsilon _n}(1 + {(1 + {c_{{V_0}}})^{1/2}})$, which contradicts to \eqref{4.3}.
\end{proof}

Lemma~\ref{4.1.} and Lemma~\ref{4.2.} imply that
\[
\mathop {\lim }\limits_{\varepsilon  \to 0} \bigl( {\mathop {\max }\limits_{s \in [0,1]} {J_\varepsilon }({\gamma _\varepsilon }(s)) - {c_\varepsilon }} \bigr) = 0,
\]
where ${\gamma _\varepsilon }(s) = {W_{\varepsilon ,s{t_0}}}$ for ${s \in [0,1]}$. Denote
\[
{{\tilde c}_\varepsilon }: = \mathop {\max }\limits_{s \in [0,1]} {J_\varepsilon }({\gamma _\varepsilon }(s)),
\]
we see that ${c_\varepsilon } \le {{\tilde c}_\varepsilon }$ and $\mathop {\lim }\limits_{\varepsilon  \to 0} {c_\varepsilon } = \mathop {\lim }\limits_{\varepsilon  \to 0} {{\tilde c}_\varepsilon } = {c_{{V_0}}}$.

In order to state the next lemma, we need some notations. For each $R > 0$, we regard $H_0^1({B_R}(0))$ as a subspace of ${H_\varepsilon }$. Namely, for any $u \in H_0^1({B_R}(0))$, we extend $u$ by defining $u(x) = 0$ for $|x| > R$, then ${\left\|  \cdot  \right\|_{{H_\varepsilon }}}$ is equivalent to the standard norm of $H_0^1({B_R}(0))$ for each $R > 0$, $\varepsilon  > 0$. Using ${\left\|  \cdot  \right\|_{{H_\varepsilon }}}$, for each $T \in {(H_0^1({B_R}(0)))^{ - 1}}$, we define
\[
{\left\| T \right\|_{ * ,\varepsilon ,R}}: = \sup \bigl\{ {Tu:u \in H_0^1({B_R}(0)),{{\left\| u \right\|}_{{H_\varepsilon }}} \le 1} \bigr\}.
\]
Note also that ${\left\|  \cdot  \right\|_{ * ,\varepsilon ,R}}$ is equivalent to the standard norm of ${(H_0^1({B_R}(0)))^{ - 1}}$.

We use the notation
\[
J_\varepsilon ^\alpha : = \{ u \in {H_\varepsilon }:{J_\varepsilon }(u) \le \alpha \}
\]
and fix a ${R_0} > 0$ such that ${B_{{R_0}}}(0) \supset \Lambda $.

Inspired by \cite{zcz}, we have the following lemma and this lemma is a key for the proof of Theorem~\ref{1.1.}:
\begin{lemma}\label{4.3.}
(i) There exists a ${d_0} > 0$ such that for any $\{ {\varepsilon _i}\} _{i = 1}^\infty $, $\{ {R_{{\varepsilon _i}}}\} $, $\{ {u_{{\varepsilon _i}}}\} $ with
\begin{equation}\label{4.4}
\left\{ \begin{gathered}
  \mathop {\lim }\limits_{i \to \infty } {\varepsilon _i} = 0,{\text{ }}{R_{{\varepsilon _i}}} \ge {R_0}/{\varepsilon _i},{\text{ }}{u_{{\varepsilon _i}}} \in X_{{\varepsilon _i}}^{{d_0}} \cap H_0^1({B_{{R_{{\varepsilon _i}}}}}(0)), \hfill \\
  \mathop {\lim }\limits_{i \to \infty } {J_{{\varepsilon _i}}}({u_{{\varepsilon _i}}}) \le {c_{{V_0}}}{\text{ and }}\mathop {\lim }\limits_{i \to \infty } {\left\| {{{J'}_{{\varepsilon _i}}}({u_{{\varepsilon _i}}})} \right\|_{ * ,{\varepsilon _i},{R_{{\varepsilon _i}}}}} = 0, \hfill \\
\end{gathered}  \right.
\end{equation}
then there exists, up to a subsequence, $\{ {y_i}\} _{i = 1}^\infty  \subset {\mathbb{R}^3}$, ${x_0} \in \mathcal{M}$, $U \in {S_{{V_0}}}$ such that
\[
\mathop {\lim }\limits_{i \to \infty } |{\varepsilon _i}{y_i} - {x_0}| = 0{\text{ and }}\mathop {\lim }\limits_{i \to \infty } {\left\| {{u_{{\varepsilon _i}}} - \varphi ({\varepsilon _i}x - {\varepsilon _i}{y_i})U(x - {y_i})} \right\|_{{H_{{\varepsilon _i}}}}} = 0.
\]
(ii) If we drop $\{ {R_{{\varepsilon _i}}}\} $ and replace \eqref{4.4} by
\begin{equation}\label{4.5}
\mathop {\lim }\limits_{i \to \infty } {\varepsilon _i} = 0,{\text{ }}{u_{{\varepsilon _i}}} \in X_{{\varepsilon _i}}^{{d_0}},{\text{ }}\mathop {\lim }\limits_{i \to \infty } {J_{{\varepsilon _i}}}({u_{{\varepsilon _i}}}) \le {c_{{V_0}}}{\text{ and }}\mathop {\lim }\limits_{i \to \infty } {\left\| {{{J'}_{{\varepsilon _i}}}({u_{{\varepsilon _i}}})} \right\|_{{{({H_{{\varepsilon _i}}})}^{ - 1}}}} = 0,
\end{equation}
then the same conclusion holds.
\end{lemma}

\begin{proof}
We only prove (i). The proof of (ii) is similar. For notational brevity, we write $\varepsilon $ for ${\varepsilon _i}$, and still use $\varepsilon $ after taking a subsequence. By the definition of $X_\varepsilon ^{{d_0}}$, there exist $\{ {U_\varepsilon }\}  \subset {S_{{V_0}}}$ and $\{x_\varepsilon \} \subset {\mathcal{M}^\beta }$ such that
\[
{\left\| {{u_\varepsilon } - \varphi (\varepsilon x - {x_\varepsilon }){U_\varepsilon }\Bigl( {x - \frac{{{x_\varepsilon }}}
{\varepsilon }} \Bigr)} \right\|_{{H_\varepsilon }}} \le \frac{3}
{2}{d_0}.
\]
Since ${S_{{V_0}}}$ and ${\mathcal{M}^\beta }$ are compact, there exist ${U_0} \in {S_{{V_0}}}$, ${x_0} \in {\mathcal{M}^\beta }$ such that ${U_\varepsilon } \to {U_0}$ in ${H^1}({\mathbb{R}^3})$ and ${x_\varepsilon } \to {x_0}$ as $\varepsilon  \to 0$. Thus, for $\varepsilon  > 0$ small,
\begin{equation}\label{4.6}
{\left\| {{u_\varepsilon } - \varphi (\varepsilon x - {x_\varepsilon }){U_0}\Bigl( {x - \frac{{{x_\varepsilon }}}
{\varepsilon }} \Bigr)} \right\|_{{H_\varepsilon }}} \le 2{d_0}.
\end{equation}
\textbf{Step 1}: We claim that
\begin{equation}\label{4.7}
\mathop {\lim }\limits_{\varepsilon  \to 0} \mathop {\sup }\limits_{y \in {A_\varepsilon }} \int_{{B_1}(y)} {|{u_\varepsilon }{|^6}}  = 0,
\end{equation}
where ${A_\varepsilon } = {B_{3\beta /\varepsilon }}({x_\varepsilon }/\varepsilon )\backslash {B_{\beta /2\varepsilon }}({x_\varepsilon }/\varepsilon )$.

If the claim is true, by Lemma~\ref{2.5.}, we see that
\begin{equation}\label{4.8}
\mathop {\lim }\limits_{\varepsilon  \to 0} \int_{{B_\varepsilon }} {|{u_\varepsilon }{|^6}}  = 0,
\end{equation}
where ${B_\varepsilon } = {B_{2\beta /\varepsilon }}({x_\varepsilon }/\varepsilon )\backslash {B_{\beta /\varepsilon }}({x_\varepsilon }/\varepsilon )$.

Indeed, since
\[
\mathop {\sup }\limits_{y \in {A_\varepsilon }} \int_{{B_1}(y)} {|{u_\varepsilon }{|^6}}  \ge \mathop {\sup }\limits_{y \in {\mathbb{R}^3}} \int_{{B_1}(y)} {|{u_\varepsilon } \cdot {\chi _{A_\varepsilon ^1}}{|^6}} ,
\]
where $A_\varepsilon ^1 = {B_{(3\beta /\varepsilon ) - 1}}({x_\varepsilon }/\varepsilon )\backslash {B_{(\beta /2\varepsilon ) + 1}}({x_\varepsilon }/\varepsilon )$, then
\[
\mathop {\lim }\limits_{\varepsilon  \to 0} \mathop {\sup }\limits_{y \in {\mathbb{R}^3}} \int_{{B_1}(y)} {|{u_\varepsilon } \cdot {\chi _{A_\varepsilon ^1}}{|^6}}  = 0.
\]
By Lemma~\ref{2.5.}, we have
\[
\int_{{\mathbb{R}^3}} {|{u_\varepsilon } \cdot {\chi _{A_\varepsilon ^1}}{|^6}}  \to 0{\text{ as }}\varepsilon  \to 0.
\]
Since $A_\varepsilon ^1 \supset {B_\varepsilon }$ for $\varepsilon  > 0$ small, \eqref{4.8} holds.

Next, we will prove \eqref{4.7}. Assuming the contrary, there exists $r>0$ such that
\[
\mathop {\underline {\lim } }\limits_{\varepsilon  \to 0} \mathop {\sup }\limits_{y \in {A_\varepsilon }} \int_{{B_1}(y)} {|{u_\varepsilon }{|^6}}  = 2r > 0,
\]
then there exists ${y_\varepsilon } \in {A_\varepsilon }$ such that for $\varepsilon  > 0$ small, $\int_{{B_1}({y_\varepsilon })} {|{u_\varepsilon }{|^6}}  \ge r > 0$. Note also that ${y_\varepsilon } \in {A_\varepsilon }$, there exists ${x^ * } \in {\mathcal{M}^{4\beta }} \subset \Lambda $ such that $\varepsilon {y_\varepsilon } \to {x^ * }$ as $\varepsilon  \to 0$. Set ${v_\varepsilon }(x): = {u_\varepsilon }(x + {y_\varepsilon })$, then, for $\varepsilon  > 0$ small,
\begin{equation}\label{4.9}
\int_{{B_1}(0)} {|{v_\varepsilon }{|^6}}  \ge r > 0,
\end{equation}
up to a subsequence, ${v_\varepsilon } \rightharpoonup v$ in ${H^1}({\mathbb{R}^3})$ and $v$ satisfies
\[
 - \Delta v + V({x^ * })v + {\phi _v}v = \lambda {v^{p - 1}} + {v^5}{\text{ in }}{\mathbb{R}^3},{\text{ }}v \ge 0.
\]
Case 1: If $v \ne 0$, then
\[
{c_{V({x^*})}} \le {I_{V({x^*})}}(v) - \frac{1}
{3}{G_{V({x^*})}}(v) = \frac{1}
{3}V({x^*})\int_{{\mathbb{R}^3}} {{v^2}}  + \frac{{2p - 6}}
{{3p}}\lambda\int_{{\mathbb{R}^3}} {{v^p}}  + \frac{1}
{3}\int_{{\mathbb{R}^3}} {{v^6}} ,
\]
we have
\[\begin{gathered}
  {\text{  }}{\left\| V \right\|_{{L^\infty }(\bar \Lambda )}}\int_{{\mathbb{R}^3}} {{v^2}}  + \frac{{2p - 6}}
{p}\lambda\int_{{\mathbb{R}^3}} {{v^p}}  + \int_{{\mathbb{R}^3}} {{v^6}}  \hfill \\
   \ge V({x^*})\int_{{\mathbb{R}^3}} {{v^2} + \frac{{2p - 6}}
{p}\lambda\int_{{\mathbb{R}^3}} {{v^p}} }  + \int_{{\mathbb{R}^3}} {{v^6}}  \ge 3{c_{V({x^*})}} \ge 3{c_{{V_0}}}. \hfill \\
\end{gathered} \]
Hence, for sufficiently large $R$,
\[\begin{gathered}
  {\text{   }}\mathop {\underline {\lim } }\limits_{\varepsilon  \to 0} \Bigl[ {{{\left\| V \right\|}_{{L^\infty }(\bar \Lambda )}}\int_{{B_R}({y_\varepsilon })} {u_\varepsilon ^2}  + \frac{{2p - 6}}
{p}\lambda\int_{{B_R}({y_\varepsilon })} {u_\varepsilon ^p}  + \int_{{B_R}({y_\varepsilon })} {u_\varepsilon ^6} } \Bigr] \hfill \\
   = \mathop {\underline {\lim } }\limits_{\varepsilon  \to 0} \Bigl[ {{{\left\| V \right\|}_{{L^\infty }(\bar \Lambda )}}\int_{{B_R}(0)} {v_\varepsilon ^2}  + \frac{{2p - 6}}
{p}\lambda\int_{{B_R}(0)} {v_\varepsilon ^p}  + \int_{{B_R}(0)} {v_\varepsilon ^6} } \Bigr] \hfill \\
   \ge \Bigl[ {{{\left\| V \right\|}_{{L^\infty }(\bar \Lambda )}}\int_{{B_R}(0)} {{v^2}}  + \frac{{2p - 6}}
{p}\lambda\int_{{B_R}(0)} {{v^p}}  + \int_{{B_R}(0)} {{v^6}} } \Bigr] \hfill \\
   \ge \frac{1}
{2}\Bigl[ {{{\left\| V \right\|}_{{L^\infty }(\bar \Lambda )}}\int_{{\mathbb{R}^3}} {{v^2}}  + \frac{{2p - 6}}
{p}\lambda\int_{{\mathbb{R}^3}} {{v^p}}  + \int_{{\mathbb{R}^3}} {{v^6}} } \Bigr] \ge \frac{3}
{2}{c_{{V_0}}} > 0. \hfill \\
\end{gathered} \]
On the other hand, by the Sobolev's Imbedding Theorem and \eqref{4.6},
\begin{equation}\label{4.10}
\begin{gathered}
  {\text{  }}{\left\| V \right\|_{{L^\infty }(\bar \Lambda )}}\int_{{B_R}({y_\varepsilon })} {u_\varepsilon ^2}  + \frac{{2p - 6}}
{p}\lambda\int_{{B_R}({y_\varepsilon })} {u_\varepsilon ^p}  + \int_{{B_R}({y_\varepsilon })} {u_\varepsilon ^6}  \hfill \\
   \le C{d_0} + C\int_{{B_R}({y_\varepsilon })} {{{\Bigl| {\varphi (\varepsilon x - {x_\varepsilon }){U_0}\bigl( {x - \frac{{{x_\varepsilon }}}
{\varepsilon }} \bigr)} \Bigr|}^2}}  + C\lambda\int_{{B_R}({y_\varepsilon })} {{{\Bigl| {\varphi (\varepsilon x - {x_\varepsilon }){U_0}\bigl( {x - \frac{{{x_\varepsilon }}}
{\varepsilon }} \bigr)} \Bigr|}^p}}  \hfill \\
  {\text{   }}+C\int_{{B_R}({y_\varepsilon })} {{{\Bigl| {\varphi (\varepsilon x - {x_\varepsilon }){U_0}\bigl( {x - \frac{{{x_\varepsilon }}}
{\varepsilon }} \bigr)} \Bigr|}^6}}  \hfill \\
   \le C{d_0} + C\int_{{B_R}({y_\varepsilon } - \frac{{{x_\varepsilon }}}
{\varepsilon })} {|{U_0}(x){|^2}}  + C\lambda\int_{{B_R}({y_\varepsilon } - \frac{{{x_\varepsilon }}}
{\varepsilon })} {|{U_0}(x){|^p}}  + C\int_{{B_R}({y_\varepsilon } - \frac{{{x_\varepsilon }}}
{\varepsilon })} {|{U_0}(x){|^6}}  \hfill \\
   = C{d_0} + o(1), \hfill \\
\end{gathered}
\end{equation}
where $o(1) \to 0$ as $\varepsilon  \to 0$, and we have used the fact that $| {{y_\varepsilon } - \frac{{{x_\varepsilon }}}
{\varepsilon }} | \ge \beta /2\varepsilon $. This leads to a contradiction if $d_0$ is small enough.

\noindent Case 2: If $v = 0$, i.e. ${v_\varepsilon } \rightharpoonup 0$ in ${H^1}({\mathbb{R}^3})$, then ${v_\varepsilon } \to 0$ in $L_{{\text{loc}}}^s({\mathbb{R}^3})$ for $s \in [1,6)$. Thus, by \eqref{4.9} and the Sobolev's Imbedding $H_{{\text{loc}}}^1({\mathbb{R}^3}) \hookrightarrow L_{{\text{loc}}}^s({\mathbb{R}^3})$, $\exists C > 0$ (independent of $\varepsilon $) such that, for $\varepsilon  > 0$ small,
\begin{equation}\label{4.11}
\int_{{B_1}(0)} {|\nabla {v_\varepsilon }{|^2}}  \ge C{r^{1/3}} > 0.
\end{equation}
Now we claim that:
\begin{equation}\label{4.12}
\mathop {\lim }\limits_{\varepsilon  \to 0} \mathop {\sup }\limits_{\varphi  \in C_c^\infty ({B_2}(0)),{{\left\| \varphi  \right\|}_{{H^1}({\mathbb{R}^3})}} = 1} | {\langle {{\rho _\varepsilon },\varphi } \rangle } | = 0,
\end{equation}
where ${\rho _\varepsilon } = \Delta {v_\varepsilon } + {(v_\varepsilon ^ + )^5} \in {({H^1}({\mathbb{R}^3}))^{ - 1}}$. It is easy to check that for $\varepsilon  > 0$ small, $\int_{{\mathbb{R}^3}} {{\chi _\varepsilon }(x){u_\varepsilon }(x)\varphi (x - {y_\varepsilon })}  \equiv 0$ uniformly for any ${\varphi  \in C_c^\infty ({B_2}(0))}$. Thus for any ${\varphi  \in C_c^\infty ({B_2}(0))}$ with ${{{\left\| \varphi  \right\|}_{{H^1}({\mathbb{R}^3})}} = 1}$,
\[
\begin{gathered}
  \left\langle {{\rho _\varepsilon },\varphi } \right\rangle  =  - \left\langle {J'({u_\varepsilon }),\varphi (x - {y_\varepsilon })} \right\rangle  + \int_{{\mathbb{R}^3}} {V(\varepsilon x){u_\varepsilon }(x)\varphi (x - {y_\varepsilon })}  \hfill \\
  {\text{             }} + \int_{{\mathbb{R}^3}} {{\phi _{{u_\varepsilon }}}(x){u_\varepsilon }(x)\varphi (x - {y_\varepsilon })}  - \lambda \int_{{\mathbb{R}^3}} {{{(u_\varepsilon ^ + )}^{p - 1}}(x)\varphi (x - {y_\varepsilon })}  \hfill \\
  {\text{           }} = {J_1} + {J_2} + {J_3} + {J_4}. \hfill \\
\end{gathered}
\]
In view of the facts that ${\left\| {{{J'}_\varepsilon }({u_\varepsilon })} \right\|_{ * ,\varepsilon ,{R_\varepsilon }}} \to 0$, ${\text{supp}}\varphi  \subset {B_2}(0)$, $\mathop {\sup }\limits_{x \in {B_2}(0)} V(\varepsilon x + \varepsilon {y_\varepsilon }) \le C$ uniformly for all $\varepsilon  > 0$ small, ${v_\varepsilon } \to 0$ in $L_{{\text{loc}}}^s({\mathbb{R}^3})$ for $s \in [1,6)$ and Lemma~\ref{2.1.}, we have
\[\begin{gathered}
  |{J_1}| \le {\left\| {{{J'}_\varepsilon }({u_\varepsilon })} \right\|_{ * ,\varepsilon ,{R_\varepsilon }}}{\left\| {\varphi (x - {y_\varepsilon })} \right\|_{{H_\varepsilon }}} = o(1){\left\| {\varphi (x - {y_\varepsilon })} \right\|_{{H_\varepsilon }}} \hfill \\
   \le o(1){\left\| {\varphi (x - {y_\varepsilon })} \right\|_{{H^1}({\mathbb{R}^3})}} \to 0, \hfill \\
\end{gathered} \]
\[
|{J_2}| \le \mathop {\sup }\limits_{x \in {B_2}(0)} V(\varepsilon x + \varepsilon {y_\varepsilon }){\Bigl( {\int_{{B_2}(0)} {|{v_\varepsilon }{|^2}} } \Bigr)^{1/2}}{\Bigl( {\int_{{B_2}(0)} {{\varphi ^2}} } \Bigr)^{1/2}} \to 0,
\]
\[\begin{gathered}
  |{J_3}| = \Bigl| {\int_{{\mathbb{R}^3}} {{\phi _{{v_\varepsilon }}}{v_\varepsilon }\varphi } } \Bigr| \le {\Bigl( {\int_{{\mathbb{R}^3}} {|{\phi _{{v_\varepsilon }}}{|^6}} } \Bigr)^{1/6}}{\Bigl( {\int_{{B_2}(0)} {|{v_\varepsilon }{|^3}} } \Bigr)^{1/3}}{\Bigl( {\int_{{B_2}(0)} {{\varphi ^2}} } \Bigr)^{1/2}} \hfill \\
   \le C\left\| {{v_\varepsilon }} \right\|_{{L^{12/5}}({\mathbb{R}^3})}^2{\Bigl( {\int_{{B_2}(0)} {|{v_\varepsilon }{|^3}} } \Bigr)^{1/3}}{\Bigl( {\int_{{B_2}(0)} {{\varphi ^2}} } \Bigr)^{1/2}} \to 0 \hfill \\
\end{gathered} \]
and
\[
|{J_4}| = \lambda \Bigl| {\int_{{\mathbb{R}^3}} {{{(v_\varepsilon ^ + )}^{p - 1}}\varphi } } \Bigr| \le \lambda {\Bigl( {\int_{{B_2}(0)} {|{v_\varepsilon }{|^p}} } \Bigr)^{(p - 1)/p}}{\Bigl( {\int_{{B_2}(0)} {|\varphi {|^p}} } \Bigr)^{1/p}} \to 0
\]
as $\varepsilon  \to 0$ uniformly for all ${\varphi  \in C_c^\infty ({B_2}(0))}$ with ${{{\left\| \varphi  \right\|}_{{H^1}({\mathbb{R}^3})}} = 1}$, i.e. \eqref{4.12} holds.

In view of Lemma~\ref{2.6.}, we see from \eqref{4.9}, \eqref{4.11} and \eqref{4.12} that, there exist ${{\tilde y}_\varepsilon } \in {\mathbb{R}^3}$ and ${\sigma _\varepsilon } > 0$ with ${{\tilde y}_\varepsilon } \to \tilde y \in \overline {{B_1}(0)} $, ${\sigma _\varepsilon } \to 0$ as $\varepsilon  \to 0$ such that
\[
{w_\varepsilon }(x): = \sigma _\varepsilon ^{1/2}{v_\varepsilon }({\sigma _\varepsilon }x + {{\tilde y}_\varepsilon }) \rightharpoonup w{\text{ in }}{D^{1,2}}({\mathbb{R}^3})
\]
and $w \ge 0$ is a nontrivial solution of
\begin{equation}\label{add1}
 - \Delta u = {u^5},{\text{ }}u \in {D^{1,2}}({\mathbb{R}^3}).
\end{equation}
It is well known that
\[
w(x) = \frac{{{3^{1/4}}{\delta ^{1/2}}}}
{{{{({\delta ^2} + |x - {x_0}{|^2})}^{1/2}}}}
\]
for some $\delta  > 0$, ${x_0} \in {\mathbb{R}^3}$ and
\begin{equation}\label{add2}
\int_{{\mathbb{R}^3}} {|\nabla w{|^2}}  = \int_{{\mathbb{R}^3}} {{w^6}}  = {S^{3/2}},
\end{equation}
then $\exists R > 0$ such that
\[
\int_{{B_R}(0)} {{w^6}}  \ge \frac{1}
{2}\int_{{\mathbb{R}^3}} {{w^6}}  = \frac{1}
{2}{S^{3/2}} > 0.
\]
On the other hand,
\begin{equation}\label{4.13}
\int_{{B_R}(0)} {{w^6}}  \le \mathop {\underline {\lim } }\limits_{\varepsilon  \to 0} \int_{{B_R}(0)} {w_\varepsilon ^6}  = \mathop {\underline {\lim } }\limits_{\varepsilon  \to 0} \int_{{B_{{\sigma _\varepsilon }R}}({{\tilde y}_\varepsilon })} {v_\varepsilon ^6}  = \mathop {\underline {\lim } }\limits_{\varepsilon  \to 0} \int_{{B_{{\sigma _\varepsilon }R}}({{\tilde y}_\varepsilon } + {y_\varepsilon })} {u_\varepsilon ^6}  \le \mathop {\underline {\lim } }\limits_{\varepsilon  \to 0} \int_{{B_2}({y_\varepsilon })} {u_\varepsilon ^6} ,
\end{equation}
where we have used the facts that ${\sigma _\varepsilon } \to 0$ and ${{\tilde y}_\varepsilon } \to \tilde y \in \overline {{B_1}(0)} $ as $\varepsilon  \to 0$.

Similar to \eqref{4.10}, we can check that \eqref{4.13} leads to a contradiction for ${d_0} > 0$ small. Hence \eqref{4.7} holds.

For any $s \in (2,6)$, using the Interpolation Inequality for ${L^p}$ norms and \eqref{4.8}, we have
\begin{equation}\label{4.14}
\int_{{B_\varepsilon }} {|{u_\varepsilon }{|^s}}  \le {\Bigl( {\int_{{B_\varepsilon }} {|{u_\varepsilon }{|^2}} } \Bigr)^{\frac{3}
{2} - \frac{s}
{4}}}{\Bigl( {\int_{{B_\varepsilon }} {|{u_\varepsilon }{|^6}} } \Bigr)^{\frac{s}
{4} - \frac{1}
{2}}} \le C{\Bigl( {\int_{{B_\varepsilon }} {|{u_\varepsilon }{|^6}} } \Bigr)^{\frac{s}
{4} - \frac{1}
{2}}} \to 0{\text{ as }}\varepsilon  \to 0.
\end{equation}
It follows that
\begin{equation}\label{4.15}
\mathop {\lim }\limits_{\varepsilon  \to 0} \int_{{B_\varepsilon }} {|{u_\varepsilon }{|^s}}  = 0{\text{ for all }}s \in (2,6].
\end{equation}
\textbf{Step 2}: Let $u_{\varepsilon ,1}(x) = \varphi (\varepsilon x - {x_\varepsilon }){u_\varepsilon }(x)$, $u_{\varepsilon ,2}(x) = (1 - \varphi (\varepsilon x - {x_\varepsilon })){u_\varepsilon }(x)$. By \eqref{4.15} and direct computations, we can check that
\[
\int_{{\mathbb{R}^3}} {{{(u_\varepsilon ^ + )}^s}}  = \int_{{\mathbb{R}^3}} {{{({{(u_{\varepsilon ,1})}^ + })}^s}}  + \int_{{\mathbb{R}^3}} {{{({{(u_{\varepsilon ,2})}^ + })}^s}}  + o(1),{\text{ }}s \in (2,6],
\]
\[
\int_{{\mathbb{R}^3}} {|\nabla {u_\varepsilon }{|^2}}  \ge \int_{{\mathbb{R}^3}} {|\nabla u_{\varepsilon ,1}{|^2}}  + \int_{{\mathbb{R}^3}} {|\nabla u_{\varepsilon ,2}{|^2}}  + o(1),
\]
\[
\int_{{\mathbb{R}^3}} {V(\varepsilon x)|{u_\varepsilon }{|^2}}  \ge \int_{{\mathbb{R}^3}} {V(\varepsilon x)|u_{\varepsilon ,1}{|^2}}  + \int_{{\mathbb{R}^3}} {V(\varepsilon x)|u_{\varepsilon ,2}{|^2}} ,
\]
\[
\int_{{\mathbb{R}^3}} {{\phi _{{u_\varepsilon }}}{{({u_\varepsilon })}^2}}  \ge \int_{{\mathbb{R}^3}} {{\phi _{{u_{\varepsilon ,1}}}}{{({u_{\varepsilon ,1}})}^2}}  + \int_{{\mathbb{R}^3}} {{\phi _{{u_{\varepsilon ,2}}}}{{({u_{\varepsilon ,2}})}^2}},
\]
\[
{Q_\varepsilon }(u_{\varepsilon ,1}) = 0,{\text{ }}{Q_\varepsilon }(u_{\varepsilon ,2}) = {Q_\varepsilon }({u_\varepsilon }) \ge 0.
\]
Hence we get,
\begin{equation}\label{4.16}
{J_\varepsilon }({u_\varepsilon }) \ge {I_\varepsilon }(u_{\varepsilon ,1}) + {I_\varepsilon }(u_{\varepsilon ,2}) + o(1).
\end{equation}

Next, we claim that ${\left\| {u_{\varepsilon ,2}} \right\|_{{H_\varepsilon }}} \to 0$ as $\varepsilon  \to 0$.

By \eqref{4.6}, we have
\begin{equation}\label{4.17}
\begin{gathered}
  {\left\| {u_{\varepsilon ,2}} \right\|_{{H_\varepsilon }}} \le {\left\| {u_{\varepsilon ,1} - \varphi (\varepsilon x - {x_\varepsilon }){U_0}\bigl( {x - \frac{{{x_\varepsilon }}}
{\varepsilon }} \bigr)} \right\|_{{H_\varepsilon }}} + 2{d_0} \hfill \\
  {\text{         }} = {\left\| {u_{\varepsilon ,1} - \varphi (\varepsilon x - {x_\varepsilon }){U_0}\bigl( {x - \frac{{{x_\varepsilon }}}
{\varepsilon }} \bigr)} \right\|_{{H_\varepsilon }({B_{2\beta /\varepsilon }}({x_\varepsilon }/\varepsilon ))}} + 2{d_0} \hfill \\
  {\text{         }} \le {\left\| {u_{\varepsilon ,2}} \right\|_{{H_\varepsilon }({B_{2\beta /\varepsilon }}({x_\varepsilon }/\varepsilon ))}} + 4{d_0} \hfill \\
  {\text{         }} = {\left\| {u_{\varepsilon ,2}} \right\|_{{H_\varepsilon }({B_{2\beta /\varepsilon }}({x_\varepsilon }/\varepsilon )\backslash {B_{\beta /\varepsilon }}({x_\varepsilon }/\varepsilon ))}} + 4{d_0} \hfill \\
  {\text{         }} \le C{\left\| {{u_\varepsilon }} \right\|_{{H_\varepsilon }({B_{2\beta /\varepsilon }}({x_\varepsilon }/\varepsilon )\backslash {B_{\beta /\varepsilon }}({x_\varepsilon }/\varepsilon ))}} + 4{d_0} \hfill \\
  {\text{         }} \le C{\left\| {\varphi (\varepsilon x - {x_\varepsilon }){U_0}\bigl( {x - \frac{{{x_\varepsilon }}}
{\varepsilon }} \bigr)} \right\|_{{H_\varepsilon }({B_{2\beta /\varepsilon }}({x_\varepsilon }/\varepsilon )\backslash {B_{\beta /\varepsilon }}({x_\varepsilon }/\varepsilon ))}} + C{d_0} \hfill \\
  {\text{         }} \le C{\left\| {{U_0}\bigl( {x - \frac{{{x_\varepsilon }}}
{\varepsilon }} \bigr)} \right\|_{{H^1}({B_{2\beta /\varepsilon }}({x_\varepsilon }/\varepsilon )\backslash {B_{\beta /\varepsilon }}({x_\varepsilon }/\varepsilon ))}} + C{d_0} \hfill \\
  {\text{         }} \le C{\left\| {{U_0}} \right\|_{{H^1}({B_{2\beta /\varepsilon }}(0)\backslash {B_{\beta /\varepsilon }}(0))}} + C{d_0} = C{d_0} + o(1), \hfill \\
\end{gathered}
\end{equation}
where $o(1) \to 0$ as $\varepsilon  \to 0$. Hence we have $\mathop {\overline {\lim } }\limits_{\varepsilon  \to 0} {\left\| {u_{\varepsilon ,2}} \right\|_{{H_\varepsilon }}} \le C{d_0}$.

By \eqref{4.15} and the facts that $\left\langle {{{J'}_\varepsilon }({u_\varepsilon }),{u_{\varepsilon ,2}}} \right\rangle  \to 0$ as $\varepsilon  \to 0$ and $\left\langle {{{Q'}_\varepsilon }({u_\varepsilon }),{u_{\varepsilon ,2}}} \right\rangle  = \left\langle {{{Q'}_\varepsilon }({u_{\varepsilon ,2}}),{u_{\varepsilon ,2}}} \right\rangle  \ge 0$, we get
\[\begin{gathered}
  {\text{ }}\int_{{\mathbb{R}^3}} {\nabla {u_\varepsilon } \cdot \nabla {u_{\varepsilon ,2}}}  + \int_{{\mathbb{R}^3}} {V(\varepsilon x){u_\varepsilon }{u_{\varepsilon ,2}}}  + \int_{{\mathbb{R}^3}} {{\phi _{{u_\varepsilon }}}{u_\varepsilon }{u_{\varepsilon ,2}}}  + \left\langle {{{Q'}_\varepsilon }({u_{\varepsilon ,2}}),{u_{\varepsilon ,2}}} \right\rangle  \hfill \\
   = \lambda \int_{{\mathbb{R}^3}} {{{(u_\varepsilon ^ + )}^{p - 1}}{u_{\varepsilon ,2}}}  + \int\limits_{{\mathbb{R}^3}} {{{(u_\varepsilon ^ + )}^5}{u_{\varepsilon ,2}}}  + o(1), \hfill \\
\end{gathered} \]
then
\[\begin{gathered}
  {\text{ }}\left\| {{u_{\varepsilon ,2}}} \right\|_{{H_\varepsilon }}^2 \le \lambda \int_{{\mathbb{R}^3}} {|{u_{\varepsilon ,2}}{|^p}}  + \int_{{\mathbb{R}^3}} {|{u_{\varepsilon ,2}}{|^6}}  + o(1) \hfill \\
   \le C\lambda \left\| {{u_{\varepsilon ,2}}} \right\|_{{H_\varepsilon }}^p + C\left\| {{u_{\varepsilon ,2}}} \right\|_{{H_\varepsilon }}^6 + o(1) \le \frac{1}
{2}\left\| {{u_{\varepsilon ,2}}} \right\|_{{H_\varepsilon }}^2 + C\left\| {{u_{\varepsilon ,2}}} \right\|_{{H_\varepsilon }}^6 + o(1), \hfill \\
\end{gathered} \]
i.e. $\left\| {{u_{\varepsilon ,2}}} \right\|_{{H_\varepsilon }}^2 \le C\left\| {{u_{\varepsilon ,2}}} \right\|_{{H_\varepsilon }}^6 + o(1)$.

Taking ${d_0} > 0$ small, we have ${\left\| {u_{\varepsilon ,2}} \right\|_{{H_\varepsilon }}} = o(1)$. From \eqref{4.16}, it holds that
\begin{equation}\label{4.18}
{J_\varepsilon }({u_\varepsilon }) \ge {I_\varepsilon }(u_{\varepsilon ,1}) + o(1).
\end{equation}
\textbf{Step 3}: Let ${\tilde w_\varepsilon }(x) = {u_{\varepsilon ,1}}\bigl( {x + \frac{{{x_\varepsilon }}}
{\varepsilon }} \bigr) = \varphi (\varepsilon x){u_\varepsilon }\bigl( {x + \frac{{{x_\varepsilon }}}
{\varepsilon }} \bigr)$, up to a subsequence, $\exists \tilde w \in {H^1}({\mathbb{R}^3})$ such that
\begin{equation}\label{4.19}
{{\tilde w}_\varepsilon } \rightharpoonup \tilde w{\text{ in }}{H^1}({\mathbb{R}^3})
\end{equation}
and
\begin{equation}\label{4.20}
{{\tilde w}_\varepsilon } \to \tilde w{\text{ a}}{\text{.e}}{\text{. in }}{\mathbb{R}^3}.
\end{equation}
We claim that
\begin{equation}\label{4.21}
{{\tilde w}_\varepsilon } \to \tilde w{\text{ in }}{L^6}({\mathbb{R}^3}).
\end{equation}
In view of Lemma~\ref{2.5.}, assuming the contrary that $\exists r > 0$ such that
\[
\mathop {\underline {\lim } }\limits_{\varepsilon  \to 0} \mathop {\sup }\limits_{z \in {\mathbb{R}^3}} \int_{{B_1}(z)} {|{{\tilde w}_\varepsilon } - \tilde w{|^6}}  = 2r > 0.
\]
Then, for $\varepsilon  > 0$ small, there exists ${z_\varepsilon } \in {\mathbb{R}^3}$ such that
\begin{equation}\label{4.22}
\int_{{B_1}({z_\varepsilon })} {|{{\tilde w}_\varepsilon } - \tilde w{|^6}}  \ge r > 0.
\end{equation}
Case 1: $\{ {z_\varepsilon }\} $ is bounded, i.e. $|{z_\varepsilon }| \le \alpha $ for some $\alpha  > 0$, then for $\varepsilon  > 0$ small,
\begin{equation}\label{4.23}
\int_{{B_{\alpha  + 1}}(0)} {|{{\tilde v}_\varepsilon }{|^6}}  \ge r > 0,
\end{equation}
where ${{\tilde v}_\varepsilon } = {{\tilde w}_\varepsilon } - \tilde w$ and ${{\tilde v}_\varepsilon } \rightharpoonup 0$ in ${H^1}({\mathbb{R}^3})$. Similar as in \textbf{Step 1}, $\exists C > 0$ (independent of $\varepsilon $), such that for $\varepsilon  > 0$ small,
\begin{equation}\label{4.24}
\int_{{B_{\alpha  + 1}}(0)} {|\nabla {{\tilde v}_\varepsilon }{|^2}}  \ge C{r^{1/3}} > 0.
\end{equation}
Now, we claim that
\begin{equation}\label{4.25}
\mathop {\lim }\limits_{\varepsilon  \to 0} \mathop {\sup }\limits_{\tilde \varphi  \in C_c^\infty ({B_{\alpha  + 2}}(0)),{{\left\| {\tilde \varphi } \right\|}_{{H^1}({\mathbb{R}^3})}} = 1} \left| {\left\langle {{{\tilde \rho }_\varepsilon },\tilde \varphi } \right\rangle } \right| = 0,
\end{equation}
where ${{\tilde \rho }_\varepsilon } = \Delta {{\tilde v}_\varepsilon } + {(\tilde v_\varepsilon ^ + )^5} \in {({H^1}({\mathbb{R}^3}))^{ - 1}}$. It is easy to check that for $\varepsilon  > 0$ small, $\int_{{\mathbb{R}^3}} {{\chi _\varepsilon }(x){u_\varepsilon }(x)\tilde \varphi \Bigl( {x - \frac{{{x_\varepsilon }}}
{\varepsilon }} \Bigr)}  \equiv 0$ uniformly for all ${\tilde \varphi  \in C_c^\infty ({B_{\alpha  + 2}}(0))}$. Hence, we have
\begin{equation}\label{4.26}
\begin{gathered}
  o(1) = \Bigl\langle {{{J'}_\varepsilon }({u_\varepsilon }),\tilde \varphi \bigl( {x - \frac{{{x_\varepsilon }}}
{\varepsilon }} \bigr)} \Bigr\rangle  \hfill \\
   = \int_{{\mathbb{R}^3}} {\nabla {u_\varepsilon }\bigl( {x + \frac{{{x_\varepsilon }}}
{\varepsilon }} \bigr) \cdot \nabla \tilde \varphi }  + \int_{{\mathbb{R}^3}} {V(\varepsilon x + {x_\varepsilon }){u_\varepsilon }\bigl( {x + \frac{{{x_\varepsilon }}}
{\varepsilon }} \bigr)\tilde \varphi }  + \int_{{\mathbb{R}^3}} {{\phi _{{u_\varepsilon }( {x + \frac{{{x_\varepsilon }}}
{\varepsilon }} )}}{u_\varepsilon }\bigl( {x + \frac{{{x_\varepsilon }}}
{\varepsilon }} \bigr)\tilde \varphi }  \hfill \\
   \text{}- \lambda \int_{{\mathbb{R}^3}} {{{\Bigl( {u_\varepsilon ^ + \bigl( {x + \frac{{{x_\varepsilon }}}
{\varepsilon }} \bigr)} \Bigr)}^{p - 1}}\tilde \varphi }  - \lambda \int_{{\mathbb{R}^3}} {{{\Bigl( {u_\varepsilon ^ + \bigl( {x + \frac{{{x_\varepsilon }}}
{\varepsilon }} \bigr)} \Bigr)}^5}\tilde \varphi }  \hfill \\
   = \int_{{\mathbb{R}^3}} {\nabla {{\tilde w}_\varepsilon } \cdot \nabla \tilde \varphi }  + \int_{{\mathbb{R}^3}} {V(\varepsilon x + {x_\varepsilon }){{\tilde w}_\varepsilon }\tilde \varphi }  + \int_{{\mathbb{R}^3}} {{\phi _{{{\tilde w}_\varepsilon }}}{{\tilde w}_\varepsilon }\tilde \varphi }  \hfill \\
   \text{}- \lambda \int_{{\mathbb{R}^3}} {{{(\tilde w_\varepsilon ^ + )}^{p - 1}}\tilde \varphi }  - \lambda \int_{{\mathbb{R}^3}} {{{(\tilde w_\varepsilon ^ + )}^5}\tilde \varphi }  + o(1), \hfill \\
\end{gathered}
\end{equation}
where we have used the fact that ${\| {{u_{\varepsilon ,2}}} \|_{{H_\varepsilon }}} \to 0$ as $\varepsilon  \to 0$ and note that $o(1) \to 0$ as $\varepsilon  \to 0$ uniformly for all ${\tilde \varphi  \in C_c^\infty ({B_{\alpha  + 2}}(0))}$ with ${{{\| {\tilde \varphi } \|}_{{H^1}({\mathbb{R}^3})}} = 1}$.

By \eqref{4.26} and the fact that ${x_\varepsilon } \to {x_0} \in {\mathcal{M}^\beta }$ as $\varepsilon  \to 0$, we see that $\tilde w \ge 0$ and satisfies
\begin{equation}\label{4.27}
 - \Delta \tilde w + V({x_0})\tilde w + {\phi _{\tilde w}}\tilde w = \lambda {{\tilde w}^{p - 1}} + {{\tilde w}^5}{\text{ in }}{\mathbb{R}^3}.
\end{equation}
By Lemma~\ref{2.2.}(ii) and direct computations, we can check that the following Brezis-Lieb splitting properties hold, as $\varepsilon  \to 0$,
\begin{equation}\label{4.28}
\left\{ \begin{gathered}
  \int_{{\mathbb{R}^3}} {{{(\tilde w_\varepsilon ^ + )}^5}\tilde \varphi  - {{(\tilde v_\varepsilon ^ + )}^5}\tilde \varphi  - {{(\tilde w)}^5}\tilde \varphi }  \to 0, \hfill \\
  \int_{{\mathbb{R}^3}} {{{(\tilde w_\varepsilon ^ + )}^{p - 1}}\tilde \varphi  - {{(\tilde v_\varepsilon ^ + )}^{p - 1}}\tilde \varphi  - {{(\tilde w)}^{p - 1}}\tilde \varphi }  \to 0, \hfill \\
  \int_{{\mathbb{R}^3}} {{\phi _{{{\tilde w}_\varepsilon }}}{{\tilde w}_\varepsilon }\tilde \varphi  - {\phi _{{{\tilde v}_\varepsilon }}}{{\tilde v}_\varepsilon }\tilde \varphi  - {\phi _{\tilde w}}\tilde w\tilde \varphi }  \to 0, \hfill \\
  \int_{{\mathbb{R}^3}} {\nabla {{\tilde w}_\varepsilon } \cdot \nabla \tilde \varphi  - \nabla {{\tilde v}_\varepsilon } \cdot \nabla \tilde \varphi  - \nabla \tilde w \cdot \nabla \tilde \varphi }  = 0 \hfill \\
\end{gathered}  \right.
\end{equation}
and
\begin{equation}\label{4.29}
\int_{{\mathbb{R}^3}} {(V(\varepsilon x + {x_\varepsilon }){{\tilde w}_\varepsilon } - V({x_0})\tilde w)\tilde \varphi }  \to 0
\end{equation}
uniformly for all ${\tilde \varphi  \in C_c^\infty ({B_{\alpha  + 2}}(0))}$ with ${{{\| {\tilde \varphi } \|}_{{H^1}({\mathbb{R}^3})}} = 1}$. From \eqref{4.26}, \eqref{4.27}, \eqref{4.28} and \eqref{4.29}, we can verify \eqref{4.25}.

By Lemma~\ref{2.6.}, we see from \eqref{4.23}, \eqref{4.24} and \eqref{4.25} that, there exist ${\tilde z_\varepsilon } \in {\mathbb{R}^3}$ and ${\delta _\varepsilon } > 0$ such that ${\tilde z_\varepsilon } \to \tilde z \in {B_{\alpha  + 1}}(0)$, ${\delta _\varepsilon } \to 0$ and
\[
{{\hat w}_\varepsilon }(x): = \delta _\varepsilon ^{1/2}{{\tilde v}_\varepsilon }({\delta _\varepsilon }x + {{\tilde z}_\varepsilon }) \rightharpoonup \hat w(x){\text{ in }}{D^{1,2}}({\mathbb{R}^3}),
\]
where $\hat w \ge 0$ is a nontrivial solution of \eqref{add1} and satisfies \eqref{add2}.

Since
\begin{equation}\label{4.31}
\int_{{\mathbb{R}^3}} {|\hat w{|^6}}  \le \mathop {\underline {\lim } }\limits_{\varepsilon  \to 0} \int_{{\mathbb{R}^3}} {|{{\hat w}_\varepsilon }{|^6}}  = \mathop {\underline {\lim } }\limits_{\varepsilon  \to 0} \int_{{\mathbb{R}^3}} {|{{\tilde v}_\varepsilon }{|^6}}  = \mathop {\underline {\lim } }\limits_{\varepsilon  \to 0} \int_{{\mathbb{R}^3}} {|{{\tilde w}_\varepsilon }{|^6}}  - \int_{{\mathbb{R}^3}} {|\tilde w{|^6}}  \le \mathop {\underline {\lim } }\limits_{\varepsilon  \to 0} \int_{{\mathbb{R}^3}} {|{u_\varepsilon }{|^6}} ,
\end{equation}
then by \eqref{4.6} and the Sobolev's Imbedding Theorem, we get
\[
\int_{{\mathbb{R}^3}} {|{u_\varepsilon }{|^6}}  \le C{d_0} + \int_{{\mathbb{R}^3}} {{{\Bigl| {\varphi (\varepsilon x - {x_\varepsilon }){U_0}\bigl( {x - \frac{{{x_\varepsilon }}}
{\varepsilon }} \bigr)} \Bigr|}^6}}  \le C{d_0} + \int_{{\mathbb{R}^3}} {U_0^6} ,
\]
and combining with \eqref{4.31}, it holds that
\begin{equation}\label{4.33}
\int_{{\mathbb{R}^3}} {|\hat w{|^6}}  \le C{d_0} + \int_{{\mathbb{R}^3}} {U_0^6} .
\end{equation}
Thus
\[\begin{gathered}
  {c_{{V_0}}} = {I_{{V_0}}}({U_0}) - \frac{1}
{3}{G_{{V_0}}}({U_0}) = \frac{1}
{3}\int_{{\mathbb{R}^3}} {U_0^2}  + \frac{{2p - 6}}
{{3p}}\lambda\int_{{\mathbb{R}^3}} {U_0^p}  + \frac{1}
{3}\int_{{\mathbb{R}^3}} {U_0^6}  \hfill \\
   \text{}\ge \frac{1}
{3}\int_{{\mathbb{R}^3}} {|\hat w{|^6}}  - C{d_0} \ge \frac{1}
{3}{S^{\frac{3}
{2}}} - C{d_0}, \hfill \\
\end{gathered} \]
where we have used \eqref{add2} and \eqref{4.33}. Letting ${d_0} \to 0$, we have
\[
{c_{{V_0}}} \ge \frac{1}
{3}{S^{\frac{3}
{2}}},
\]
which contradicts to Lemma~\ref{3.4.}.

\noindent Case 2: $\{ {z_\varepsilon }\} $ is unbounded. Without loss of generality, $\mathop {\lim }\limits_{\varepsilon  \to 0} |{z_\varepsilon }| = \infty $. Then, by \eqref{4.22},
\begin{equation}\label{4.34}
\mathop {\underline {\lim } }\limits_{\varepsilon  \to 0} \int_{{B_1}({z_\varepsilon })} {|{{\tilde w}_\varepsilon }{|^6}}  \ge r > 0,
\end{equation}
i.e.
\[
\mathop {\underline {\lim } }\limits_{\varepsilon  \to 0} \int_{{B_1}({z_\varepsilon })} {{{\Bigl| {\varphi (\varepsilon x){u_\varepsilon }\bigl( {x + \frac{{{x_\varepsilon }}}
{\varepsilon }} \bigr)} \Bigr|}^6}}  \ge r > 0.
\]
 Since $\varphi (x) = 0$ for $|x| \ge 2\beta $, we see that $|{z_\varepsilon }| \le 3\beta /\varepsilon $ for $\varepsilon  > 0$ small. If $|{z_\varepsilon }| \ge \beta /2\varepsilon $, then ${z_\varepsilon } \in {B_{3\beta /\varepsilon }}(0)\backslash {B_{\beta /2\varepsilon }}(0)$ and by \textbf{Step 1}, we get
\[
\mathop {\underline {\lim } }\limits_{\varepsilon  \to 0} \int_{{B_1}({z_\varepsilon })} {|{{\tilde w}_\varepsilon }{|^6}}  \le \mathop {\underline {\lim } }\limits_{\varepsilon  \to 0} \mathop {\sup }\limits_{z \in {B_{3\beta /\varepsilon }}(0)\backslash {B_{\beta /2\varepsilon }}(0)} \int_{{B_1}(z)} {{{\Bigl| {{u_\varepsilon }\bigl( {x + \frac{{{x_\varepsilon }}}
{\varepsilon }} \bigr)} \Bigr|}^6}}  = \mathop {\underline {\lim } }\limits_{\varepsilon  \to 0} \mathop {\sup }\limits_{z \in {A_\varepsilon }} \int_{{B_1}(z)} {|{u_\varepsilon }{|^6}}  = 0,
\]
which contradicts to \eqref{4.34}. Thus $|{z_\varepsilon }| \le \beta /2\varepsilon $ for $\varepsilon  > 0$ small. Assume that $\varepsilon {z_\varepsilon } \to {z_0} \in \overline {{B_{\beta /2}}(0)} $ and ${{\bar w}_\varepsilon }(x): = {{\tilde w}_\varepsilon }(x + {z_\varepsilon }) \rightharpoonup \bar w(x)$ in ${H^1}({\mathbb{R}^3})$. If $\bar w \ne 0$, we see that $\bar w$ satisfies
\[
 - \Delta \bar w + V({x_0} + {z_0})\bar w + {\phi _{\bar w}}\bar w = \lambda {{\bar w}^{p - 1}} + {{\bar w}^5}{\text{ in }}{\mathbb{R}^3},{\text{ }}\bar w \ge 0.
\]
Similar as in \textbf{Step 1} \eqref{4.10}, we get a contradiction if ${d_0} > 0$ is small enough. Thus $\bar w \equiv 0$, i.e.
\[
{{\bar w}_\varepsilon } \rightharpoonup 0{\text{ in }}{H^1}({\mathbb{R}^3}).
\]
By \eqref{4.34}, we have
\begin{equation}\label{4.35}
\mathop {\underline {\lim } }\limits_{\varepsilon  \to 0} \int_{{B_1}(0)} {|{{\bar w}_\varepsilon }{|^6}}  \ge r > 0
\end{equation}
and similar as in \textbf{Step 1}, we can check that $\exists C > 0$ (independent of $\varepsilon $) such that for $\varepsilon  > 0$ small,
\begin{equation}\label{4.36}
\int_{{B_1}(0)} {|\nabla {{\bar w}_\varepsilon }{|^2}}  \ge C{r^{1/3}} > 0
\end{equation}
and
\begin{equation}\label{4.37}
\mathop {\lim }\limits_{\varepsilon  \to 0} \mathop {\sup }\limits_{\bar \varphi  \in C_c^\infty ({B_2}(0)),{{\left\| {\bar \varphi } \right\|}_{{H^1}({\mathbb{R}^3})}} = 1} \left| {\left\langle {{{\bar \rho }_\varepsilon },\bar \varphi } \right\rangle } \right| = 0,
\end{equation}
where ${{\bar \rho }_\varepsilon } = \Delta {{\bar w}_\varepsilon } + {(\bar w_\varepsilon ^ + )^5} \in {({H^1}({\mathbb{R}^3}))^{ - 1}}$. By Lemma~\ref{2.6.} again, we see from \eqref{4.35}, \eqref{4.36} and \eqref{4.37} that $\exists {{\tilde x}_\varepsilon } \in {\mathbb{R}^3}$ and ${\gamma _\varepsilon } > 0$ such that ${{\tilde x}_\varepsilon } \to \tilde x \in \overline {{B_1}(0)} $, ${\gamma _\varepsilon } \to 0$ as $\varepsilon  \to 0$ and
\[
w_\varepsilon ^ * (x): = \gamma _\varepsilon ^{1/2}{{\bar w}_\varepsilon }({\gamma _\varepsilon }x + {{\tilde x}_\varepsilon }) \rightharpoonup {w^ * }(x){\text{ in }}{D^{1,2}}({\mathbb{R}^3}),
\]
where ${w^ * } \ge 0$ is a nontrivial solution of \eqref{add1} and satisfies \eqref{add2}.
Thus, $\exists R > 0$ such that
\[
\int_{{B_R}(0)} {|{w^ * }{|^6}}  \ge \frac{1}
{2}\int_{{\mathbb{R}^3}} {|{w^ * }{|^6}}  = \frac{1}
{2}{S^{\frac{3}
{2}}} > 0.
\]
On the other hand,
\[\begin{gathered}
  \frac{1}
{2}{S^{\frac{3}
{2}}} \le \int_{{B_R}(0)} {|{w^ * }{|^6}}  \le \mathop {\underline {\lim } }\limits_{\varepsilon  \to 0} \int_{{B_R}(0)} {|w_\varepsilon ^ * {|^6}}  = \mathop {\underline {\lim } }\limits_{\varepsilon  \to 0} \int_{{B_{{\gamma _\varepsilon }R}}({{\tilde x}_\varepsilon })} {|{{\bar w}_\varepsilon }{|^6}}  \hfill \\
  {\text{    }} \le \mathop {\underline {\lim } }\limits_{\varepsilon  \to 0} \int_{{B_{{\gamma _\varepsilon }R}}({{\tilde x}_\varepsilon } + {z_\varepsilon } + \frac{{{x_\varepsilon }}}
{\varepsilon })} {|{u_\varepsilon }{|^6}}  \le \mathop {\underline {\lim } }\limits_{\varepsilon  \to 0} \int_{{B_2}({z_\varepsilon } + \frac{{{x_\varepsilon }}}
{\varepsilon })} {|{u_\varepsilon }{|^6}} , \hfill \\
\end{gathered} \]
which contradicts to \eqref{4.6} for ${d_0} > 0$ small. Therefore
\[
\mathop { {\lim } }\limits_{\varepsilon  \to 0} \mathop {\sup }\limits_{z \in {\mathbb{R}^3}} \int_{{B_1}(z)} {|{{\tilde w}_\varepsilon } - \tilde w{|^6}}  = 0.
\]
By Lemma~\ref{2.5.}, \eqref{4.21} holds. Similar to \eqref{4.14}, using the Interpolation Inequality for ${L^p}$ norms, we have
\begin{equation}\label{4.38}
{{\tilde w}_\varepsilon } \to \tilde w{\text{ in }}{L^s}({\mathbb{R}^3}),{\text{ }}s \in (2,6].
\end{equation}
In view of \eqref{4.18} and recall that ${{\tilde w}_\varepsilon }(x) = {u_{\varepsilon ,1}}\bigl( {x + \frac{{{x_\varepsilon }}}
{\varepsilon }} \bigr)$, we have
\[\begin{gathered}
  \frac{1}
{2}\int_{{\mathbb{R}^3}} {|\nabla {{\tilde w}_\varepsilon }{|^2}}  + \frac{1}
{2}\int_{{\mathbb{R}^3}} {V(\varepsilon x + {x_\varepsilon })\tilde w_\varepsilon ^2}  + \frac{1}
{4}\int_{{\mathbb{R}^3}} {{\phi _{{{\tilde w}_\varepsilon }}}\tilde w_\varepsilon ^2}  \hfill \\
  {\text{   }} - \frac{1}
{p}\lambda \int_{{\mathbb{R}^3}} {{{(\tilde w_\varepsilon ^ + )}^p}}  - \frac{1}
{6}\int_{{\mathbb{R}^3}} {{{(\tilde w_\varepsilon ^ + )}^6}}  \le {c_{{V_0}}} + o(1). \hfill \\
\end{gathered} \]
By Lemma~\ref{2.1.}(iii), \eqref{4.19}, \eqref{4.20} and \eqref{4.38}, we get
\[
\frac{1}
{2}\int_{{\mathbb{R}^3}} {|\nabla \tilde w{|^2}}  + \frac{1}
{2}\int_{{\mathbb{R}^3}} {V({x_0}){{\tilde w}^2}}  + \frac{1}
{4}\int_{{\mathbb{R}^3}} {{\phi _{\tilde w}}{{\tilde w}^2}}  - \frac{1}
{p}\lambda \int_{{\mathbb{R}^3}} {{{({{\tilde w}^ + })}^p}}  - \frac{1}
{6}\int_{{\mathbb{R}^3}} {{{({{\tilde w}^ + })}^6}}  \le {c_{{V_0}}},
\]
i.e.
\begin{equation}\label{4.39}
{I_{V({x_0})}}(\tilde w) \le {c_{{V_0}}}.
\end{equation}
Since $\left\langle {{{J'}_\varepsilon }({u_\varepsilon }),{u_{\varepsilon ,1}}} \right\rangle  \to 0$, ${\left\| {{u_{\varepsilon ,2}}} \right\|_{{H_\varepsilon }}} \to 0$ as $\varepsilon  \to 0$ and $\left\langle {{{Q'}_\varepsilon }({u_\varepsilon }),{u_{\varepsilon ,1}}} \right\rangle  \equiv 0$ and together with the fact that ${{\tilde w}_\varepsilon }(x) = {u_{\varepsilon ,1}}\bigl( {x + \frac{{{x_\varepsilon }}}
{\varepsilon }} \bigr)$, we get
\[
\int_{{\mathbb{R}^3}} {|\nabla {{\tilde w}_\varepsilon }{|^2}}  + \int_{{\mathbb{R}^3}} {V(\varepsilon x + {x_\varepsilon })\tilde w_\varepsilon ^2}  + \int_{{\mathbb{R}^3}} {{\phi _{{{\tilde w}_\varepsilon }}}\tilde w_\varepsilon ^2}  = \lambda \int_{{\mathbb{R}^3}} {{{(\tilde w_\varepsilon ^ + )}^p}}  + \int_{{\mathbb{R}^3}} {{{(\tilde w_\varepsilon ^ + )}^6}}  + o(1),
\]
then by \eqref{4.27}, we have
\[\begin{gathered}
  {\text{   }}\int_{{\mathbb{R}^3}} {|\nabla \tilde w{|^2}}  + \int_{{\mathbb{R}^3}} {V({x_0}){{\tilde w}^2}}  + \int_{{\mathbb{R}^3}} {{\phi _{\tilde w}}{{\tilde w}^2}}  \hfill \\
   \le \mathop {\underline {\lim } }\limits_{\varepsilon  \to 0} \int_{{\mathbb{R}^3}} {|\nabla {{\tilde w}_\varepsilon }{|^2}}  + \int_{{\mathbb{R}^3}} {V(\varepsilon x + {x_\varepsilon })\tilde w_\varepsilon ^2}  + \int_{{\mathbb{R}^3}} {{\phi _{{{\tilde w}_\varepsilon }}}\tilde w_\varepsilon ^2}  \hfill \\
   = \mathop {\underline {\lim } }\limits_{\varepsilon  \to 0} \lambda \int_{{\mathbb{R}^3}} {{{(\tilde w_\varepsilon ^ + )}^p}}  + \int_{{\mathbb{R}^3}} {{{(\tilde w_\varepsilon ^ + )}^6}}  \hfill \\
   = \lambda \int_{{\mathbb{R}^3}} {{{({{\tilde w}^ + })}^p}}  + \int_{{\mathbb{R}^3}} {{{({{\tilde w}^ + })}^6}}  \hfill \\
   = \int_{{\mathbb{R}^3}} {|\nabla \tilde w{|^2}}  + \int_{{\mathbb{R}^3}} {V({x_0}){{\tilde w}^2}}  + \int_{{\mathbb{R}^3}} {{\phi _{\tilde w}}{{\tilde w}^2}} , \hfill \\
\end{gathered} \]
hence as $\varepsilon  \to 0$,
\begin{equation}\label{4.40}
\int_{{\mathbb{R}^3}} {V(\varepsilon x + {x_\varepsilon })\tilde w_\varepsilon ^2}  \to \int_{{\mathbb{R}^3}} {V({x_0}){{\tilde w}^2}}
\end{equation}
and
\begin{equation}\label{4.41}
\int_{{\mathbb{R}^3}} {|\nabla {{\tilde w}_\varepsilon }{|^2}}  \to \int_{{\mathbb{R}^3}} {|\nabla \tilde w{|^2}} .
\end{equation}
In view of \eqref{4.6}, \eqref{4.38} and the fact that ${\left\| {{u_{\varepsilon ,2}}} \right\|_{{H_\varepsilon }}} \to 0$ as $\varepsilon  \to 0$, taking ${d_0} > 0$ small, we can check that $\tilde w \ne 0$. By \eqref{4.27}, we have
\begin{equation}\label{4.42}
{I_{V({x_0})}}(\tilde w) \ge {c_{V({x_0})}}.
\end{equation}
Since ${x_0} \in {\mathcal{M}^\beta } \subset \Lambda $, \eqref{4.39} and \eqref{4.42} imply that $V({x_0}) = {V_0}$ and ${x_0} \in \mathcal{M}$. At this point, it is clear that $\exists U \in {S_{{V_0}}}$ and ${z_0} \in {\mathbb{R}^3}$ such that $\tilde w(x) = U(x - {z_0})$. Since
\[
\int_{{\mathbb{R}^3}} {V({x_0})\tilde w_\varepsilon ^2}  \le \int_{{\mathbb{R}^3}} {V(\varepsilon x + {x_\varepsilon })\tilde w_\varepsilon ^2} ,
\]
by \eqref{4.40} and \eqref{4.41}, we have
\[
{{\tilde w}_\varepsilon } \to \tilde w{\text{ in }}{H^1}({\mathbb{R}^3}),
\]
which implies that
\[
{\left\| {{u_\varepsilon } - \varphi (\varepsilon x - ({x_\varepsilon } + \varepsilon {z_0}))U\Bigl( {x - \bigl( {\frac{{{x_\varepsilon }}}
{\varepsilon } + {z_0}} \bigr)} \Bigr)} \right\|_{{H_\varepsilon }}} \to 0{\text{ as }}\varepsilon  \to 0.
\]
And we recall that ${x_\varepsilon } \to {x_0} \in \mathcal{M}$ as $\varepsilon  \to 0$, this completes the proof.
\end{proof}

\begin{lemma}\label{4.4.}
Let ${d_0}$ be the number given in Lemma~\ref{4.3.}, then for any $d \in (0,{d_0})$, there exist ${\varepsilon _d} > 0$, ${\rho _d} > 0$ and ${\omega _d} > 0$ such that
\[
{\left\| {{{J'}_\varepsilon }(u)} \right\|_{ * ,\varepsilon ,R}} \ge {\omega _d} > 0
\]
for all $u \in J_\varepsilon ^{{c_{{V_0}}} + {\rho _d}} \cap (X_\varepsilon ^{{d_0}}\backslash X_\varepsilon ^d) \cap H_0^1({B_R}(0))$ with $\varepsilon  \in (0,{\varepsilon _d})$ and $R \ge {R_0}/\varepsilon $.
\end{lemma}
\begin{proof}
If the lemma does not hold, there exist $d \in (0,{d_0})$, $\{ {\varepsilon _i}\} $, $\{ {\rho _i}\} $ with ${\varepsilon _i},{\text{ }}{\rho _i} \to 0$, ${R_{{\varepsilon _i}}} \ge {R_0}/{\varepsilon _i}$ and ${u_i} \in J_{{\varepsilon _i}}^{{c_{{V_0}}} + {\rho _i}} \cap (X_{{\varepsilon _i}}^{{d_0}}\backslash X_{{\varepsilon _i}}^d) \cap H_0^1({B_{{R_{{\varepsilon _i}}}}}(0))$ such that
\[
{\left\| {{{J'}_{{\varepsilon _i}}}({u_i})} \right\|_{ * ,{\varepsilon _i},{R_{{\varepsilon _i}}}}} \to 0{\text{ as }}i \to \infty .
\]
By Lemma~\ref{4.3.}(i), we can find $\{ {y_i}\} _{i = 1}^\infty  \subset {\mathbb{R}^3}$, ${x_0} \in \mathcal{M}$, $U \in {S_{{V_0}}}$ such that
\[
\mathop {\lim }\limits_{i \to \infty } |{\varepsilon _i}{y_i} - {x_0}| = 0{\text{ and }}\mathop {\lim }\limits_{i \to \infty } {\left\| {{u_i} - \varphi ({\varepsilon _i}x - {\varepsilon _i}{y_i})U(x - {y_i})} \right\|_{{H_{{\varepsilon _i}}}}} = 0,
\]
which implies that ${u_i} \in X_{{\varepsilon _i}}^d$ for sufficiently large $i$. This contradicts that ${u_i} \notin X_{{\varepsilon _i}}^d$.
\end{proof}

\begin{lemma}\label{4.5.}
There exists ${T_0} > 0$ with the following property: for any $\delta  > 0$ small, there exist ${\alpha _\delta } > 0$ and ${\varepsilon _\delta } > 0$ such that if ${J_\varepsilon }({\gamma _\varepsilon }(s)) \ge {c_{{V_0}}} - {\alpha _\delta }$ and $\varepsilon  \in (0,{\varepsilon _\delta })$, then ${\gamma _\varepsilon }(s) \in X_\varepsilon ^{{T_0}\delta }$, where ${\gamma _\varepsilon }(s): = {W_{\varepsilon ,s{t_0}}}$, $s \in [0,1]$.
\end{lemma}
\begin{proof}
First, we may find a ${T_0} > 0$ such that for any $u \in {H^1}({\mathbb{R}^3})$,
\begin{equation}\label{4.43}
{\left\| {\varphi (\varepsilon x)u(x)} \right\|_{{H_\varepsilon }}} \le {T_0}{\left\| {u(x)} \right\|_{{H^1}({\mathbb{R}^3})}}.
\end{equation}
Define
\[
{\alpha _\delta } = \frac{1}
{4}\min \left\{ {{c_{{V_0}}} - {I_{{V_0}}}({s^2}t_0^2{U^ * }(s{t_0}x)):s \in [0,1],{{\left\| {{s^2}t_0^2{U^ * }(s{t_0}x) - {U^ * }(x)} \right\|}_{{H^1}({\mathbb{R}^3})}} \ge \delta } \right\} > 0,
\]
we have
\begin{equation}\label{4.44}
{I_{{V_0}}}({s^2}t_0^2{U^ * }(s{t_0}x)) \ge {c_{{V_0}}} - 2{\alpha _\delta }{\text{ implies }}{\left\| {{s^2}t_0^2{U^ * }(s{t_0}x) - {U^ * }(x)} \right\|_{{H^1}({\mathbb{R}^3})}} \le \delta .
\end{equation}
Similar as in the proof of \eqref{4.2}, we have
\begin{equation}\label{4.45}
\mathop {\max }\limits_{0 \le s \le 1} |{J_\varepsilon }({\gamma _\varepsilon }(s)) - {I_{{V_0}}}({s^2}t_0^2{U^ * }(s{t_0}x))| \le {\alpha _\delta }
\end{equation}
for all $\varepsilon  \in (0,{\varepsilon _\delta })$. Thus if $\varepsilon  \in (0,{\varepsilon _\delta })$ and ${J_\varepsilon }({\gamma _\varepsilon }(s)) \ge {c_{{V_0}}} - {\alpha _\delta }$, by \eqref{4.44} and \eqref{4.45}, we have ${\left\| {{s^2}t_0^2{U^ * }(s{t_0}x) - {U^ * }(x)} \right\|_{{H^1}({\mathbb{R}^3})}} \le \delta $, then by \eqref{4.43}, we have
\[
\begin{gathered}
  {\text{  }}{\left\| {{W_{\varepsilon ,s{t_0}}}(x) - \varphi (\varepsilon x){U^ * }(x)} \right\|_{{H_\varepsilon }}} \hfill \\
   = {\left\| {\varphi (\varepsilon x){s^2}t_0^2{U^ * }(s{t_0}x) - \varphi (\varepsilon x){U^ * }(x)} \right\|_{{H_\varepsilon }}} \hfill \\
   \le {T_0}{\left\| {{s^2}t_0^2{U^ * }(s{t_0}x) - {U^ * }(x)} \right\|_{{H^1}({\mathbb{R}^3})}} \hfill \\
   \le {T_0}\delta . \hfill \\
\end{gathered}
\]
Recall that $0 \in \mathcal{M}$, we have ${\gamma _\varepsilon }(s): = {W_{\varepsilon ,s{t_0}}} \in X_\varepsilon ^{{T_0}\delta }$.
\end{proof}

For each $R > {R_0}/\varepsilon $, we have
\[
{\gamma _\varepsilon }(s): = {W_{\varepsilon ,s{t_0}}} \in H_0^1({B_R}(0)){\text{ for each }}s \in [0,1],{\text{ }}{X_\varepsilon } \subset H_0^1({B_R}(0)).
\]
Define
\[
{c_{\varepsilon ,R}}: = \mathop {\inf }\limits_{\gamma  \in {\Gamma _{\varepsilon ,R}}} \mathop {\max }\limits_{0 \le t \le 1} {J_\varepsilon }(\gamma (t)),
\]
where
\[
{\Gamma _{\varepsilon ,R}}: = \bigl\{ {\gamma  \in C([0,1],{\text{ }}H_0^1({B_R}(0))):\gamma (0) = 0,{\text{ }}\gamma (1) = {\gamma _\varepsilon }(1) = {W_{\varepsilon ,{t_0}}}} \bigr\}.
\]
Remark that ${\gamma _\varepsilon }(s): = {W_{\varepsilon ,s{t_0}}} \in {\Gamma _{\varepsilon ,R}}$, ${c_\varepsilon } \le {c_{\varepsilon ,R}} \le {{\tilde c}_\varepsilon }$ and $J_\varepsilon ^{{{\tilde c}_\varepsilon }} \cap {X_\varepsilon } \cap H_0^1({B_R}(0)) \ne \emptyset $.

Choosing ${\delta _1} > 0$ such that ${T_0}{\delta _1} < {d_0}/4$ in Lemma~\ref{4.5.} and fixing $d = {d_0}/4: = {d_1}$ in Lemma~\ref{4.4.}. The next Lemma comes from \cite{fis}, for reader's convenience, we give a detailed proof.

\begin{lemma}\label{4.6.}
$\exists \bar \varepsilon  > 0$ such that for each $\varepsilon  \in (0,\bar \varepsilon ]$ and $R > {R_0}/\varepsilon $, there exists a sequence $\{ v_{n,\varepsilon }^R\} _{n = 1}^\infty  \subset J_\varepsilon ^{{{\tilde c}_\varepsilon } + \varepsilon } \cap X_\varepsilon ^{{d_0}} \cap H_0^1({B_R}(0))$ such that ${{J'}_\varepsilon }(v_{n,\varepsilon }^R) \to 0$ in ${(H_0^1({B_R}(0)))^{ - 1}}$ as $n \to \infty $.
\end{lemma}
\begin{proof}
Since ${J_\varepsilon }({\gamma _\varepsilon }(1)) \to {I_{{V_0}}}(U_{{t_0}}^ * ) <  - 3$ as $\varepsilon  \to 0$, we choose $0 < \bar \varepsilon  \le \min \{ {\varepsilon _{{d_1}}},{\varepsilon _{{\delta _1}}}\} $ such that for each $\varepsilon  \in (0,\bar \varepsilon ]$,
\begin{equation}\label{4.46}
{{\tilde c}_\varepsilon } + \varepsilon  \le {c_{{V_0}}} + {\rho _{{d_1}}},{\text{ }}{{\tilde c}_\varepsilon } - {c_\varepsilon } < \frac{1}
{8}{\omega _{{d_1}}}{d_0},{\text{ }}{c_{{V_0}}} - \frac{1}
{2}{\alpha _{{\delta _1}}} < {c_\varepsilon },{\text{ }}{J_\varepsilon }({\gamma _\varepsilon }(1)) < 0.
\end{equation}
Assuming the contrary that for some ${\varepsilon ^ * } \in (0,\bar \varepsilon ]$ and ${R^ * } > {R_0}/{\varepsilon ^ * }$, there exists a $\gamma ({\varepsilon ^ * },{R^ * }) > 0$ such that
\begin{equation}\label{4.47}
{\left\| {{{J'}_{{\varepsilon ^*}}}(u)} \right\|_{ * ,{\varepsilon ^*},{R^*}}} \ge \gamma ({\varepsilon ^*},{R^*}) > 0
\end{equation}
for all $u \in J_{{\varepsilon ^ * }}^{{{\tilde c}_{{\varepsilon ^ * }}} + {\varepsilon ^ * }} \cap X_{{\varepsilon ^ * }}^{{d_0}} \cap H_0^1({B_{{R^ * }}}(0))$.

Let $Y$ be a pseudo-gradient vector field for ${{{J'}_{{\varepsilon ^ * }}}}$ in $H_0^1({B_{{R^ * }}}(0))$, i.e. $Y:J_{{\varepsilon ^ * }}^{{{\tilde c}_{{\varepsilon ^ * }}} + {\varepsilon ^ * }} \cap X_{{\varepsilon ^ * }}^{{d_0}} \cap H_0^1({B_{{R^ * }}}(0)) \to H_0^1({B_{{R^ * }}}(0))$ is a locally Lipschitz continuous vector field such that for every $u \in J_{{\varepsilon ^ * }}^{{{\tilde c}_{{\varepsilon ^ * }}} + {\varepsilon ^ * }} \cap X_{{\varepsilon ^ * }}^{{d_0}} \cap H_0^1({B_{{R^ * }}}(0))$,
\begin{equation}\label{4.48}
{\left\| {Y(u)} \right\|_{{H_{{\varepsilon ^ * }}}}} \le 2{\left\| {{{J'}_{{\varepsilon ^*}}}(u)} \right\|_{ * ,{\varepsilon ^*},{R^*}}},
\end{equation}
\begin{equation}\label{4.49}
\left\langle {{{J'}_{{\varepsilon ^*}}}(u),Y(u)} \right\rangle  \ge \left\| {{{J'}_{{\varepsilon ^*}}}(u)} \right\|_{ * ,{\varepsilon ^*},{R^*}}^2.
\end{equation}
Let ${\psi _1}$, ${\psi _2}$ be locally Lipschitz continuous functions in $H_0^1({B_{{R^ * }}}(0))$ such that $0 \le {\psi _1},{\psi _2} \le 1$ and
\[
{\psi _1}(u) = \left\{ \begin{gathered}
  1{\text{  if }}{c_{{V_0}}} - {\alpha _{{\delta _1}}} \le {J_{{\varepsilon ^*}}}(u) \le {{\tilde c}_{{\varepsilon ^ * }}}, \hfill \\
  0{\text{  if }}{J_{{\varepsilon ^*}}}(u) \le {c_{{V_0}}} - 2{\alpha _{{\delta _1}}}{\text{ or }}{{\tilde c}_{{\varepsilon ^ * }}} + {\varepsilon ^ * } \le {J_{{\varepsilon ^*}}}(u), \hfill \\
\end{gathered}  \right.
\]
\[
{\psi _2}(u) = \left\{ \begin{gathered}
  1{\text{  if }}{\left\| {u - {X_{{\varepsilon ^ * }}}} \right\|_{{H_{{\varepsilon ^ * }}}}} \le \frac{3}
{4}{d_0}, \hfill \\
  0{\text{  if }}{\left\| {u - {X_{{\varepsilon ^ * }}}} \right\|_{{H_{{\varepsilon ^ * }}}}} \ge {d_0}. \hfill \\
\end{gathered}  \right.
\]
Consider the following ordinary differential equations:
\[\left\{ \begin{gathered}
  \frac{d}
{{ds}}\eta (s,u) =  - \frac{{Y(\eta (s,u))}}
{{{{\left\| {Y(\eta (s,u))} \right\|}_{{H_{{\varepsilon ^ * }}}}}}}{\psi _1}(\eta (s,u)){\psi _2}(\eta (s,u)), \hfill \\
  \eta (0,u) = u. \hfill \\
\end{gathered}  \right.\]
By \eqref{4.48} and \eqref{4.49}, we have
\[\begin{gathered}
  {\text{   }}\frac{d}
{{ds}}{J_{{\varepsilon ^*}}}(\eta (s,u)) \hfill \\
   = \Bigl\langle {{{J'}_{{\varepsilon ^*}}}(\eta (s,u)),\frac{d}
{{ds}}\eta (s,u)} \Bigr\rangle  \hfill \\
   = \Bigl\langle {{{J'}_{{\varepsilon ^*}}}(\eta (s,u)), - \frac{{Y(\eta (s,u))}}
{{{{\left\| {Y(\eta (s,u))} \right\|}_{{H_{{\varepsilon ^ * }}}}}}}{\psi _1}(\eta (s,u)){\psi _2}(\eta (s,u))} \Bigr\rangle  \hfill \\
   \le  - \frac{{{\psi _1}(\eta (s,u)){\psi _2}(\eta (s,u))}}
{{{{\left\| {Y(\eta (s,u))} \right\|}_{{H_{{\varepsilon ^ * }}}}}}}\left\| {{{J'}_{{\varepsilon ^*}}}(\eta (s,u))} \right\|_{ * ,{\varepsilon ^*},{R^*}}^2 \hfill \\
   \le  - \frac{1}
{2}{\psi _1}(\eta (s,u)){\psi _2}(\eta (s,u)){\left\| {{{J'}_{{\varepsilon ^*}}}(\eta (s,u))} \right\|_{ * ,{\varepsilon ^*},{R^*}}} \hfill \\
\end{gathered} \]
and combining with \eqref{4.46}, \eqref{4.47} and Lemma~\ref{4.4.}, it is standard to show that $\eta  \in C([0,\infty ) \times H_0^1({B_{{R^ * }}}(0)),H_0^1({B_{{R^ * }}}(0)))$ and satisfies\\
(i) $\frac{d}
{{ds}}{J_{{\varepsilon ^*}}}(\eta (s,u)) \le 0$ for each $s \in [0,\infty )$ and $u \in H_0^1({B_{{R^ * }}}(0))$;\\
(ii) $\frac{d}
{{ds}}{J_{{\varepsilon ^*}}}(\eta (s,u)) \le  - {\omega _{{d_1}}}/2$ if $\eta (s,u) \in \overline {J_{{\varepsilon ^*}}^{{{\tilde c}_{{\varepsilon ^ * }}}}\backslash J_{{\varepsilon ^*}}^{{c_{{V_0}}} - {\alpha _{{\delta _1}}}}}  \cap \overline {X_{{\varepsilon ^*}}^{3{d_0}/4}\backslash X_{{\varepsilon ^*}}^{{d_0}/4}} $;\\
(iii) $\frac{d}
{{ds}}{J_{{\varepsilon ^*}}}(\eta (s,u)) \le  - \gamma ({\varepsilon ^*},{R^*})/2$ if $\eta (s,u) \in \overline {J_{{\varepsilon ^*}}^{{{\tilde c}_{{\varepsilon ^ * }}}}\backslash J_{{\varepsilon ^*}}^{{c_{{V_0}}} - {\alpha _{{\delta _1}}}}}  \cap X_{{\varepsilon ^*}}^{3{d_0}/4}$;\\
(iv) $\eta (s,u) = u$ if ${J_{{\varepsilon ^*}}}(u) \le 0$.

Set ${s_1}: = {\omega _{{d_1}}}{d_0}{(\gamma ({\varepsilon ^*},{R^*}))^{ - 1}}$ and $\xi (t): = \eta ({s_1},{\gamma _{{\varepsilon ^*}}}(t))$, by \eqref{4.46} and (iv), we have $\xi (t) \in {\Gamma _{{\varepsilon ^*},{R^*}}}$. In view of \eqref{4.46} and (i), we may find a ${t_1} \in (0,1)$ such that
\begin{equation}\label{4.50}
{c_{{V_0}}} - {\alpha _{{\delta _1}}}/2 \le {c_{{\varepsilon ^ * }}} \le {c_{{\varepsilon ^*},{R^*}}} \le {J_{{\varepsilon ^*}}}(\xi ({t_1})) \le {J_{{\varepsilon ^*}}}({\gamma _{{\varepsilon ^*}}}({t_1})) \le {{\tilde c}_{{\varepsilon ^*}}}.
\end{equation}
Hence, Lemma~\ref{4.5.} yields
\[
{\gamma _{{\varepsilon ^*}}}({t_1}) \in X_{{\varepsilon ^*}}^{{d_0}/4} \cap \overline {J_{{\varepsilon ^*}}^{{{\tilde c}_{{\varepsilon ^ * }}}}\backslash J_{{\varepsilon ^*}}^{{c_{{V_0}}} - {\alpha _{{\delta _1}}}}} .
\]
Now, we have two cases:\\
Case 1: $\eta (s,{\gamma _{{\varepsilon ^*}}}({t_1})) \notin X_{{\varepsilon ^*}}^{3{d_0}/4}$ for some $s \in [0,{s_1}]$;\\
Case 2: $\eta (s,{\gamma _{{\varepsilon ^*}}}({t_1})) \in X_{{\varepsilon ^*}}^{3{d_0}/4}$ for all $s \in [0,{s_1}]$.

In Case 1, denote
\[
{s_2}: = \inf \{ s \in [0,{s_1}]|\eta (s,{\gamma _{{\varepsilon ^*}}}({t_1})) \notin X_{{\varepsilon ^*}}^{3{d_0}/4}\}
\]
and
\[
{s_3}: = \sup \{ s \in [0,{s_2}]|\eta (s,{\gamma _{{\varepsilon ^*}}}({t_1})) \in X_{{\varepsilon ^*}}^{{d_0}/4}\} ,
\]
then
\[
{s_2} - {s_3} \ge \frac{1}
{2}{d_0},{\text{ }}\eta (s,{\gamma _{{\varepsilon ^*}}}({t_1})) \in \overline {X_{{\varepsilon ^*}}^{3{d_0}/4}\backslash X_{{\varepsilon ^*}}^{{d_0}/4}} {\text{ for every }}s \in [{s_3},{s_2}].
\]
By (i) and \eqref{4.50}, for all $s \in [0,{s_1}]$,
\[\begin{gathered}
  {c_{{V_0}}} - \frac{1}
{2}{\alpha _{{\delta _1}}} \le {J_{{\varepsilon ^*}}}(\eta ({s_1},{\gamma _{{\varepsilon ^*}}}({t_1}))) \le {J_{{\varepsilon ^*}}}(\eta (s,{\gamma _{{\varepsilon ^*}}}({t_1}))) \hfill \\
   \le {J_{{\varepsilon ^*}}}(\eta (0,{\gamma _{{\varepsilon ^*}}}({t_1}))) = {J_{{\varepsilon ^*}}}({\gamma _{{\varepsilon ^*}}}({t_1})) \le {{\tilde c}_{{\varepsilon ^*}}}, \hfill \\
\end{gathered} \]
then by \eqref{4.46} and (ii), we obtain
\[\begin{gathered}
  {J_{{\varepsilon ^*}}}(\xi ({t_1})) = {J_{{\varepsilon ^*}}}({\gamma _{{\varepsilon ^*}}}({t_1})) + \int_0^{{s_1}} {\frac{d}
{{ds}}{J_{{\varepsilon ^*}}}(\eta (s,{\gamma _{{\varepsilon ^*}}}({t_1})))} ds \hfill \\
  {\text{ }} \le {{\tilde c}_{{\varepsilon ^*}}} + \int_{{s_3}}^{{s_2}} {\frac{d}
{{ds}}{J_{{\varepsilon ^*}}}(\eta (s,{\gamma _{{\varepsilon ^*}}}({t_1})))} ds \hfill \\
  {\text{ }} \le {{\tilde c}_{{\varepsilon ^*}}} - \frac{1}
{4}{\omega _{{d_1}}}{d_0} < {c_{{\varepsilon ^*}}}, \hfill \\
\end{gathered} \]
which contradicts to \eqref{4.50}.

In Case 2, by \eqref{4.46}, (iii) and the definition of $s_1$, we have
\[
{J_{{\varepsilon ^*}}}(\xi ({t_1})) \le {{\tilde c}_{{\varepsilon ^*}}} - \frac{1}
{2}\gamma ({\varepsilon ^*},{R^ * }){s_1} = {{\tilde c}_{{\varepsilon ^*}}} - \frac{1}
{2}{\omega _{{d_1}}}{d_0} < {c_{{\varepsilon ^*}}},
\]
which contradicts to \eqref{4.50}. The lemma is proved.
\end{proof}

\begin{proof}[ Proof of  Theorem \ref{1.1.}]
\textbf{Step 1}: By Lemma~\ref{4.6.}, $\exists \bar \varepsilon  > 0$ such that for each $\varepsilon  \in (0,\bar \varepsilon ]$ and $R > {R_0}/\varepsilon $, there exists a sequence $\{ v_{n,\varepsilon }^R\} _{n = 1}^\infty  \subset J_\varepsilon ^{{{\tilde c}_\varepsilon } + \varepsilon } \cap X_\varepsilon ^{{d_0}} \cap H_0^1({B_R}(0))$ such that ${{J'}_\varepsilon }(v_{n,\varepsilon }^R) \to 0$ in ${(H_0^1({B_R}(0)))^{ - 1}}$ as $n \to \infty $.

Since $\{ v_{n,\varepsilon }^R\} $ is bounded in $H_0^1({B_R}(0))$, up to a subsequence, as $n \to \infty $, we have
\begin{equation}\label{4.51}
\left\{ \begin{gathered}
  v_{n,\varepsilon }^R \rightharpoonup v_\varepsilon ^R{\text{ in }}H_0^1({B_R}(0)), \hfill \\
  v_{n,\varepsilon }^R \to v_\varepsilon ^R{\text{ in }}{L^s}({B_R}(0)),{\text{ }}s \in [1,6), \hfill \\
  v_{n,\varepsilon }^R \to v_\varepsilon ^R{\text{ a}}{\text{.e}}{\text{. in }}{B_R}(0). \hfill \\
\end{gathered}  \right.
\end{equation}
By standard argument, we can check that $v_\varepsilon ^R \ge 0$ and satisfies
\begin{equation}\label{4.52}
\left\{ \begin{gathered}
   - \Delta v_\varepsilon ^R + V(\varepsilon x)v_\varepsilon ^R + {\phi _{v_\varepsilon ^R}}v_\varepsilon ^R + 4{\Bigl( {\int_{{\mathbb{R}^3}} {{\chi _\varepsilon }{{(v_\varepsilon ^R)}^2}} dx - 1} \Bigr)_ + }{\chi _\varepsilon }v_\varepsilon ^R = \lambda {(v_\varepsilon ^R)^{p - 1}} + {(v_\varepsilon ^R)^5}{\text{ in }}{B_R}(0), \hfill \\
  v_\varepsilon ^R = 0{\text{ on }}\partial {B_R}(0) \hfill \\
\end{gathered}  \right.
\end{equation}
and we will show that $v_\varepsilon ^R \in J_\varepsilon ^{{{\tilde c}_\varepsilon } + \varepsilon } \cap X_\varepsilon ^{{d_0}}$ for ${d_0} > 0$ small.

Indeed, we write that $v_{n,\varepsilon }^R = u_{n,\varepsilon }^R + w_{n,\varepsilon }^R$ with $u_{n,\varepsilon }^R \in {X_\varepsilon }$ and ${\left\| {w_{n,\varepsilon }^R} \right\|_{{H_\varepsilon }}} \le {d_0}$. Since ${S_{{V_0}}}$ is compact in ${H^1}({\mathbb{R}^3})$, up to a subsequence, we can assume that $u_{n,\varepsilon }^R \to u_\varepsilon ^R $ in $H_0^1({B_R}(0))$ and $w_{n,\varepsilon }^R \rightharpoonup w_\varepsilon ^R$ in $H_0^1({B_R}(0))$ as $n \to \infty $. Then we have $v_\varepsilon ^R = u_\varepsilon ^R + w_\varepsilon ^R$ with $u_\varepsilon ^R \in {X_\varepsilon }$ and ${\left\| {w_\varepsilon ^R} \right\|_{{H_\varepsilon }}} \le {d_0}$ i.e. $v_\varepsilon ^R \in X_\varepsilon ^{{d_0}}$.

By Brezis-Lieb's Lemma (Theorem 1 of \cite{bl}), Lemma~\ref{2.1.}(i), Lemma~\ref{2.2.}(i) and \eqref{4.51}, we have
\[\begin{gathered}
  {{\tilde c}_\varepsilon } + \varepsilon  \ge {J_\varepsilon }(v_{n,\varepsilon }^R) \hfill \\
  {\text{  }} = {J_\varepsilon }(v_\varepsilon ^R) + \frac{1}
{2}\left\| {v_{n,\varepsilon }^R - v_\varepsilon ^R} \right\|_{{H_\varepsilon }}^2 - \frac{1}
{6}\left\| {v_{n,\varepsilon }^R - v_\varepsilon ^R} \right\|_{{L^6}({\mathbb{R}^3})}^6 + o(1) \hfill \\
  {\text{  }} = {J_\varepsilon }(v_\varepsilon ^R) + \frac{1}
{2}\left\| {w_{n,\varepsilon }^R - w_\varepsilon ^R} \right\|_{{H_\varepsilon }}^2 - \frac{1}
{6}\left\| {w_{n,\varepsilon }^R - w_\varepsilon ^R} \right\|_{{L^6}({\mathbb{R}^3})}^6 + o(1) \hfill \\
  {\text{  }} \ge {J_\varepsilon }(v_\varepsilon ^R) + \frac{1}
{2}\left\| {w_{n,\varepsilon }^R - w_\varepsilon ^R} \right\|_{{H_\varepsilon }}^2 - \frac{1}
{6}{S^{ - 3}}\left\| {w_{n,\varepsilon }^R - w_\varepsilon ^R} \right\|_{{H_\varepsilon }}^6 + o(1) \hfill \\
  {\text{  }} = {J_\varepsilon }(v_\varepsilon ^R) + \left\| {w_{n,\varepsilon }^R - w_\varepsilon ^R} \right\|_{{H_\varepsilon }}^2\Bigl( {\frac{1}
{2} - \frac{1}
{6}{S^{ - 3}}\left\| {w_{n,\varepsilon }^R - w_\varepsilon ^R} \right\|_{{H_\varepsilon }}^4} \Bigr) + o(1) \hfill \\
  {\text{ }} \ge {J_\varepsilon }(v_\varepsilon ^R) + o(1){\text{ for }}{d_0} > 0{\text{ small}}{\text{.}} \hfill \\
\end{gathered} \]
Letting $n \to \infty $, we have ${J_\varepsilon }(v_\varepsilon ^R) \le {{\tilde c}_\varepsilon } + \varepsilon $, that is $v_\varepsilon ^R \in J_\varepsilon ^{{{\tilde c}_\varepsilon } + \varepsilon }$.

\noindent \textbf{Step 2}: We claim that  $\exists \bar \varepsilon  > 0$ such that for any $\varepsilon  \in (0,\bar \varepsilon ]$ and $R > {R_0}/\varepsilon $,
\begin{equation}\label{4.53}
{\left\| {v_\varepsilon ^R} \right\|_{{L^\infty }({\mathbb{R}^3})}} \le C.
\end{equation}
Otherwise, $\exists {\varepsilon _j} \to 0$, ${R_j} > {R_0}/{\varepsilon _j}$ such that ${\bigl\| {v_{{\varepsilon _j}}^{{R_j}}} \bigr\|_{{L^\infty }({\mathbb{R}^3})}} \to \infty $ as $j \to \infty $. By Lemma~\ref{4.3.}(i), there exist, up to a subsequence, $\{ {y_j}\} _{i = j}^\infty  \subset {\mathbb{R}^3}$, ${x_0} \in \mathcal{M}$, $U \in {S_{{V_0}}}$ such that
\[
\mathop {\lim }\limits_{j \to \infty } |{\varepsilon _j}{y_j} - {x_0}| = 0{\text{ and }}\mathop {\lim }\limits_{j \to \infty } {\left\| {v_{{\varepsilon _j}}^{{R_j}}(x) - \varphi ({\varepsilon _j}x - {\varepsilon _j}{y_j})U(x - {y_j})} \right\|_{{H_{{\varepsilon _j}}}}} = 0,
\]
then
\[
\mathop {\lim }\limits_{j \to \infty } {\left\| {v_{{\varepsilon _j}}^{{R_j}}(x + {y_j}) - \varphi ({\varepsilon _j}x)U(x)} \right\|_{{L^6}({\mathbb{R}^3})}} = 0,
\]
which implies that as $j \to \infty $,
\[
v_{{\varepsilon _j}}^{{R_j}}(x + {y_j}) \to U(x){\text{ in }}{L^6}({\mathbb{R}^3}).
\]
Using the Brezis-Kato type argument (see also Lemma~\ref{2.4.}), we have
\[
{\bigl\| {v_{{\varepsilon _j}}^{{R_j}}(x + {y_j})} \bigr\|_{{L^\infty }({\mathbb{R}^3})}} \le C,
\]
which leads to a contradiction.

\noindent \textbf{Step 3}: Next, we claim that $v_\varepsilon ^R \to {v_\varepsilon } \in {H_\varepsilon } \cap X_\varepsilon ^{{d_0}} \cap J_\varepsilon ^{{{\tilde c}_\varepsilon } + \varepsilon }$ as $R \to \infty $ in ${H_\varepsilon }$ sense for $\varepsilon  > 0$ small but fixed.

Since ${Q_\varepsilon }(v_\varepsilon ^R)$ is uniformly bounded for all $\varepsilon  > 0$ small and $R > {R_0}/\varepsilon $, we have
\begin{equation}\label{4.54}
\int_{{\mathbb{R}^3}\backslash (\Lambda /\varepsilon )} {{{(v_\varepsilon ^R)}^2}}  \le C\varepsilon .
\end{equation}
By \eqref{4.52}, we have that for any $\delta  > 0$,
\[
 - \Delta v_\varepsilon ^R + V(\varepsilon x)v_\varepsilon ^R \le \delta v_\varepsilon ^R + {C_\delta }{(v_\varepsilon ^R)^5},
\]
taking $\delta  = \mathop {\inf }\limits_{x \in {\mathbb{R}^3}} V(x) > 0$ and combining with \eqref{4.53}, it holds that
\[
 - \Delta v_\varepsilon ^R \le C{(v_\varepsilon ^R)^5} \le C{(v_\varepsilon ^R)^{2/3}},
\]
in the weak sense. Letting $t=6$ in Lemma~\ref{2.7.}, we have
\[
\mathop {\sup }\limits_{{B_1}(y)} v_\varepsilon ^R \le C\Bigl( {{{\left\| {v_\varepsilon ^R} \right\|}_{{L^2}({B_2}(y))}} + \left\| {v_\varepsilon ^R} \right\|_{{L^2}({B_2}(y))}^{2/3}} \Bigr){\text{, }}y \in {\mathbb{R}^3}.
\]
By \eqref{4.54}, we see that
\[
v_\varepsilon ^R(x) \le C({\varepsilon ^{1/2}} + {\varepsilon ^{1/3}}){\text{ for all }}|x| \ge {R_0}/\varepsilon  + 2{\text{ and }}R > {R_0}/\varepsilon .
\]
Hence, for $\varepsilon  > 0$ small but fixed, we have
\[
\lambda {(v_\varepsilon ^R)^{p - 1}} + {(v_\varepsilon ^R)^5} \le \frac{1}
{2}V(\varepsilon x)v_\varepsilon ^R{\text{ for all }}|x| \ge {R_0}/\varepsilon  + 2{\text{ and }}R > {R_0}/\varepsilon .
\]
By the Maximum Principle (see also \cite{ly}), we have
\begin{equation}\label{4.55}
0 \le v_\varepsilon ^R(x) \le {C_1}(\varepsilon ){e^{ - {C_2}(\varepsilon )|x|}}{\text{ for all }}|x| \ge {R_0}/\varepsilon  + 2{\text{ and }}R > {R_0}/\varepsilon ,
\end{equation}
where ${C_1}(\varepsilon )$ and ${C_2}(\varepsilon )$ are independent of $R$.

Choosing a cut-off function ${\varphi _A} \in {C^\infty }({\mathbb{R}^3})$ such that $0 \le {\varphi _A} \le 1$, ${\varphi _A} = 0$ for $|x| \le A$, ${\varphi _A} = 1$ for $|x| \ge 2A$ and $|\nabla {\varphi _A}| \le C/A$. It follows from $\left\langle {{{J'}_\varepsilon }(v_\varepsilon ^R),{\varphi _A}v_\varepsilon ^R} \right\rangle  = 0$ and \eqref{4.55} that
\[\begin{gathered}
  {\text{  }}\int_{{\mathbb{R}^3}\backslash {B_{2A}}(0)} {|\nabla v_\varepsilon ^R{|^2} + V(\varepsilon x)|v_\varepsilon ^R{|^2}}  \hfill \\
   \le \frac{C}
{A}\int_{{\mathbb{R}^3}\backslash {B_A}(0)} {|\nabla v_\varepsilon ^R{|^2}+|v_\varepsilon ^R{|^2}}  + \int_{{\mathbb{R}^3}\backslash {B_A}(0)} {\lambda {{(v_\varepsilon ^R)}^p} + {{(v_\varepsilon ^R)}^6}}  \hfill \\
   \le \frac{C}
{A}\int_{{\mathbb{R}^3}} {|\nabla v_\varepsilon ^R{|^2}+|v_\varepsilon ^R{|^2}}  + C(\varepsilon )\int_{{\mathbb{R}^3}\backslash {B_A}(0)} {{e^{ - C(\varepsilon )|x|}}}  \to 0{\text{ as }}A \to \infty , \hfill \\
\end{gathered} \]
i.e. for $\varepsilon  > 0$ small but fixed,
\begin{equation}\label{4.56}
\mathop {\lim }\limits_{A \to \infty } {\text{ }}\int_{{\mathbb{R}^3}\backslash {B_{2A}}(0)} {|\nabla v_\varepsilon ^R{|^2} + V(\varepsilon x)|v_\varepsilon ^R{|^2}}  = 0.
\end{equation}
Since $\{ v_\varepsilon ^R\} $ is bounded in ${H_\varepsilon }$, we can assume that as $R \to \infty $,
\[\left\{ \begin{gathered}
  v_\varepsilon ^R \rightharpoonup {v_\varepsilon }{\text{ in }}{H_\varepsilon }, \hfill \\
  v_\varepsilon ^R \to {v_\varepsilon }{\text{ in }}L_{{\text{loc}}}^s({\mathbb{R}^3}),{\text{ }}s \in [1,6), \hfill \\
  v_\varepsilon ^R \to {v_\varepsilon }{\text{ a}}{\text{.e}}{\text{.}} \hfill \\
\end{gathered}  \right.\]
By \eqref{4.56} and Sobolev's Imbedding Theorem, we get
\[
v_\varepsilon ^R \to {v_\varepsilon }{\text{ in }}{L^s}({\mathbb{R}^3}),{\text{ }}s \in [2,6){\text{ as }}R \to \infty .
\]
By \eqref{4.53}, we have
\[
v_\varepsilon ^R \to {v_\varepsilon }{\text{ in }}{L^s}({\mathbb{R}^3}),{\text{ }}s \in [2,6]{\text{ as }}R \to \infty .
\]
Using standard argument, we can prove the claim.

Hence, ${v_\varepsilon } \in {H_\varepsilon } \cap X_\varepsilon ^{{d_0}} \cap J_\varepsilon ^{{{\tilde c}_\varepsilon } + \varepsilon }$ is a nontrivial solution of
\[
 - \Delta u + V(\varepsilon x)u + {\phi _u}u + 4{\Bigl( {\int_{{\mathbb{R}^3}} {{\chi _\varepsilon }{u^2}} dx - 1} \Bigr)_ + }{\chi _\varepsilon }u = \lambda {u^{p - 1}} + {u^5}{\text{ in }}{\mathbb{R}^3}.
 \]
Since ${S_{{V_0}}}$ is compact in ${H^1}({\mathbb{R}^3})$, it is easy to see that $0 \notin X_\varepsilon ^{{d_0}}$ for ${d_0} > 0$ small. Thus ${v_\varepsilon } \ne 0$.

\noindent \textbf{Step 4}: For any sequence $\{ {\varepsilon _j}\} $ with ${\varepsilon _j} \to 0$, by Lemma~\ref{4.3.}(ii), there exist, up to a subsequence, $\{ {y_j}\} _{i = j}^\infty  \subset {\mathbb{R}^3}$, ${x_0} \in \mathcal{M}$, $U \in {S_{{V_0}}}$ such that
\begin{equation}\label{4.57}
\mathop {\lim }\limits_{j \to \infty } |{\varepsilon _j}{y_j} - {x_0}| = 0{\text{ and }}\mathop {\lim }\limits_{j \to \infty } {\left\| {{v_{{\varepsilon _j}}}(x) - \varphi ({\varepsilon _j}x - {\varepsilon _j}{y_j})U(x - {y_j})} \right\|_{{H_{{\varepsilon _j}}}}} = 0,
\end{equation}
which implies that as $j \to \infty $,
\[
{w_{{\varepsilon _j}}}(x): = {v_{{\varepsilon _j}}}(x + {y_j}) \to U(x){\text{ in }}{L^6}({\mathbb{R}^3}).
\]
By Lemma~\ref{2.4.} (ii), we get
\begin{equation}\label{4.58}
\mathop {\lim }\limits_{|x| \to \infty } {w_{{\varepsilon _j}}}(x) = 0{\text{ uniformly for all }}{\varepsilon _j}.
\end{equation}
Proceeding as in \cite{ly}, we get
\[
{w_{{\varepsilon _j}}}(x) \le {C_1}{e^{ - {C_2}|x|}},{\text{ }}x \in {\mathbb{R}^3},
\]
where $C_1$ and $C_2$ are independent of ${{\varepsilon _j}}$.

Thus
\[
\varepsilon _j^{ - 1}\int_{{\mathbb{R}^3}\backslash (\Lambda /{\varepsilon _j})} {v_{{\varepsilon _j}}^2(x)}  = \varepsilon _j^{ - 1}\int_{{\mathbb{R}^3}\backslash (\Lambda /{\varepsilon _j} - {y_j})} {w_{{\varepsilon _j}}^2(x)}  \le \varepsilon _j^{ - 1}\int_{{\mathbb{R}^3}\backslash{B_{\beta /{\varepsilon _j}}}(0)} {{{({C_1})}^2}{e^{ - 2{C_2}|x|}}}  \to 0,{\text{ as }}j \to \infty ,
\]
i.e. ${Q_{{\varepsilon _j}}}({v_{{\varepsilon _j}}}) = 0$ for ${{\varepsilon _j}}$ small. Therefore ${v_{{\varepsilon _j}}}$ is a solution of \eqref{4.1}. Set ${u_\varepsilon }(x) = {v_\varepsilon }( {\frac{x}
{\varepsilon }} )$, ${u_{{\varepsilon _j}}}$ is a solution of \eqref{1.1}.

Let ${P_j}$ be a maximum point of ${w_{{\varepsilon _j}}}$, similar to the arguments in Proposition~\ref{3.7.}, we can check that $\exists b > 0$ such that ${w_{{\varepsilon _j}}}({P_j}) > b$, then by \eqref{4.58}, $\{ {P_j}\} $ must be bounded.

Since ${u_{{\varepsilon _j}}}(x) = {w_{{\varepsilon _j}}}( {\frac{x}
{{{\varepsilon _j}}} - {y_j}} )$, ${x_j}: = {\varepsilon _j}{P_j} + {\varepsilon _j}{y_j}$ is a maximum point of ${u_{{\varepsilon _j}}}$. From \eqref{4.57}, ${x_j} \to {x_0} \in \mathcal{M}$ as $j \to \infty $. Since the sequence $\{ {\varepsilon _j}\}$ is arbitrary, we have obtained the existence and concentration results in Theorem~\ref{1.1.}.

To complete the proof, we only need to prove the exponential decay of ${u_\varepsilon }$. Since the proof is standard (see \cite{h,ly} for example), we omit it here.
\end{proof}

\end{document}